\documentclass[pdflatex,sn-mathphys-num]{sn-jnl}

\RequirePackage{amsthm,amsmath,amsfonts,amssymb}
\RequirePackage[numbers]{natbib}

\usepackage{graphicx} 
\usepackage{mathtools,  stmaryrd}
\usepackage{ulem}
\usepackage[english]{babel}
\usepackage{todonotes}
\usepackage[T1]{fontenc}
\usepackage{enumitem}

\usepackage{appendix}

\usepackage{latexsym,amsfonts,amssymb,mathrsfs,float}
\usepackage{amsthm}
\usepackage{amsmath}

\usepackage{mathtools,  stmaryrd}
\usepackage{ulem}
\usepackage[english]{babel}
\usepackage{todonotes}
\usepackage[T1]{fontenc}
\usepackage{enumitem}
\usepackage{geometry}
\newcommand{\red}[1]{{#1}}

\newcommand{\bl}[1]{{#1}}

\newcommand{\gre}[1]{{#1}}

\def\p{{\mathbb P}}
\def\e{{\mathbb E}}
\def\q{{\mathbb Q}}
\def\h{{\mathsf H}}
\def\z{{\mathbb Z}}
\def\r{{\mathbb R}}
\def\N{{\mathbb N}}
\def\d{\, \mathrm{d}}
\def\x{{\mathbf x}}
\def\X{{\mathbf X}}
\def\Y{{\mathbf Y}}

\newcommand{\abs}[1]{\left\lvert #1 \right\rvert}

\allowdisplaybreaks
\renewcommand{\theequation}{\arabic{section}.\arabic{equation}}
\def \cal{\mathcal}
\newtheorem{thm}{Theorem}[section]

\newtheorem{lem}[thm]{Lemma}

\newtheorem{pro}[thm]{Proposition}

\newtheorem{rem}[thm]{Remark}

\numberwithin{equation}{section}
\newenvironment{probis}
  {\addtocounter{thm}{-1}%
   \begin{pro}}
  {\end{pro}}

\begin{document}

\title[Favorite Sites for Simple Random Walk in Two and More Dimensions]{Favorite Sites for Simple Random Walk in Two and More Dimensions}


\author[1]{\fnm{Chenxu} \sur{Hao}}\email{haochenxu@pku.edu.cn}

\author[1]{\fnm{Xinyi} \sur{Li}}\email{xinyili@bicmr.pku.edu.cn}

\author[2]{\fnm{Izumi} \sur{Okada}}\email{\red{iokada@ms.u-tokyo.ac.jp}}

\author*[3]{\fnm{Yushu} \sur{Zheng}}\email{yszheng666@gmail.com}

\affil[1]{
\orgname{Peking University},
\orgaddress{\city{Beijing}, \country{China}}}


\affil[2]{
\orgname{\red{The University of Tokyo}}, \orgaddress{\city{\red{Tokyo}}, \country{Japan}}}

\affil[3]{
\orgname{Academy of Mathematics and Systems Science, Chinese Academy of Sciences}, \orgaddress{\city{Beijing}, \country{China}}}


\abstract{On the trace of a discrete-time simple random walk on $\mathbb{Z}^d$ for $d\geq 2$, we consider the evolution of favorite sites, i.e., sites that achieve the maximal local time at a certain time. For $d=2$, we show that almost surely three favorite sites occur simultaneously infinitely often and eventually there is no simultaneous occurrence of four favorite sites. For $d\geq 3$, we derive sharp asymptotics of the number of favorite sites. This answers an open question of Erd\H{o}s and R\'{e}v\'{e}sz (1984), which was brought up again by Dembo (2005).}




\maketitle

\section{Introduction}\label{Introduction}
Let $(S_n)_{n\geq 0}$ be a discrete-time simple random walk on the integer lattice $\mathbb{Z}^d$ for $d\ge 1$. We write
\begin{eqnarray*}
\xi(x,n)=\#\{0\le j\le n:S_j=x\};\qquad \xi^*(n):=\sup_{y\in{\mathbb{Z}^{d}}}\xi(y,n)
\end{eqnarray*}
for the local time of site $x$ at time $n$ and maximal local time at time $n$, respectively.

There is a considerable amount of interest and a huge literature on the asymptotics of $\xi^*(n)$ (as well as its counterpart for continuous-time random walk) and the distribution of thick points (also called $\alpha$-favorite points, which are sites that achieve at least a proportion of the maximal local time). Erd\H{o}s and Taylor (1960) (see \cite{ET60}) gave sharp bounds on $\xi^*(n)$ for $d\geq 3$ and derived up-to-constants bounds for $d=2$, with upper bound conjectured to be appropriate. Nearly forty years later, the seminal work \cite{DPRZ01} by Dembo, Peres, Rosen and Zeitouni confirmed this conjecture, which states
\begin{align}\label{max local time}
\lim_{n\to\infty}\dfrac{\xi^*(n)}{(\log n)^2}=\dfrac{1}{\pi}~~~~\text{a.s.},
\end{align}
for $d=2$ and also established the growth exponent of thick points. Later on, Rosen gave in \cite{Rosen05} a random walk proof of this conjecture (see also \cite{BR07} for a refinement of this argument). The link between local time of random walk and Gaussian free field (which corresponds to log-correlated when $d=2$) sheds deeper insight into this subject. Another proof of the Erd\H{o}s-Taylor conjecture was given in \cite{Jego20} exploiting this link with GFF.  In \cite{Oka20}, one of the authors of this work derived the growth exponents of pairs of thick points. %
In \cite{Jego23}, Jego identified the scaling limit of the occupation measure of thick points of random walk as the supporting measure of thick points of Brownian motion, partially constructed for the first time in \cite{BBK94} later extended to the full range of parameters in \cite{AHS20,Jego20b,Jego21}. 
A very precise conjecture was also made in \cite{Jego21} on the finer asymptotics of $\xi^*(n)$ (see \eqref{eq:Jegoconj} for the exact statement) and the structure of near-favorite sites of 2D random walk. 
\red{The work \cite{Rosen23} made some partial progress in this direction.} 
We also mention that results of the same flavor have been established on regular trees; see \cite{BL24}.

For any $n\ge 0$, we denote by
\begin{eqnarray*}
\mathcal{K}^{(d)}(n):=\Big\{x\in\mathbb{Z}^d:\xi(x,n)=\xi^*(n)\Big\}
\end{eqnarray*}
the set of favorite sites at time $n$. 
A major direction of research is the cardinality of $\mathcal{K}^{(d)}(n)$. A classical question of Erd\H{o}s and R\'{e}v\'{e}sz \cite{ER84,ER87} reads as follows:
\begin{equation}
    \label{ERquestion}
\mbox{Can $\#\mathcal{K}^{(d)}(n)=r$ occur infinitely often for $r\ge3$ and $d\ge1$?}
\end{equation}
For $d\geq 3$, the answer is positive: Erd\H{o}s and R\'{e}v\'{e}sz themselves showed in \cite{ER91} that \bl{a.s.} $\#\mathcal{K}^{(d)}(n)=r$ i.o.~for any $r\in\mathbb{Z}^+$. When $d=1$, it was proved (first for favorite edges in \cite{TW97} and later for favorite sites in \cite{Ba01}) that \bl{a.s.} $\#\mathcal{K}^{(1)}(n)\leq 3$ eventually\footnote{We say that a sequence of events $(A_n)_{n\geq 1}$ occurs eventually, if $\liminf_{n\to\infty}1_{A_n}=1$.} and it remained open for a long time whether $\#\mathcal{K}^{(1)}(n)=3$ i.o.. Much later, Ding and Shen showed in \cite{DS18} that this is indeed the case. A similar result also exists for favorite edges (see \cite{HHMS22} for more details). We also refer readers to \cite{Oka16} for a survey on favorite sites.
\red{The work \cite{CDH18} examined this classical question in the context of null-recurrent randomly biased walks within a random environment on a supercritical Galton-Watson tree. Additionally, \cite{GNR21} addressed the same question concerning persistent random walk.}


\subsection{\red{Main results}}

In this work, we will give a complete answer to the open question \eqref{ERquestion} of Erd\H{o}s and R\'{e}v\'{e}sz by characterizing the asymptotic behavior of $\#\mathcal{K}^{(d)}(n)$ for $d\geq 2$.

For $d=2$, we show that the number of favorite sites has the same behavior as the $d=1$ case (note that trivially $\liminf_{n \to \infty} {\#\mathcal{K}^{(d)}(n)}=1$).
\begin{thm}\label{mainthm2}
$(d=2)$ Almost surely, 
$$
\limsup_{n \to \infty} {\#\mathcal{K}^{(2)}(n)}=3.
$$
\end{thm}
For $d\geq 3$, we greatly refine the result in \cite{ER91}: we derive the following sharp asymptotics of $\# \mathcal{K}^{(d)}(n)$.
\begin{thm}\label{mainthm-01}
$(d\geq 3)$ Almost surely,
\begin{align*}
\limsup_{n \to \infty} \frac{\#\mathcal{K}^{(d)}(n)}{\log\log n}
=-\frac{1}{\log \gamma_d},
\end{align*}
where $\gamma_d:= \p[S_n \neq S_0,\; \forall n\geq 1]$, which is the probability that a simple random walk never returns to the starting point.
\end{thm}

We also mention that in the proof of Theorem \ref{mainthm2}, through refining the argument of \cite{Rosen05}, the following \red{control on the} lower deviation of $\xi^*(n)$ is derived, which improves upon \eqref{max local time} and is of independent interest.
\begin{pro}\label{lower}
\red{$(d=2)$.} For any $\delta>0$, there exist constants $C=C(\delta)>0$, such that
\begin{equation}\label{eq:lower}
    \p\left(\xi^*(n)<\frac{1}{\pi}(\log n)^2-(\log n)^{8/5+\delta}\right) <Ce^{-\exp((\log n)^{3/5})}.
\end{equation}
\end{pro}
Compared to conjecture in \cite{Jego21}, which reads (under our setup) as
\begin{equation}\label{eq:Jegoconj}
  \sqrt{\xi^*(n)} - \frac{1}{\sqrt{\pi}}\log n+\frac{1}{\sqrt{\pi}}\log\log n  \overset{\rm law}{\longrightarrow} {\rm Gumbel}^*,
\end{equation}
where ${\rm Gumbel}^*$ stands for (plausibly a mixture of) Gumbel distribution.
This bound is rather crude and most likely further improvable with some extra effort
, but sufficient for our use.

\subsection{\red{Sketch of the proof}}\label{sec:sketch}

\red{We first focus on the case $d=2$ and provide a brief outline of the proof strategy.} As we are going to see, there is considerable advantage in measuring the time elapsed by the growth of $\xi^*(\cdot)$,  which we denote by $m$ throughout this work, rather than the natural time $n$. In this regard, for $m,k\geq 1$, let $T_m^{k}=\inf\big\{n> T_m^{k-1}:\#\{x\in\mathbb{Z}^d:\xi(x,n)\geq m\}=k\big\}$ (where $T_m^{0}:=0$) stand for the first time \red{at which the $k$-th site whose local time is at least $m$ is created.} Note that typically $T^1_m\approx \exp(\sqrt{\pi m})$.

The lower bound of $\#\mathcal{K}^{(2)}(\cdot)$ follows from the observation that at $T^1_m$, with high probability the random walk creates another site with local time $m$ within the next \red{$\exp(2\sqrt{m})$} steps. \red{The probability for this site to become the second favorite site, is bounded below by} that of the random walk to avoid the first favorite site \red{for $\exp(2\sqrt{m})$} steps, which is of order $m^{-1/2}$ \red{(see Lemma \ref{stopping time estimate} for details)}. The analysis of the cost of creating a third favorite site is similar. One then applies generalized second Borel-Cantelli lemma to conclude that three favorite sites occur infinitely often. 

The upper bound of $\#\mathcal{K}^{(2)}(\cdot)$ is much more delicate. \red{Heuristically speaking,} to bound the probability of seeing extra favorite sites \red{of local time $m$}, we investigate $\xi(\cdot,T_m^k)$, $k=1,2,3$, the local time profile and produce various controls on the number of near-favorite sites (with local time \red{at least $m-m^{\alpha}$ and distance $\approx\exp(m^{\alpha'})$ from $L_m^k:=S_{T_m^k}$ for some $\alpha,\alpha'>0$) which can serve as candidates for next favorite site.  Once we have good control over the number of such candidate sites with possible combinations of location and the local time gap, we can use classical hitting probability bounds to calculate the cost of producing extra favorite sites.}

A crucial observation that enables us to control the number of near-favorite sites lies in the decomposition of local time profile into the local ($\approx$lazy) and external ($\approx$non-lazy) parts: if we were looking at a lazy random walk, then conditioned on the profile of the non-lazy local time, the lazy local time of a certain site depends only on the non-lazy local time of that site, and this conditional law is explicit. This observation allows us to modify independently the local time of a certain site and calculate precisely the cost of such operations. Similar but weaker properties exist for simple random walk, where we regard numbers of two-step excursions in a certain direction as lazy local time \red{(see \eqref{eq:twostep} and below for precise definitions). See Section \ref{Preliminaries}, in particular Propositions \ref{condilaw1}, \ref{condilaw11}, \ref{condilaw2} and \ref{condilaw21} for precise statements of this observation under different settings. As we are only able to manipulate half of the sites at the same time, we partition $\mathbb{Z}^2$ in pairs (i.e., domino tilings) and work with only one site in each pair instead; see \eqref{eq:Xdef}, \eqref{eq:Ydef} and \eqref{eq:Yprimedef} for precise definitions of pairings and Proposition \ref{pro:finalreduction} for actual usage.}

We then apply three rounds of screening \red{estimates} using different methods. In the first round, we apply Proposition \ref{lower} \red{(which implies Lemma \ref{Tkm} and then in turn allows us to obtain Proposition \ref{Mestimates}; see also Remark \ref{rem:1.3use} for the necessity of this ingredient} to produce a preliminary bound on the \red{``balancedness''} of the lazy local time (see Proposition \ref{Theta} for precise statements). We then use an analogy with the urn model (see Lemma \ref{ty2}) to produce a second screening (see Proposition \ref{bigalpha}).  Finally, through the analogy with another variant of the urn model (see Lemma \ref{ty1}) we produce a third bound (see Proposition \ref{smallalpha}) \red{which is sufficient for us to arrive at the final bound Proposition \ref{pro:finalreduction}.}

We also remark here that in the proof of upper bound we did not pursue a sharp control on the cost of seeing an extra favorite site. Heuristically speaking, Proposition \ref{pro:finalreduction} indicates that when maximal local time is $m$, the cost is bounded by \red{$O(m^{-\kappa})$ with some $\kappa$ slightly above $1/3$ (see \eqref{eq:kappa1}, \eqref{eq:kappa2} and \eqref{eq:Kappa} for relevant definitions)}, which is barely sufficient for the application of Borel-Cantelli but does not match at all the explicit lower bound $\Omega(m^{-1/2})$ from Section \ref{se:2dlb}.

Finally, we turn to Theorem \ref{mainthm-01}. The key observation is that when $d\geq 3$, thick points are formed in an independent fashion. In the proof, we use a quantitative bound (see \eqref{as2} for a precise statement) from \cite{Csaki} on the rareness of pairs of thick points to exclude possible pairs of favorite sites that are formed within a short period (see Lemma \ref{lem0} for details). We then use bounds on hitting probability of random walk and the Borel-Cantelli lemmas to derive precise asymptotics on the number of favorite sites.
\begin{rem}
We now briefly discuss some open questions that spring from this work.
\begin{enumerate}
    \item In a series of work \cite{BG85,LS04,Ba23}, precise asymptotic escape rate of the elements of $\mathcal{K}^{(1)}(n)$ were derived. The transience of $\mathcal{K}^{(2)}(n)$ was established in \cite{DPRZ01} (transience of $\mathcal{K}^{(d)}(n)$ is trivial for $d\ge3$). It is then very natural to wonder how fast favorite sites escape to infinity in two dimensions. 
    
    
    \item The proof of the upper bound in Theorem \ref{mainthm2} exploits the distinctive structure of the graph $\mathbb{Z}^2$. However, we expect the same behavior for the asymptotics of favorite sites for all periodical planar lattices. Can one develop an argument which is more robust and works on all lattices (and more general walks)?

\end{enumerate}
\end{rem}

\subsection{\red{Organization and conventions}}
We now briefly discuss the organization of this work. Section \ref{se:2} is dedicated to preliminaries, including various classical estimates on hitting and escape probabilities of random walk, decomposition of local time profiles and relevant calculations, as well as two versions of the urn model that serve as toy models in the analysis of the local time profile. Sections \ref{se:3} and \ref{se:4} contain the proof of Theorems \ref{mainthm-01} and \ref{mainthm2}. 
Finally, we give the proof of Proposition \ref{lower} in the Appendix, as it is rather technical and relatively independent from other parts of this work.

Finally, let us explain the convention concerning constants and asymptotic relations. We denote by $c,c',C$ and $C'$ positive and finite constants whose values are universal (apart from dependence on $d$) and may change from line to line. 
Extra dependence of these constants upon other variables is marked at the first occurrence. If $(a_n)$ and $(b_n)$ are non-negative sequences, then, we write
$a_n\lesssim b_n$ if there exists $c>0$ such that $a_n \leq c b_n$ for all $n$ and $a_n \asymp b_n$ if $a_n\lesssim b_n$ and $b_n\lesssim a_n$. For a sequence $(a'_n)$, we write $a'_n=O(b_n)$ if $|a'_n| \lesssim b_n$.
For a sequence $(c_n)$, we write $c_n=o(b_n)$ if $\lim_{n\to\infty}\frac{b_n}{c_n}=\infty$.

\section*{Acknowledgement} CH is supported by China Postdoctoral Science Foundation (No.\ GZC20230089). XL is supported by National Key R\&D Program of China (No.\ 2021YFA1002700 and No.\ 2020YFA0712900) and NSFC (No.\ 12071012). IO is supported in part by JSPS KAKENHI Grant-in-Aid  for Early-Career Scientists (No.~JP24K16931). YZ is supported by China Postdoctoral Science Foundation (No.\ 2023M743721 and No.\ 2025T180850).

\section{Preliminaries}\label{se:2}
We start with general notation. For any $l \in \mathbb{R}$, $[l]$ stands for the largest integer less than or equal to $l$. On $\z^2$, we set ${\bf e}_1=(1,0)$, ${\bf e}_2=(0,1)$, ${\bf e}_3=(-1,0)$, and ${\bf e}_4=(0,-1)$. Denote by
\[\z^2_{\rm e}:=\{x=(x_1,x_2)\in \z^2:x_1+x_2\text{ is even}\},\quad\mbox{ and }\quad\z^2_{\rm o}:=\z^2\setminus \z^2_{\rm e}.\]

For $x\in\mathbb{Z}^d$ and $r \in \mathbb{R}^+$, we write
$$
D(x,r)=\{z\in\mathbb{Z}^d:\ |x-z|\leq r\}
$$
for the (discrete) ball of radius $r$ centered at $x$.


\medskip

We now turn to random walk. 
Recall that $(S_n)_{n\geq 0}$ is a discrete-time simple random walk on the integer lattice $\mathbb{Z}^d$, $d\ge 1$, starting at $S_0=a$ and let $\p^a$ and $\e^a$ stand for the respective probability measure and expectation. We write $\p=\p^0$ and $\e=\e^0$ for short. We write  $S_{[i,j]}$ for the sub-path $(S_{i},S_{i+1},...,S_{j-1},S_{j})$. From time to time we also tacitly regard $S_{[i,j]}$ as a subset of $\mathbb{Z}^d$, when no confusion arises. We also write $S_{(i,j)}=S_{[i+1,j-1]}$ for short. 
For $A\subset\mathbb{Z}^d$, $x\in\mathbb{Z}^d$, $n\in\mathbb{N}$, we use $\h$ to denote the first hitting time. More precisely,
$$\h_A:=\inf\{k>0: S_k\in A\};~~~
\h_{x}:=\h_{\{x\}};~~~\h_{x,y}:=\h_{\{x,y\}}.$$
We also write $$\h_A(n)=\inf\{k\geq n: S_k\in A\}-n$$ for the first hitting time after time $n$.
For a subset $A\subset\z^2$ and $x,y\in \z^2$, define the following Green's function $G_A(x,y)$ as
\[G_A(x,y)=\sum_{i=0}^\infty\p^x(S_i=y,i<\mathsf{H}_{A^c}).\]
By Theorem 1.6.6 and Proposition 1.6.7 of \cite{La91}, we have the following estimates for the Green's function,
\begin{align*}
	G_{D(0,R)}(x_0,0)=\begin{cases}
		\frac{2}{\pi}\log R+c_0 +O(R^{-1})&\text{ if }x_0=0,\\
		\frac{2}{\pi}\log(R/\abs{x_0})+O(\abs{x_0}^{-1})&\text{ if }x_0\neq 0,
	\end{cases}
\end{align*}
where $c_0$ is a finite constant.

We now record a few classical estimates on hitting probabilities of 2D random walk.

\begin{lem}[$d=2$]\label{2dlem02} For any $n
\geq 1$,
\begin{align}\label{ET60 2.5}
\gre{\p\big(\h_0 \geq n \big)-\frac{\pi}{\log n}=O\left(\frac{1}{(\log n)^2}\right)};
\end{align}
\begin{align}\label{2d-lem for}
{\inf}_{x\in\mathbb{Z}^2}\p\big(\h_{0,x}\geq n\big)\asymp \dfrac{1}{\log n}.
\end{align}
For any $x\in\mathbb{Z}^2$,
\begin{align}\label{2dlem02 for-2}
\p\big(\h_{x}<\h_{0}\big)\gtrsim \frac{1}{\log |x|}.
\end{align}
For any $r>0$ and $x_0\neq0$,
\begin{align}\label{hitting}
\left\{\begin{aligned}
	\p(\mathsf{H}_{D(0,r)^c}<\mathsf{H}_0)&=\frac{1}{\frac{2}{\pi}\log r+c_0}\left[1+O\left(\frac{1}{r\log r}\right)\right],\\
	\p^{x_0}(\mathsf{H}_0<\mathsf{H}_{D(0,r)^c})&=\frac{\frac{2}{\pi}\log(r/\abs{x_0})}{\frac{2}{\pi}\log r+c_0}\left[1+O\left(\frac{1}{\abs{x_0}\log(r/\abs{x_0})}\right)\right].
\end{aligned}\right.
\end{align}
Moreover, we have that uniformly in $r<\abs{x}<R$,
\begin{align}\label{hitting2}
\left\{\begin{aligned}
	\p^x\left(\mathsf{H}_{\partial D(0,R)}<\mathsf{H}_{\partial D(0,r)}\right)&=\frac{\log(\abs{x}/r)+O(r^{-1})}{\log(R/r)},\\
	\p^x\left(\mathsf{H}_{\partial D(0,r)}<\mathsf{H}_{\partial D(0,R)}\right)&=\frac{\log(R/\abs{x})+O(r^{-1})}{\log(R/r)}.
\end{aligned}\right.	
\end{align}
\end{lem}
See \cite[(2.5)]{ET60} for a proof of \eqref{ET60 2.5}. A quick modification of \cite[Lemma 20.1]{Re13} and \cite[(19.7)]{Re13} yields \eqref{2d-lem for} - \eqref{2dlem02 for-2}. We can derive (\ref{hitting}) from the observations:
\[\p^0(\mathsf{H}_{D(0,r)^c}<\mathsf{H}_0)=\frac{1}{G_{D(0,r)}(0,0)}\text{ and }\p^{x_0}(\mathsf{H}_0<\mathsf{H}_{D(0,r)^c})=\frac{G_{D(0,r)}(x_0,0)}{G_{D(0,r)}(0,0)},\]
and derive (\ref{hitting2}) from \cite[Exercise 1.6.8]{La91}. We omit the proofs.



We will also need the following estimate for $d\geq 3$.
\begin{lem}[$d\ge3$]\label{Green} For any $n,k\geq 1$,
\begin{align}\label{hitk}
\sup_{x_1, \ldots ,x_k \in \mathbb{Z}^d}
\p\Big(\h_{\{x_1, \ldots ,x_k\}}(n)<\infty \Big) \lesssim kn^{-d/2+1}.
\end{align}
\end{lem}
\begin{proof}
By a simple computation, 
\begin{equation*}
\sup_{x_1, \ldots ,x_k \in \mathbb{Z}^d}
\p\Big(\h_{\{x_1, \ldots ,x_k\}}(n)<\infty \Big)
\le  k \sup_{x \in \mathbb{Z}^d}
\p\Big(\h_x(n) <\infty \Big)
\lesssim  k n^{-d/2+1},
\end{equation*}
where the last inequality follows from the classic bound transition kernel of simple random walk (see \cite[Theorem 1.2.1]{La91}).
\end{proof}

We now turn to favorite sites. As discussed in Section \ref{Introduction}, it is much more convenient to make calculations in terms of the growth of maximal local time. To this end, we define the following stopping times. For any $m\in \z^+$, let
\begin{align}\label{tmkdef}
T_m^{0}=0~\mbox{and}~T_m^{k}=\inf\big\{n> T_m^{k-1}:\#\{x\in\mathbb{Z}^d:\xi(x,n)\geq m\}=k\big\}~\mbox{for}~k\ge 1\mbox{ and }L_m^k=S_{T_m^k}
\end{align}
stand for the time and location where the $k$-th site with local time no less than $m$ is created.
 We also define corresponding $\sigma$-algebras:
\begin{equation}\label{sigma}
\mathcal{F}_{m}^k:=\sigma\{S_{[0,T_{m}^k]}\}.
\end{equation}
Let 
\begin{align}\label{Mmk}
M_m^k:={\big\{\exists\, n\ge 0\;\text{s.t.}\;\#\mathcal{K}^{(d)}(n)=k\text{ and }\xi^*(n)=m\big\}}=\{T_{m}^k<T_{m+1}^1\}
\end{align}
be the event that $k$ favorite sites of local time $m$ occur simultaneously. {Since $\xi^*(n)\rightarrow\infty$ a.s., we have for any $k\in \z^+$,
\[\big\{\#\mathcal{K}^{(d)}(n)\ge k\text{ infinitely often in  }n\big\}\overset{\text{a.s.}}{=}\big\{M^k_m\text{ infinitely often in  }m\big\},\]
which the readers should always keep in mind.}
Note that the event $M^k_m$ can also be represented as
\begin{align}\label{representation}
	M_m^k=U^2_m\cap\cdots\cap U^k_m\text{ with }U^j_m:=\big\{S_n\notin\{L^1_m,\ldots,L^{j-1}_m\}\,\forall\,n\in(T^{j-1}_m,T^j_m]\big\}.
\end{align}
From the observation that $M_m^k= \{T_m^k<T_{m+1}^1\}$, it is easy to infer that for all $m,k\geq 1$,
\begin{equation}\label{Mmkmeasure}
    M_m^k \in \mathcal{F}_{m+1}^1.
\end{equation}



As discussed in Section \ref{Introduction}, the  decomposition of local time is a crucial tool in the analysis of the 2D favorite sites. 
{Observe that for each $n\ge 0$, the process $S$ perform a geometric number of excursions $(S_n,S_n+{\bf e}_1,S_n)$ before following another two-step path. Heuristically, these excursions are considered as the ``holding times'' of $S$ and the remaining part of $S$ is considered as the ``jump chain''. In addition, it is more convenient to do this decomposition according to the parity of the time when the excursions start. To this end, we let 
\begin{equation}\label{eq:twostep}
\begin{split}
	\mathcal{L}&:=\{k\ge 2:\text{$k$ is even and } S_{k-2}=S_{k-1}-{\bf e}_1=S_k\},\\
	\mathcal{L}'&:=\{k\ge 2:\text{$k$ is odd and } S_{k-2}=S_{k-1}+{\bf e}_1=S_k\},        
\end{split}
\end{equation}
and write
\begin{align}\label{N_n}
N_n:=n-2\#(\mathcal{L}\cap [2,n])-1_{\{n+1\in \mathcal{L}\}},\quad N'_n:=n-2\#(\mathcal{L}'\cap [2,n])-1_{\{n+1\in \mathcal{L}'\}}.
\end{align}
Denote by $\widetilde{S}$ (resp. $\widetilde{S}'$) the trajectory obtained from $S$ by removing all two-step excursions $(S_{k-2},S_{k-1},S_k)$ with $k\in \mathcal{L}$ (resp. $k\in \mathcal{L}'$). In particular, for any $n\ge 0$, $\widetilde{S}{[0,N_n]}$ (resp. $\widetilde{S}'{[0,N'_n]}$) can be obtained from $S{[0,n]}$ by removing all two-step excursions $(S_{k-2},S_{k-1},S_k)$ with $k\in \mathcal{L}\cap[2,n]$ (resp. $k\in \mathcal{L}'\cap[2,n]$) and the one-step path $(S_{n-1},S_n)$ if $n+1\in\mathcal{L}$.

We decompose the local time $\xi$ of $S$ into the contributions of ``jump chain'' and ``holding times'' respectively. Precisely,
{let 
\[\widetilde{\xi}(x,n):=\#\{j\in[0,n]:\widetilde{S}_j=x\},\qquad \widetilde{\xi}'(x,n):=\#\{j\in [0,n]:\widetilde{S}'_j=x\}\]
be the local times of $\widetilde{S}$ and $\widetilde{S}'$ respectively and for $x\in \z^2_{\rm e}$,
\begin{align*}
	\xi_{\rm L}(x,n)&:=\#\{k\in \mathcal{L}\cap[2,n]:S_{k-2}=x\},\quad \xi_{\rm L}(x+{\bf e}_1,n):=\xi_{\rm L}(x,n)+1_{\{n+1\in\mathcal{L},S_{n-1}=x\}},\\
 \xi'_{\rm L}(x,n)&:=\#\{k\in \mathcal{L}'\cap[2,n]:S_{k-2}=x\},\quad \xi'_{\rm L}(x+{\bf e}_1,n):=\xi'_{\rm L}(x,n)+1_{\{n+1\in\mathcal{L}',S_{n-1}=x\}}.
\end{align*}
Then we have for all $x\in \z^2$,
\begin{align}\label{2d local time}
\xi(x,n)=\widetilde{\xi}(x,N_n)+\xi_{\rm L}(x,n)=\widetilde{\xi}'(x,N'_n)+\xi'_{\rm L}(x,n).
\end{align}
We also call $\widetilde{\xi}(x,N_n)$ and $\xi_{\rm L}(x,n)$ (as well as $\widetilde{\xi}'(x,N'_n)$ and $\xi_{\rm L}'(x,n)$) the ``external'' and ``local'' (or ``lazy'') parts of $\xi$ respectively.  }


We record the following crucial observation on such decomposition of local time: the distribution of lazy local times depends on the external local times in an explicit fashion. Various versions of the conditional law of $\xi_{\rm L}$ given $\widetilde{\xi}$ are presented in Section \ref{Preliminaries}. Since each step of the ``jump chain'' is accompanied by a geometric ``holding time'' with success probability $15/16$, the conditional laws are closely related to the laws of a finite i.i.d. sum of geometric random variables, i.e., negative binomial distribution. Hence we introduce the following notation.
\begin{align}\label{CLT*}
p(i,j):=\frac{(i+j-1)!}{(i-1)!j!} \frac{15^i}{16^{i+j}} =
\p\Big(\sum_{\ell=1}^i\gamma_\ell=j\Big),
\end{align}
where $\{\gamma_\ell:\ell\ge 1\}$ are i.i.d. geometric random variables with success probability $15/16$. We also let
\begin{align*}
    \bar{p}(i,j):=p(i,j-i),
\end{align*}
which is related to the conditional law of $\xi$ given $\widetilde{\xi}$.

We will need the following estimates on $p(i,j)$. Lemma \ref{stirling01} allows us to compare the order of $\bar{p}(i,j_1)$ and $\bar{p}(i,j_2)$ for different $j_1$ and $j_2$, which will be crucial in the proof of Proposition \ref{bigalpha} relating the analysis of local time to urn models. 
\begin{lem}[Local central limit theorem]\label{stirling01}
	Let $\sigma^2=\frac{16}{225}$ be the variance of $\gamma_\ell$. Then there exists $\rho>0$ such that uniformly in all $i,j\ge 1$ with $\abs{i-j}<\rho i$,
	\begin{align*}
		\bar{p}(i,j)=\frac{1}{\sqrt{2\pi}\sigma}\exp\left\{-\frac{\big(j-\frac{16}{15}i\big)^2}{2\sigma^2 i}+O\left(\frac{1}{\sqrt{i}}+\frac{\abs{j-\frac{16}{15}i}^3}{i^2}\right)\right\}.
	\end{align*}
\end{lem}

Lemma \ref{moderate} contains moderate deviation type estimates that will be useful in the proof of Proposition \ref{Theta}. 
\begin{lem}[Moderate deviation]\label{moderate}
	For any sequence $(a_n)$ with $\frac{a_n}{\sqrt{n}}\rightarrow \infty$ and $\frac{a_n}{n}\rightarrow 0$,
	\begin{align}
		\lim_{n\rightarrow \infty} \frac{n}{a_n^2}\log \left(\sum_{j>\frac{1}{15}n+a_n}p(n,j)\right)=-\frac{1}{2\sigma^2},\label{moderate1}\\
		\lim_{n\rightarrow \infty} \frac{n}{a_n^2}\log \left(\sum_{j<\frac{1}{15}n-a_n}p(n,j)\right)=-\frac{1}{2\sigma^2}.\label{moderate2}
	\end{align}
\end{lem}
Lemmas \ref{stirling01} and \ref{moderate} follow from the standard theories. The readers are referred to \cite[Section 2.3.1]{lawler2010random} and \cite[Section 3.7]{dembo2009} for details.

In the proof of Theorem \ref{mainthm2}, we also need bounds on the upper deviations of the local times $\xi$ and $\widetilde{\xi}$ at the origin, and the stopping time $T^k_m$ when $k$ favorite sites of local time $m$ do occur.
\begin{lem}[$d=2$]\label{erdos_taylor}
$\phantom{a}$
	\begin{itemize}
		\item\label{one_point1} For any $\alpha>0$, there exists $c=c(\alpha)>0$ such that for all $n\ge 2$,
		\begin{align}\label{one_point_xi}
			\p\left\{\xi(0,n)\ge \alpha(\log n)^2\right\}\le n^{-\pi \alpha}(\log n)^c.
		\end{align} 
		\item\label{one_point2} Let $K_2=K_2(\beta):=\frac{15}{16\pi}(\log n)^2-2(\log n)^{\beta}$ with $\beta\in(1,2)$. Then for $1<\beta<2$ and all $n$ large ,
\begin{align}\label{side 01}
\p\left(\widetilde{\xi}(0,n)\ge K_2 \right)
\le n^{-1}e^{8(\log n)^{\beta-1}}.
\end{align}
	\end{itemize}
\end{lem}
\begin{proof}
	\eqref{one_point_xi} is \cite[(3.11)]{ET60}. For \eqref{side 01}, we note that $(\widetilde{S}_{2n})_{n\ge 0}$ is a Markov chain. The covariance matrix of the increments equals $\frac{16}{15}I_2$. It follows from the local central limit theorem that
\begin{align*} 
\p\left(\widetilde{S}_{2n}=0 \right)= \frac{15}{16}\frac{1}{\pi n}\left[1+O\left(\frac{1}{\sqrt{n}}\right)\right].
\end{align*}
Following the same routine as that in \cite[Sections 2 and 3]{ET60}, we get
	\begin{align*}
		\p\left(\widetilde{\xi}(0,n)\ge K_2 \right)&\le\left[1-\frac{1}{\frac{15}{16\pi}\big(\log n +O(1)\big)}\right]^{K_2}\\
		&=n^{-1}e^{2\cdot\left(\frac{15}{16\pi}\right)^{-1}(\log n)^{\beta-1}+O(1)}\le n^{-1}e^{8(\log n)^{\beta-1}}	\end{align*}
	for $n$ large.
\end{proof}
\begin{lem}[$d=2$]\label{Tkm}
Let 
\begin{equation}\label{eq:psidef}
\psi_m=\psi_m(\delta):=\exp\big\{\pi^{1/2}m^{1/2}+\pi^{13/10+\delta/2}m^{3/10+\delta/2}\big\}.
\end{equation} 
Then for any $\delta>0$, there exists $c=c(\delta)>0$ such that for all $m,k\ge 1$,
	\begin{align}\label{Tkm1}
		\p\left(T^k_m>\psi_m,M^k_m\right)<e^{-c m}.
	\end{align}
\end{lem}
\begin{proof}
We can readily check that 
\[m=\frac{1}{\pi}(\log \psi_m)^2-2(\log \psi_m)^{8/5+\delta}+O\big[(\log \psi_m)^{6/5+2\delta}\big].\]
Thus by Proposition \ref{lower}, for $m$ large,
\begin{align*}
		&\quad\,\p\left(T^k_m>\psi_m,M^k_m\right)\le \p\big(\xi^*({\psi_m})\le m\big)\\
  &\le\p\left(\xi^*({\psi_m})<\frac{1}{\pi}(\log \psi_m)^2-(\log \psi_m)^{8/5+\delta}\right)\\
  &<Ce^{-\exp\big((\log \psi_m)^{3/5}\big)}<e^{-cm}
	\end{align*}
for some $c>0$.
\end{proof}

\subsection{Two toy models}\label{tysec}

In this subsection, we give two versions of the urn model which serve as toy models for the analysis of the local time of random walk.  In the text below, the readers should think of balls as sites, labels of urns as values of local times and ``ball $x$ being placed in $m$-th urn'' as ``site $x$ having local time $m$''. See Propositions \ref{bigalpha} and \ref{smallalpha} for actual statements in terms of local time.

Consider infinitely many urns labeled by integers \( 1, 2, 3, \ldots \), and \( n \) distinguishable balls labeled from \( 1 \) to \( n \).
Each ball is independently placed into $k$-th urn with probability $p_k>0$.
Let \( F_k \) denote the number of balls in the \( k \)-th box, and write \( X_n := \max\{k : F_k > 0\} \), for the urn with largest label which contains at least one ball.

{\begin{lem}\label{ty2}
If for any $m\geq 1$,
\begin{align*}
p_{m} \lesssim  p_{m+1},
\end{align*}
then there exists $c>0$ such that for any $m,h\ge 1$,
\begin{align*}
\p\big(F_{m}=h,\, X_n=m\big) \le \exp (-c h).
\end{align*}
\end{lem}
\begin{proof}
    Suppose $p_m\le Cp_{m+1}$ for all $m\ge 1$. It is not hard to see that conditionally on $X_n\le m+1$ and $F_m+F_{m+1}=h$, $F_m$ follows a binomial distribution with parameter $\big(h,\frac{p_m}{p_m+p_{m+1}}\big)$. Hence
    \begin{align*}
        \p\big(F_{m}=h,\, X_n=m\big)&\le \p\big\{F_m=h\,\big\vert\,X_n\le m+1,F_m+F_{m+1}=h\big\}\\
        &\le \left(\frac{p_m}{p_m+p_{m+1}}\right)^{h}\le \left(\frac{C}{1+C}\right)^{h}\le \exp (-c h)
    \end{align*}
    for some $c=c(C)>0$.
\end{proof}}


\begin{lem}\label{ty1}
For any $m>0$, pick two integers $0<g_m<f_m <m$ and some integer $J_m>0$. If uniformly for all $m>0$, $k\in ({m-g_m,m}]$ and $l\in ({m-f_m, m-g_m}]$,
\begin{align}\label{tp01}
p_k\lesssim p_l,
\end{align}
then, for any $j\in \N$,
\begin{align}\label{toy1}
\p\Big(\sum_{k= m-g_m+1}^{m} F_{k}=j+1\Big|X_n=m,\sum_{k= m-f_m+1}^{m} F_{k} \le J_m\Big)\lesssim \Big(\frac{g_m}{f_m}\Big)^{j}\cdot J_m^{j}.
\end{align}
\end{lem}
\begin{proof}
By (\ref{tp01}),
\begin{eqnarray*}
\max_{{m-g_m+1\le i\le m}}\{p_{i}\}\lesssim \min_{{m-f_m+1\le i\le m-g_m}}\{p_{i}\}
\Longrightarrow\dfrac{p_{m}+...+p_{m-{g_m}+1}}{p_{m}+...+p_{{m}-{f_m}+1}}\lesssim\dfrac{g_m}{f_m}.
\end{eqnarray*}
For any $l\in[j+1,J_m]$,
\begin{eqnarray*}
&&\p\Big(\sum_{k= m-g_m+1}^{m} F_{k}=j+1\Big|X_n=m,\sum_{k= m-f_m+1}^{m} F_{k} = l\Big)\\
&\lesssim&\frac{l!}{(l-j)!j!}\Big(\dfrac{p_{m}+...+p_{m-{g_m}+1}}{p_{m}+...+p_{{m}-{f_m}+1}}\Big)^j\cdot
\Big(\dfrac{p_{m-g_m}+...+p_{m-{f_m}+1}}{p_{m}+...+p_{{m}-{f_m}+1}}\Big)^{l-j}\\
&\lesssim& \frac{J_m!}{(J_m-j)!j!}\cdot\Big(\frac{g_m}{f_m}\Big)^{j}\lesssim \Big(\frac{g_m}{f_m}\Big)^{j}\cdot J_m^{j}.
\end{eqnarray*}
This implies \eqref{toy1}.
\end{proof}

\section{Favorite sites for $d\ge3$}\label{se:3}
In this section, we prove Theorem \ref{mainthm-01}. We first define various events which facilitate the analysis of favorite sites for $d\geq 3$ and then prove some key properties. We then proceed to the proof of Theorem \ref{mainthm-01}. We tacitly assume $d\ge 3$ throughout this section.
\subsection{Characterization of the event $M_m^k$}
In this subsection, we rewrite the event $M_m^k$ (recall \eqref{Mmk} for the definition), i.e., the occurrence of $k$-favorite sites of local time $m$ as the intersection of events that are easier to analyze. Recall the definition of $T_m^k$ and $L^k_m$ in and below \eqref{tmkdef}.
For $m,k\in \mathbb{Z}^+$, let
\begin{align*}
A_m^{k}
:=&\Big\{S_n \not\in \{L_m^1, \ldots L_m^{k-2} \}, \text{ for any }n \in  [T_m^{k-1}+m/2, {T_m^{k}\wedge T_{m+1}^1}]\Big\} \\
&\qquad \qquad \bigcap \Big\{S_n \neq L_m^{k-1} , \text{ for any }n \in  (T_m^{k-1}, {T_m^{k}\wedge T_{m+1}^1}]\Big\};
\end{align*}
stand for the event that the walk $(S_n)$ does not hit the sites $\{L_m^1, \ldots L_m^{k-2} \}$ between $T_m^{k-1}+\frac{m}{2}$ and $T_m^{k}\wedge T^1_{m+1}$ and also does not hit $L_m^{k-1}$ between $T_m^{k-1}+1$ and $T_m^{k}\wedge T^1_{m+1}$. We also define
\begin{align*}
& \widetilde{A}_m^{k}
:=\big\{S_n \not\in \{L_m^1, \ldots L_m^{k-2} \}, \text{ for any }n \in  \big[ T_m^{k-1}, {(T_m^{k-1}+m/2)\wedge T^k_m\wedge T_{m+1}^1}\big) \big\}
\end{align*}
and 
$$
B_m^{k}:=A_m^{1}\cap\ldots \cap A_m^{k};~~~\widetilde{B}_m^{k}:=\widetilde{A}_m^{1}\cap\ldots \cap \widetilde{A}_m^{k}.
$$
In fact, it is not difficult to see that
\begin{align}\label{characterization}
    M_m^k=B_m^{k}\cap\widetilde{B}_m^{k}.
\end{align}
{Moreover, we have the following result.
\begin{lem}\label{contain}
    Almost surely, there exists $M=M(\omega)\in \z^+$ such that for all $m>M$ and $k\ge 1$, $1_{M^k_m}(\omega)=1_{B^k_m}(\omega)$. Consequently, for any $k\ge 1$, 
    \[\big\{M^k_m\text{ infinitely often in }m\big\}\overset{\text{a.s.}}{=}\big\{B^k_m\text{ infinitely often in  }m\big\}.\]
\end{lem}}


We postpone the proof till the end of this subsection and turn to the maximal local time.
Recall the statement of Theorem \ref{mainthm-01} for the definition of $\gamma_d$, and write
\begin{align}\label{alpha}
\alpha:=\alpha(d)=-\dfrac{1}{\log (1-\gamma_d)}.
\end{align}
For $\epsilon\in (0,1)$, we denote by
\begin{align*}
D_n^\epsilon:=\left\{\left(1-\frac{\epsilon}{2}\right)\alpha\log n \le \xi^*(n) \le \left(1+\frac{\epsilon}{2}\right)\alpha\log n\right\}
\end{align*}
the event that at time $n$, $\xi^*(n)$ is close to its expected value,
and by
$$
E_n^\epsilon:=\{\mathsf{P}_n=0\}, 
$$
where
\begin{equation}\label{Pmmdef}
\mathsf{P}_n:= \sum_{i=1}^n \sum_{j=i+1}^{(i+\alpha\log n)\wedge n}
{1}_{\{\min(\xi(S_i,n), \xi(S_j,n))\ge (1-\epsilon)\alpha\log n,S_i, S_j \not\in S_{(i,j)}, S_i \neq S_j \}}
\end{equation}
i.e., the event that there are no pairs of thick points ``within a short distance of each other'' before time $n$. 
\begin{lem}\label{lem0}
Let $\delta>0$ be defined as \eqref{deltadef} below {and fix $\rho>0$ such that $\rho\delta/2>1+\delta$.} Then for any $\epsilon\in (0,1)$ with $(1+\delta)(1-\epsilon)^2>1+\delta/2$, there exists $C=C(\epsilon)>0$ such that 
\begin{align}\label{infsup}
\p\Big( \bigcup_{n\geq N}(D_n^\epsilon \cap E_n^\epsilon)^{\rm c} \Big)\le C {N^{-\delta/\rho}}.
\end{align}
Consequently,
\[\p\big(D^\epsilon_n\cap E^\epsilon_n\text{ occurs eventually}\big)=1.\]
\end{lem}
\begin{proof}[Proof of Lemma \ref{lem0}]
The bound on $\p((D_n^\epsilon)^{\rm c})$ is quite classical; see e.g. \cite[lines 4-7 in the proof of Theorem 13 and line 4 on p.\ 162]{ET60}. 
It then suffices to bound $(E_n^\epsilon)^{\rm c}$. 

By \cite[(4.1)]{Csaki}, with $t_y:=\p(\h_y<\infty)$, for any $u\in \N$,
\begin{align}\label{cfrrs}
    \p\big(\xi(0,\infty)+\xi(y,\infty)>u\big)=\big(1-\gamma_d/(1+t_y)\big)^u.
\end{align}
By \cite[Lemma 17.13]{Re13}, we know for any $y\in \mathbb{Z}^d\setminus \{0\}$, $t_y^2<t_y\le 1-\gamma_d$,
which implies
\begin{align}\label{cfrrs2}
    -2\log(1-\gamma_d/(1+t_y))>-\log(1-\gamma_d)=\alpha^{-1}.
\end{align}
Then it is not hard to deduce from \eqref{cfrrs} and \eqref{cfrrs2} that with
\begin{align}\label{deltadef}
    \delta:=\inf_{y\in \z^d\setminus\{0\}}\frac{-2\log(1-\gamma_d/(1+t_y))}{\alpha^{-1}}-1>0,
\end{align}
for any $a>0$ and $n\in \mathbb{N}$, 
\begin{align}\label{as2}
\sup_{y\in\z^d\setminus\{0\}}\p\Big(\xi(0,\infty)+\xi(y,\infty) > 2 \alpha a \log n\Big)
\le n^{-a(1+\delta)}.
\end{align}
Recall the constant $\rho$ in the statement of the lemma and the definition of $\mathsf{P}_n$ in \eqref{Pmmdef}. We now define a variant: for $n,\widetilde{n}>0$ let
\begin{eqnarray*}
&&\mathsf{P}_{n,\widetilde{n}}:= \sum_{i=1}^n \sum_{j=i+1}^{(i+\alpha\log n)\wedge n}
{1}_{\left\{S_i, S_j \not\in S_{(i,j},S_i \neq S_j, \min(\xi(S_i,n), \xi(S_j,n))\ge (1-\epsilon)\alpha\log \widetilde{n} \right\}}.
\end{eqnarray*}
Note that for any $N\in [n^\rho, (n+1)^\rho) $,
$ \mathsf{P}_N \le \mathsf{P}_{(n+1)^\rho,n^\rho}$. Hence
\[{\bigcup_{n\ge N^\rho}(E^\epsilon_n)^{\rm c}\subset \bigcup_{n\ge N}\big\{\mathsf{P}_{(n+1)^\rho,n^\rho}>0\big\}}\]
and it suffices to prove
\begin{align}\label{no two near favorites}
\p\Big(\mathsf{P}_{(n+1)^\rho,n^\rho}>0\Big) \le Cn^{-(1+\delta)}.
\end{align}

Writing $\hat{S}_l:=S_{i-l}-S_i$ and $\widetilde{S}_l:=S_{j+l}-S_j$, for any $1\le i\le {(n+1)^\rho}, i+1\le j\le {(i+\alpha\log (n+1)^\rho)\wedge (n+1)^\rho}$,
\begin{equation}\label{back}
\begin{split}
&\;\p\Big(S_i, S_j \not\in S_{(i,j)},S_i \neq S_j,\min(\xi(S_i,(n+1)^\rho),\xi(S_j,(n+1)^\rho))\ge (1-\epsilon)\alpha\log n^\rho\Big)\\
\le \;& \sum_{z \in \mathbb{Z}^d} \p\Big(S_j-S_i =z,
\sum_{l=0}^i {1}_{\{\hat{S}_l\in \{0,z \} \}} +\sum_{l=0}^{(n+1)^\rho-j} {1}_{\{\widetilde{S}_l\in \{0,-z\} \}}\ge 2(1-\epsilon)\alpha\log n^\rho\Big).
\end{split}
\end{equation}
Observing the simple fact that for non-negative random variables $X,Y$,
\[\big\{X+Y\ge q\big\}\subset \bigcup_{r=0}^{[1/\epsilon]}\big\{X\ge r\epsilon q,\, Y\ge [1-(r+1)\epsilon]q\big\},\]
one has,
\begin{equation*}
\begin{split}
& \mbox{RHS of }\eqref{back}
\\
\le \;\,& \sum_{z \in \mathbb{Z}^d} \p(S_j-S_i =z)
\cdot\sum_{r=0}^{[1/\epsilon]}\p\Big(\sum_{l=0}^i {1}_{\{\hat{S}_l\in \{0,z \} \}}\ge 2r\epsilon(1-\epsilon)\alpha\log n^\rho\Big)\\
&\qquad\qquad\cdot\p\Big(\sum_{l=0}^{(n+1)^\rho-j} {1}_{\{\widetilde{S}_l\in \{0,-z\} \}}\ge 2[1-(r+1)\epsilon](1-\epsilon)\alpha\log n^\rho\Big)\\
\le\;\,&([1/\epsilon]+1){\sup_{r=0,\ldots,[1/\epsilon]]}}\Big\{\sup_{y\in \z^d\setminus\{0\}}\p\Big(\xi(0,\infty)+\xi(y,\infty)\ge 2r\epsilon(1-\epsilon)\alpha\log n^\rho\Big)\\
&\qquad\qquad\cdot\sup_{y\in \z^d\setminus\{0\}}\p\Big(\xi(0,\infty)+\xi(y,\infty)\ge 2[1-(r+1)\epsilon](1-\epsilon)\alpha\log n^\rho\Big)\Big\}\\
\overset{\eqref{as2}}{\le}\,& Cn^{-\rho(1+\delta)(1-\epsilon)^2}\le  Cn^{-\rho(1+\gre{\delta/2})}.
\end{split}
\end{equation*}
Summing over $i,j$ and using that $\rho\delta/2>1+\delta$ gives (\ref{no two near favorites}).
\end{proof}
Note that 
\begin{align*}
    \p\Big(\xi(0,\infty) > \alpha a \log n\Big)=(1-\gamma_d)^{\alpha a\log n}=n^{-a}.
\end{align*}
Compared with \eqref{as2}, we see that it is much harder to produce a pair of thick points than to produce one thick point. This is the main intuition behind the proof.

We now turn back to Lemma \ref{contain}.
{\begin{proof}[Proof of Lemma \ref{contain}]
We fix some $\epsilon$ satisfying the conditions in Lemma \ref{lem0}. By Lemma \ref{lem0}, we know that almost surely, there exists $N_0=N_0(\omega)\in\mathbb{Z}^+$ such that for any $n>N_0$, $D_n^\epsilon\cap E_n^\epsilon$ occurs. In particular, if $\xi^*(n) =m$, then
\begin{equation}\label{nlrelation}
   \exp\left(\frac{m}{(1+\epsilon/2)\alpha}\right) \le n \le \exp\left(\frac{m}{(1-\epsilon/2)\alpha}\right).
\end{equation}
\begin{lem}\label{contain1}
Fix $M'=M'(\epsilon)>0$ such that
\begin{align}\label{M_0}
    (1-\epsilon)\alpha \log \left\{\exp\left(\frac{m+1}{(1-\epsilon/2)\alpha}\right)+\frac{m}{2}\right\}\le m\quad\text{ for all }m\ge M',
\end{align}
and write $M=M(\varepsilon,\omega):=M'\vee 2\alpha \log N_0$.
    Then almost surely, for any $m>M$ and $k\ge 1$, if $M^k_m$ occurs, then $\widetilde{A}^{k+1}_m$ occurs.
\end{lem}
\begin{proof}
We proceed by contradiction. Assume that for some $m>M$ and $k\geq 1$, $M_m^{k}$ occurs but
$\widetilde{A}_m^{k+1}$ does not occur. In this case, we observe that by \eqref{nlrelation},
\begin{equation}\label{mnbound}
\exp\left(\frac{m}{2\alpha}\right)\leq T_m^1< T_m^{k}<T_{m+1}^1\leq \exp\left(\frac{m+1}{(1-\epsilon/2)\alpha}\right),
\end{equation}
and the random walk $(S_n)_{n\geq 0}$ must have hit $L_m^{k}$ at time $T_m^{k}$ and one of $L_m^1,\ldots,L_m^{k-1}$ (say $L_m^{k'}$) before $n':=T_m^{k}+m/2$.
Since
$$
n'>N_0,\quad\alpha \log n' \geq m/2,\quad\mbox{and}\quad m \geq (1-\epsilon)\alpha \log n'.
$$
by \eqref{mnbound} and \eqref{M_0}, we know that on the one hand, $E_{n'}^\epsilon$ occurs, and on the other hand, $L_m^{k}$ and $L_m^{k'}$ qualify as a pair of thick points in the calculation of $\mathsf{P}_{n'}$, which yields a contradiction. 
\end{proof}
Turn back to the proof of Lemma \ref{contain}. Thanks to \eqref{characterization}, it is plain that if $M^k_m$ occurs, then $B^k_m$ occurs. For the other direction, we claim that for any $\omega$ satisfying the conclusion in Lemma \ref{contain1}, $m>M$, and $k\ge 1$, if $B^k_m$ occurs, then $M^k_m$ occurs. 
Indeed, the event $M^1_m$ always occurs. Based on this and further using Lemma \ref{contain1} and the relation that
\[M^j_m=M^{j-1}_m\cap A^j_m\cap \widetilde{A}^j_m,\] 
we can recursively prove that $M^j_m$ occurs for $j=2,\ldots, k$. 
This completes the proof.
\end{proof}}

\subsection{Proof of Theorem \ref{mainthm-01}}

\begin{proof}[Proof of Theorem \ref{mainthm-01}]
By Lemma \ref{lem0}, it suffices to show that a.s.,
\begin{align*}
\limsup_{m\to \infty} \frac{\mathcal{N}_m}{\log m}=-\frac{1}{\log \gamma_d}(=:\beta),
\end{align*}
where $$\mathcal{N}_m:=\sup\{k:T_m^k<T_{m+1}^1\}.$$

We start with the upper bound. Let
$
U_m:=[(1+\epsilon) \beta \log m].
$
On the event
$$
G_m:= \bigcap_{n\geq \exp(m/2\alpha)}\left(D^\epsilon_n\ \cap E^\epsilon_n\right),
$$
if $\{\mathcal{N}_m > U_m\}=M^{U_m+1}_m$
occurs, it follows that
$$
T^1_{m+1}-T^j_m>T_{m}^{j+1}-T_m^j\ge m/2,
$$
and thus 
$$
S_{T_m^j}\notin S_{[T_m^j+1,T_m^j+m/2)}\quad \mbox{ for any }j=1,\ldots,U_m.
$$
By applying the strong Markov property at $T_m^j$, $j=1,\ldots,U_m$  successively,
we see that
\begin{eqnarray*}
 \p\Big( \{\mathcal{N}_m > U_m\} \cap G_m\Big)
&\le &\p\Big(S_0\notin S_{[1,m/2)}\Big)^{U_m}\\
&\le &\Big(\p\big(S_0\notin S_{[1,\infty)}\big)+\p\big(S_0\in S_{[m/2,\infty)}\big)\Big)^{U_m}\\
&\overset{(\ref{hitk})}{\le} & (\gamma_d+cm^{-d/2+1})^{U_m} \le  Cm^{-(1+\epsilon)}.
\end{eqnarray*}
Therefore,
\begin{eqnarray*}
\p\Big(\mathcal{N}_m > U_m\Big) \le  \p\Big( \{\mathcal{N}_m \ge U_m \} \cap G_m \Big)
+\p\Big(G_m^{\rm c} \Big) \overset{(\ref{infsup})}{\le} Cm^{-(1+\epsilon)}.
\end{eqnarray*}
The upper bound then follows from the Borel-Cantelli lemma.

We now turn to the lower bound. 
Note that for $k>1$,
\begin{equation}\label{amkcomplement}
A_m^k\supset \{\mathsf{H}_{\{L_m^1,\ldots, L_m^{k-1}\}}(T_m^{k-1}+m/2)=\infty \} \cup \{\mathsf{H}_{L_m^{k-1}}(T_m^{k-1})=\infty\}.
\end{equation}
Recall the definition of $\mathcal{F}_m^k$ in \eqref{sigma}. By the strong Markov property, the definition of $\gamma_d$ and (\ref{hitk}), \eqref{amkcomplement} implies that
$$
\p(A_m^k | \mathcal{F}_{m}^{k-1}) \geq \gamma_d- ck  m^{-d/2+1}.
$$
In addition, note that trivially $A_m^1$ holds.
Since for any $m,k\geq 1$, $A_m^k,B_m^k \in \mathcal{F}_{m}^k$, picking
$$K_m:=[(1-\epsilon)\beta \log m],$$
by tower property of conditional expectation applied to the filtration $(\mathcal{F}_m^k)_{k\geq 1}$,
$$
\p(B_m^{K_m} | \mathcal{F}_m^1) \geq \prod_{j=2}^{K_m} (\gamma_d-cj m^{-d/2+1})\geq   cm^{-1+\epsilon}.
$$
Hence,
\begin{align*}
\sum_{m=1}^{\infty}\p\Big(B_m^{K_m}\Big|\cal{F}_{m}^1\Big)\ge\sum_{m=1}^{\infty}cm^{-1+\epsilon}=\infty.
\end{align*}
Note that similar to \eqref{Mmkmeasure}, $B_m^k \in \mathcal{F}_{m+1}^1$. 
The generalized second Borel-Cantelli lemma (see e.g., \cite[Theorem 5.3.2]{Du19}) then implies that a.s., $B_m^{K_m}$ occurs infinitely often. By Lemma \ref{contain}, this implies a.s., $M_m^{K_m}$ also occurs i.o.. In other words,
$
\p(\mathcal{N}_m \ge K_m \;\text{ i.o.})=1.$
This finishes the proof of the lower bound.
\end{proof}

\section{Favorite sites for $d=2$}\label{se:4}

In this section, we prove Theorem \ref{mainthm2}. In Section \ref{se:2dlb}, we give the more straightforward lower bound. The proof of the upper bound is much more demanding and will be laid out in the subsequent subsections. 

\subsection{The lower bound in Theorem \ref{mainthm2}}\label{se:2dlb}
{Recall the representation \[M_m^k=U^2_m\cap\cdots\cap U^k_m\] in \eqref{representation}.
Then we have
\begin{lem}\label{stopping time estimate}
There exists $c>0$ such that for $i=2,3$,
\begin{align}\label{stopping time for}
\p\big(U^{i}_m\big| \mathcal{F}_m^{i-1} \big)\ge cm^{-1/2}.
\end{align}
\end{lem}
\begin{proof}
By the observation that it is easier to create another site with local time $m$ than the first one, and the strong Markov property, we have
\begin{align}\label{u3m}
\begin{aligned}
    \p\left(U_m^3\big\vert\mathcal{F}^2_m\right)&\ge \p\big\{S_n\notin \{L_m^1,L_m^2\}\text{ for }n\in (T^2_m,T^2_m+T^1_m\circ \theta_{T^2_m}]\big\vert\mathcal{F}^2_m\big\}\\
    &=\p\big(\h_{\{0,y\}}>T^1_m\big)\big\vert_{y=S_{T^1_m}-S_{T^2_m}}.    
\end{aligned}
\end{align}
\red{Here $\theta_{T^2_m}$ denotes the time-shift operator, 
so that $T^1_m \circ \theta_{T^2_m}$ represents the first time, along the shifted path 
after $T^2_m$, that a site accumulates local time $m$.}
Note that for any $y\in \z^2$,
\begin{align*}
    \p\big(\h_{\{0,y\}}>T^1_m\big)&\ge \p\big\{\h_{\{0,y\}}>T^1_m,\, T^1_m\le e^{2m^{1/2}}\big\}\\
    &\ge \p\big\{\h_{\{0,y\}}>e^{2m^{1/2}}\big\}-\p\big(T^1_m>e^{2m^{1/2}}\big)\\
    &\overset{(\ref{2d-lem for})}{\underset{(\ref{Tkm1})}{\ge}}cm^{-1/2}
\end{align*}
for some universal $c>0$. Substituting back to \eqref{u3m} gives \eqref{stopping time for} for $i=3$. The case $i=2$ is proved in the same way.
\end{proof}
\begin{proof}[Proof of lower bound in Theorem \ref{mainthm2}]
By tower property of conditional expectation, we have
\begin{align*}
    \p\big(M^3_m\big\vert \mathcal{F}^1_m\big)&=\e\Big\{\p\big(U^3_m\big\vert \mathcal{F}^2_m\big)1_{U^2_m}\Big\vert\mathcal{F}^1_m\Big\}\\
    &\overset{\eqref{stopping time for}}{\ge}cm^{-1/2}\p\big(U^2_m\big\vert \mathcal{F}^1_m\big)\\
    &\overset{\eqref{stopping time for}}{\ge}cm^{-1/2}\cdot cm^{-1/2}=c^2m^{-1}.
\end{align*}
Again, by the generalized second Borel-Cantelli lemma (recall \eqref{Mmkmeasure} for measurability), we get that almost surely, $M_m^3$ occurs infinitely often, or equivalently $\limsup_{n \to \infty} {\#\mathcal{K}^{(2)}(n)}\geq 3$, as desired.
\end{proof}}

\subsection{Preliminaries for the upper bound in Theorem \ref{mainthm2}: decomposition of local time}\label{Preliminaries}
In this subsection, we elaborate on the idea of decomposition of local time, first introduced in Section \ref{se:2}, and pair the lattice according to parity and state some relatively straightforward observations.

Recall the definitions of $N_n$, $N'_n$, $\widetilde{S}$, $\widetilde{S}'$,  $\widetilde{\xi}$, $\widetilde{\xi}'$ and $p(i,j)$ in Section \ref{se:2}.
We define 
\begin{align}\label{Txl}
    \widetilde{T}(x,l):=\inf\{j\ge 0:\widetilde{\xi}(x,j)=l\},\quad\mbox{ and }\quad\widetilde{T}'(x,l):=\inf\{j\ge 0:\widetilde{\xi}'(x,j)=l\}.
\end{align}

 Write 
 \begin{equation}\label{eq:Xdef}
     \mathbf{X}:=\big\{(x,x+{\bf e}_1):x\in \z^2_{\rm e}\big\}
 \end{equation}
 for a ``pairing'' of $\mathbb{Z}^2$. The motivation for introducing such a pairing will be revealed in Section \ref{se:1.1ub} (where more pairings will be introduced too).
We use $\mathbf{x}$ to denote a generic element in $\mathbf{X}$ and for $x\in \z^2$, $\mathbf{x}(x)$ represents the unique element in $\mathbf{X}$ which contains $x$. For $\mathbf{x}=(x,x+{\bf e}_1)$ and $n\ge 1$, we write
\begin{align*}
    \xi(\mathbf{x},n)&:=\big(\xi(x,n),\xi(x+{\bf e}_1,n)\big),\\
    {\xi_{\rm L}(\x,n)}&{:=\big(\xi_{\rm L}(x,n),\xi_{\rm L}(x+{\bf e}_1,n)\big)},\quad\widetilde{\xi}(\mathbf{x},n):=\big(\widetilde{\xi}(x,n),\widetilde{\xi}(x+{\bf e}_1,n)\big).
\end{align*}
We also write
$$
\xi^+(\x,n):=\xi(x,n)\vee \xi(x+{\bf e}_1,n),\quad\mbox{ and }\quad\xi^-(\x,n):=\xi(x,n)\wedge \xi(x+{\bf e}_1,n),$$ 
and define $\widetilde{\xi}^\pm(\x,n)$ in a similar fashion.
Moreover, we also define $\xi'$, $\xi'_{\rm L}$, $\widetilde{\xi}'$, $(\xi')^\pm$, $(\widetilde{\xi}')^\pm$ for $(\x,n)\in \X\times \N$ through replacing $\xi$ and $\widetilde{\xi}$ in the above definitions by $\xi'$ and $\widetilde{\xi}'$ respectively. In the following, we occasionally identify $m\ge 1$ with the 2D vector $(m,m)$. For example, we will write $m-\widetilde{\xi}(\mathbf{x},n)$ instead of $(m-\widetilde{\xi}(x,n),m- \widetilde{\xi}(x+{\bf e}_1,n))$.

 Let
\begin{align*}
	N^{-1}_-(n):=\inf\left\{j\ge 0:N_j=n\right\},\quad N^{-1}_+(n):=\sup\left\{j\ge 0:N_j=n\right\},
\end{align*}
and define
\begin{align*}
    h_n:=\left(N^{-1}_+(n)-N^{-1}_-(n)\right)/2\text{ for even }n,
\end{align*}
which is the number of consecutive excursions $(\widetilde{S}_n,\widetilde{S}_n+{\bf e}_1,\widetilde{S}_n)$ in the path of $S$ right after the $N^{-1}_-(n)$-th step. We write 
\begin{align}\label{hxl}
    h(x,l)=h_n\text{ when }n=\widetilde{T}(x,l).
\end{align}
Moreover, we also define  $(N')^{-1}_\pm$, $h'_n$, and $h'(x,l)$ through replacing $N_j$ by $N'_j$ in the above definitions.

The following propositions are more sophisticated versions of \eqref{CLT*} whose proofs are still fairly straightforward. We only prove Proposition \ref{condilaw2}. {Observe that for $n\in \N$ and $\x=(x,x+{\bf e}_1)\in \X$, 
\begin{align}\label{iff}
    \xi_{\rm L}(x,n)\neq \xi_{\rm L}(x+{\bf e}_1,n)\text{ if and only if }n+1\in \mathcal{L}\text{ and }x=S_{n-1}.
\end{align}
Since $\mathcal{L}$ contains neither $N^{-1}_-(n)+1$ nor $N^{-1}_+(n)+1$ (otherwise $N_{N^{-1}_-(n)+1}\text{ or }N_{N^{-1}_+(n)+1}=n-1<n=N_{N^{-1}_\pm(n)}$), we have for any $n\in \N$ and $\x=(x,x+{\bf e}_1)\in \X$,
\[\xi_{\rm L}(x,N^{-1}_-(n))=\xi_{\rm L}(x+{\bf e}_1,N^{-1}_-(n))\qquad\mbox{and}\qquad \xi_{\rm L}(x,N^{-1}_+(n))=\xi_{\rm L}(x+{\bf e}_1,N^{-1}_+(n)).\]}
Also note that the $\sigma$-algebra generated by $\widetilde{S}_{[0,n]}$ contains the information on $(\widetilde{\xi}(x,n):x\in \z^2)$. 
\begin{pro}\label{condilaw1}
For any $n\ge 0$, conditionally on $\widetilde{S}_{[0,n]}$, $\left\{h_j:0\le j\le n\right\}$ are i.i.d. following a $\{0,1,\ldots\}$-valued geometric distribution with success probability $15/16$. Consequently,
$\big(\xi_{\rm L}(\mathbf{x},N^{-1}_{\pm}(n)):\mathbf{x}\in\mathbf{X}\big)$ are conditionally independent with the following distribution. 
\begin{itemize}
	\item {If $n$ is odd}, then for any $\x=(x,x+{\bf e}_1)$, $\xi_{\rm L}(\x,N^{-1}_-(n))=\xi_{\rm L}(\x,N^{-1}_+(n))$ and
		\[\p\left\{\xi_{\rm L}(\x,N^{-1}_-(n))=l\,\Big\vert\,\widetilde{S}_{[0,n]}\right\}=p(\widetilde{\xi}(x,n),l).\]
	\item {If $n$ is even}, then for $\x=(x,x+{\bf e}_1)\neq \x(\widetilde{S}_n)$, $\xi_{\rm L}(\mathbf{x},N^{-1}_{\pm}(n))$ has the same law as that in the previous case; while for $\x(\widetilde{S}_n)=(x,x+{\bf e}_1)$,
		\[\p\left\{\Big(\xi_{\rm L}(\x,N^{-1}_-(n)),\,\xi_{\rm L}(\x,N^{-1}_+(n))\Big)=(l_1,l_2)\,\Big\vert\,\widetilde{S}_{[0,n]}\right\}=p(\widetilde{\xi}(x,n)-1, l_1)p(1,l_2-l_1) .\]
\end{itemize}
\end{pro}
\begin{probis}\label{condilaw11}
For any $n\ge 0$, conditionally on $\widetilde{S}'[0,n]$, $\left\{h'_j:0\le j\le n\right\}$ are i.i.d. following a $\{0,1,\ldots\}$-valued geometric distribution with success probability $15/16$. Consequently,
$\big(\xi'_{\rm L}(\mathbf{x},(N')^{-1}_{\pm}(n)):\mathbf{x}\in\mathbf{X}\big)$ are conditionally independent with the following distribution.
\begin{itemize}
	\item {If $n$ is even}, then for any $\x=(x,x+{\bf e}_1)$, $\xi'_{\rm L}(\x,(N')^{-1}_-(n))=\xi'_{\rm L}(\x,(N')^{-1}_+(n))$ and
		\[\p\left\{\xi'_{\rm L}(\x,(N')^{-1}_-(n))=l\,\Big\vert\,\widetilde{S}'[0,n]\right\}=p(\widetilde{\xi}'(x+{\bf e}_1,n),l).\]
	\item {If $n$ is odd}, then for $\x=(x,x+{\bf e}_1)\neq \x(\widetilde{S}'_n)$, $\xi'_{\rm L}(\mathbf{x},(N')^{-1}_{\pm}(n))$ has the same law as that in the previous case; while for $\x(\widetilde{S}'_n)=(x,x+{\bf e}_1)$,
	\begin{align*}
		&\quad\,\p\left\{\Big(\xi'_{\rm L}(\x,(N')^{-1}_-(n)),\,\xi'_{\rm L}(\x,(N')^{-1}_+(n))\Big)=(l_1,l_2)\,\Big\vert\,\widetilde{S}'[0,n]\right\}\\
		&=p(\widetilde{\xi}'(x+{\bf e}_1,n)-1, l_1)p(1,l_2-l_1).
	\end{align*}
\end{itemize}
\end{probis}
We now give the ``stopping-time versions'' of Propositions \ref{condilaw1} and \ref{condilaw11}. {Note that by \eqref{iff}, for any $m,k\ge 1$ and $\x=(x,x+{\bf e}_1)\in \X\setminus\{\x(L^k_m)\}$, }
\[{\xi_{\rm L}(x,T^k_m)=\xi_{\rm L}(x+{\bf e}_1,T^k_m).}\]
\begin{pro}\label{condilaw2}
	For any $m,k\ge 1$ and $k$ distinct points $\{y_1,\ldots,y_k\}\subset \z^2$, conditionally on $\widetilde{S}_{[0,N_{T^k_m}]}$, $L^j_m=y_j$ for $j=1,\ldots,k$, and $M^k_m$, $\Big(\xi_{\rm L}(\mathbf{x},T^k_m):\mathbf{x}\in\mathbf{X}\setminus \{\x(y_j):j=1,\ldots,k\}\Big)$ are independent with the following distribution. 
	For $\x=(x,x+{\bf e}_1)\in\mathbf{X}\setminus \{\x(y_j):j=1,\ldots,k\}$, 
  \begin{align*}
      &\p\left\{\xi_{\rm L}\left(\x,T^k_m\right)=l\,\Big\vert\, \widetilde{S}_{[0,N_{T^k_m}]},L^j_m=y_j\,\forall\,j=1,\ldots,k,M^k_m\right\}\\=\;&1_{\{l< m-\widetilde{\xi}^+(\x,N_{T^k_m})\}}\frac{p(\widetilde{\xi}(x,N_{T^k_m}),l)}{\sum_{i=0}^{m-\widetilde{\xi}^+(\x,N_{T^k_m})-1}p(\widetilde{\xi}(x,N_{T^k_m}),i)}.
  \end{align*}
\end{pro}
\begin{proof}
Fixing $m$ and $k$, for any even $n\in [0,N_{T^k_m}]$, let
\[\hat{h}_n:=\begin{cases}
    h_n,&\text{ if }n<N_{T^k_m},\\
    \Big[\big(T^k_m\wedge N^{-1}_+(n)-N^{-1}_-(n)\big)/2\Big],&\text{ if }n=N_{T^k_m},
\end{cases}\]
which is the number of consecutive excursions $(\widetilde{S}_n,\widetilde{S}_n+{\bf e}_1,\widetilde{S}_n)$ right after the $N^{-1}_-(n)$-th step before time $T^k_m$. We define $\hat{h}(x,l):=\hat{h}_n$ when $n=\widetilde{T}(x,l)$ (recall \eqref{Txl}).
Note that by adding such excursions to the path $\widetilde{S}_{[0,N_{T^k_m}]}$, we get either $S_{[0,T^k_m]}$ or $S_{[0,T^k_m-1]}$. Here the latter case happens if and only if $T^k_m+1\in\mathcal{L}$, which is also equivalent to $\{L^k_m=\widetilde{S}_{N_{T^k_m}}+{\bf e}_1\}$.

Thus under the conditioning, to get the whole path $S_{[0,T^k_m]}$, it remains to decide $$\Big\{\hat{h}_n:0\le n\le N_{T^k_m},n\text{ is even}\Big\}$$ or equivalently 
\[\left\{\hat{h}(x,l):x\in \widetilde{S}_{[0,N_{T^k_m}]}\cap \z^2_{\rm e},l=1,\ldots,\widetilde{\xi}(x,N_{T^k_m})\right\}.\]
This naturally defines a bijection $\Phi$ from all possibilities of $\{\hat{h}(x,l)\}$, denoted by $\hat{H}$, to all possibilities of $S_{[0,T^k_m]}$. We observe that $\hat{H}$ is the collection of $\{\rho(x,l):x,l\}$ satisfying
\begin{itemize}
    \item for any $x\in \widetilde{S}_{[0,N_{T^k_m}]}\cap \z^2_{\rm e}$ with $\x(x)\in \left\{\x(y_j):j=1,\ldots,k\right\}$,
    \begin{align}\label{hatH1}
        \sum_{l=1}^{\widetilde{\xi}(x,N_{T^k_m})}\rho(x,l)+\widetilde{\xi}^+(\x,N_{T^k_m})=m-1{\left\{y_k=\widetilde{S}_{N_{T^k_m}}+{\bf e}_1=x+{\bf e}_1\right\}};
    \end{align}
    \item for any $x\in \widetilde{S}_{[0,N_{T^k_m}]}\cap \z^2_{\rm e}$ with $\x(x)\notin \left\{\x(y_j):j=1,\ldots,k\right\}$,
    \begin{align}\label{hatH2}
        \sum_{l=1}^{\widetilde{\xi}(x,N_{T^k_m})}\rho(x,l)+\widetilde{\xi}^+(\x,N_{T^k_m})<m.
    \end{align}
\end{itemize}

Also note that for any $\rho\in \hat{H}$,
\begin{align*}
    \p\left(S_{[0,T^k_m]}=\Phi(\rho)\right)&\asymp \left(\frac{1}{4}\right)^{\text{length of }\Phi(\rho)}\asymp\left(\frac{1}{4}\right)^{N_{T^k_m}+2\sum_{x,l}\rho(x,l)}\\
    &=\left(\frac{1}{4}\right)^{N_{T^k_m}}\prod_{x,l}\left(\frac{1}{16}\right)^{\rho(x,l)}\\
    &\overset{(*)}{=}\left(\frac{1}{4}\right)^{N_{T^k_m}}\left(\frac{16}{15}\right)^{[N_{T^k_m}/2]+1}\prod_{x,l}\left\{\left(\frac{1}{16}\right)^{\rho(x,l)}\left(\frac{15}{16}\right)\right\},
\end{align*}
where $N_{T^k_m}$ on the right-hand side particularly represents the length of the conditioned path of $S_{[0,T^k_m]}$ and $(*)$ follows from the observation that 
\[\#\big\{(x,l):x\in \widetilde{S}_{[0,N_{T^k_m}]}\cap \z^2_{\rm e},l=1,\ldots,\widetilde{\xi}(x,N_{T^k_m})\big\}=[N_{T^k_m}/2]+1.\]
Hence we get $\{\hat{h}(x,l):l\}$ for $x\in \widetilde{S}_{[0,N_{T^k_m}]}\cap \z^2_{\rm e}$ are conditionally independent, and for each $x$, $\{\hat{h}(x,l):l\}$ has the same law as independent geometric random variables with success probability $15/16$ conditioned on their sum satisfying \eqref{hatH1} or \eqref{hatH2} according as $\x(x)\in \left\{\x(y_j):j=1,\ldots,k\right\}$ or $\x(x)\notin \left\{\x(y_j):j=1,\ldots,k\right\}$. This easily leads to the conclusion.
\end{proof}

\begin{probis}\label{condilaw21}
	For any $m,k\ge 1$, conditionally on $\widetilde{S}'_{[0,N'_{T^k_m}]}$, $\left(L^j_m:j=1,\ldots,k\right)$, and $M^k_m$, $\left(\xi_{\rm L}(\mathbf{x},T^k_m):\mathbf{x}\in\mathbf{X}\setminus \{\x(L^j_m):j=1,\ldots,k\}\right)$ are independent with the following distribution. 
		For $\x=(x,x+{\bf e}_1)\in\mathbf{X}\setminus \{\x(y_j):j=1,\ldots,k\}$,
  \begin{align*}
      \p\left\{\xi_{\rm L}\left(\x,T^k_m\right)=l\,\Big\vert\, \widetilde{S}'_{[0,N'_{T^k_m}]},L^j_m=y_j\,\forall\,j=1,\ldots,k,M^k_m\right\}\\=1_{\{l< m-\widetilde{\xi}^+(\x,N'_{T^k_m})\}}\frac{p(\widetilde{\xi}(x+{\bf e}_1,N'_{T^k_m}),l)}{\sum_{i=0}^{m-\widetilde{\xi}^+(\x,N'_{T^k_m})-1}p(\widetilde{\xi}(x+{\bf e}_1,N'_{T^k_m}),i)}.
  \end{align*}
\end{probis}

\subsection{Preliminaries for the upper bound in Theorem \ref{mainthm2}: local time bounds}

The goal of this subsection is to prove Proposition \ref{Theta}, a crucial step in the proof of the upper bound. This proposition dictates in a quantitative fashion that with high probability, the lazy part of the local time of ``near-favorite'' sites (whose formal definition will be introduced in the next subsection) must behave relatively regularly. 

We start with a large deviation bound on the stopping time $T_m^k$ from Lemma \ref{Tkm}. \red{For the subsequent estimates, we fix 
\begin{equation}\label{eq:kappa1}
    \kappa_1 \in (1/3,7/20),
\end{equation}
and treat it as a constant (i.e., omitting the notation for the dependence of other constants upon $\kappa_1$).}
With this choice, we set $\delta=7/5-4\kappa_1$ in the definition of $\psi_m$ in \eqref{eq:psidef}, i.e., let
\begin{equation}
    \psi_m=\psi_m(7/5-4\kappa_1):=\exp\left\{\pi^{1/2}m^{1/2}+\pi^{2-2\kappa_1}m^{1-2\kappa_1}\right\},
\end{equation}
which satisfies
\begin{align}\label{psi_m}
    m=\frac{1}{\pi}(\log \psi_m)^2-2(\log \psi_m)^{3-4\kappa_1}+O\big[(\log m)^{4-8\kappa_1}\big].
\end{align}
Then Lemma \ref{Tkm} particularly implies that there exists $c>0$ such that  for any $m,k\ge 1$,
\begin{align}\label{Mestimates1}
    \p\left(T^k_m>\psi_m,M^k_m\right)<e^{-c m}.
\end{align}

Next, we record an estimate on the number of near-maxima of local time. 
\begin{pro}\label{Mestimates}
For sufficiently large $m$,
\begin{align}
\p\left(\#\Big\{x\in \z^2_{\rm e}:\widetilde{\xi}(x,\psi_m)>\frac{15}{16}m-m^{4/5}\Big\}>e^{16m^{1-2\kappa_1}}\right)< e^{-m^{1-2\kappa_1}},\label{Mestimates_1}\\
\p\left(\#\Big\{x\in \z^2_{\rm o}:\widetilde{\xi}'(x,\psi_m)>\frac{15}{16}m-m^{4/5}\Big\}>e^{16m^{1-2\kappa_1}}\right)< e^{-m^{1-2\kappa_1}}\label{Mestimates_2}.
\end{align}
\end{pro}
\begin{proof}
\eqref{Mestimates_2} follows from \eqref{Mestimates_1} by considering the shifted path $\widetilde{S}\circ\theta_1$ and using the Markov property. Thus we only need to prove \eqref{Mestimates_1}.

	Using \eqref{psi_m}, we have for $m$ large,
\begin{align}\label{Markov}
\begin{aligned}
	 &\quad\,\p\left(\#\Big\{x\in \z^2_{\rm e}:\widetilde{\xi}(x,\red{\psi_m})\ge \frac{15}{16}m-m^{4/5}\Big\}> e^{16m^{1-2\kappa_1}}\right)\\
	&\le \p\left(\#\Big\{x\in \z^2_{\rm e}:\widetilde{\xi}(x,\psi_m)\ge \frac{15}{16\pi}(\log \psi_m)^2-2(\log \psi_m)^{3-4\kappa_1}\Big\}> e^{16m^{1-2\kappa_1}}\right).	
\end{aligned}
\end{align}
Note that
\begin{align*}
&\quad\,\e\left(\#\left\{x\in \z^2_{\rm e}:\widetilde{\xi}(x,n)\ge \frac{15}{16\pi}(\log n)^2-2(\log n)^{3-4\kappa_1}\right\}\right)\\
&\le \e\left(\#\left\{0\le j\le [n/2]:\widetilde{\xi}(0,n)\circ \theta_{2j}\ge \frac{15}{16\pi}(\log n)^2-2(\log n)^{3-4\kappa_1}\right\}\right)\\
&\le (n/2+1)\p\left\{\widetilde{\xi}(0,n)\ge \frac{15}{16\pi}(\log n)^2-2(\log n)^{3-4\kappa_1}\right\} \\
&\overset{\eqref{side 01}}{\le} e^{8(\log n)^{2-4\kappa_1}}
\end{align*}
for sufficiently large $n$. Thus an application of Markov's inequality yields that 
\[\mbox{ RHS of }\eqref{Markov}\le e^{8(\log \psi_m)^{2-4\kappa_1}}/e^{16m^{1-2\kappa_1}}<e^{-m^{1-2\kappa_1}}.\qedhere\]
\end{proof}

We finally arrive at the goal of this subsection.
 Recall $\sigma^2=\frac{16}{225}$ defined in Lemma \ref{moderate}. Henceforth we fix a large $c_*$  such that $$\frac{1}{2\sigma^2}\frac{16}{15}c^2_*\ge 18.$$ For any interval $I=[a,b)\subset [0,m)$, let
  \begin{equation}\label{eq:thetadef}
  \Theta(k,m,I):=\Theta_-(k,m,I)\cup \Theta_+(k,m,I)    
  \end{equation}
   where
	\begin{align*}
	\Theta_-(k,m,I)&:= \Big\{x\in \z^2_{\rm e}:\widetilde{\xi}(x,N_{T^k_m})\le \frac{15}{16}a-c_* m^{1-\kappa_1},\xi(x,T^k_m)\in I\Big\},\\
	\Theta_+(k,m,I)&:= \Big\{x\in \z^2_{\rm e}:\widetilde{\xi}(x,N_{T^k_m})> \frac{15}{16}b+c_* m^{1-\kappa_1},\xi(x,T^k_m)\in I\Big\},
\end{align*}
and $\Theta'(k,m,I)$ is defined by replacing $\z^2_{\rm e}$, $\widetilde{\xi}$, $N$, $\xi$ by $\z^2_{\rm o}$, $\widetilde{\xi}'$, $N'$, $\xi'$ respectively.  
\begin{pro}\label{Theta}
For any $\varepsilon>0$, there exists $c=c(\varepsilon)>0$ such that for all $m,k\ge 1$ and $I=[a,b)\subset (m-m^{4/5-\varepsilon},m]$,
\begin{align*}
	\p\left\{\Theta(k,m,I)\cup \Theta'(k,m,I)\neq \emptyset,M^k_m\right\}
 <e^{-cm^{1-2\kappa_1}}.
\end{align*} 
\end{pro}
\begin{proof}
It suffices to show that
\begin{align}
	&\p\left\{\Theta(k,m,I)\neq \emptyset,M^k_m\right\}<e^{-cm^{1-2\kappa_1}},\mbox{ and }\label{Thetakm1}\\
	&\p\left\{\Theta'(k,m,I)\neq \emptyset,M^k_m\right\}<e^{-cm^{1-2\kappa_1}}.\label{Thetakm2}
\end{align} 
	We will prove only \eqref{Thetakm1}. The proof of \eqref{Thetakm2} is quite similar and hence omitted. 
	
	First, let us prove
\begin{align}\label{Theta-}
	\p\left(\Theta_-(k,m,I)\neq \emptyset,M^k_m\right)<e^{-cm^{1-2\kappa_1}}.
\end{align}
By \eqref{Mestimates1}, 
\begin{align}\label{Thetaest1}
	\p(N_{T^k_m}> \psi_m,M^k_m)\le \p(T^k_m> \psi_m,M^k_m)<e^{-c_1m}
\end{align}
for some universal $c_1>0$. {Recall the notation $h(x,l)$ defined in \eqref{hxl}.} We observe that on $\{N_{T^k_m}\le \psi_m\}$,
\begin{align}\label{Thetaest2}
	\left\{\Theta_-(k,m,I)\neq \emptyset\right\}\subset\Big\{\dot{\Theta}_{\psi_m}(m,a)\neq \emptyset\Big\},
\end{align}
where 
\[\dot{\Theta}_{\psi_m}(m,a):=\bigg\{x\in \z^2_{\rm e}:\dot{\Xi}_m(x)+\sum_{j=1}^{\dot{\Xi}_m(x)}h(x,j)\ge a\bigg\}\]
with $\dot{\Xi}_m(x):=\widetilde{\xi}(x,\psi_m)\wedge (\frac{15}{16}a-c_*m^{1-\kappa_1})$. Indeed, on $\{N_{T^k_m}\le \psi_m\}$, for any $x\in \z^2$,
\begin{align*}
	&\quad\,\left\{\widetilde{\xi}(x,N_{T^k_m})\le \frac{15}{16}a-c_* m^{1-\kappa_1},\xi(x,T^k_m)\ge a\right\}\\
	&\subset \bigcup_{n\in[0,\psi_m]}\left\{\widetilde{\xi}(x,n)\le \frac{15}{16}a-c_* m^{1-\kappa_1},\xi\big(x,N^{-1}_+(n)\big)\ge a\right\}\subset \left\{\dot{\Xi}_m(x)+\sum_{j=1}^{\dot{\Xi}_m(x)}h(x,j)\ge a\right\}.
\end{align*}
Let 
\[E_m:=\Big\{\#\big\{x\in \z^2_{\rm e}:\widetilde{\xi}(x,\psi_m)\ge \frac{15}{16}m-m^{4/5}\big\}\le  e^{16m^{1-2\kappa_1}}\Big\}.\] 
Combining \eqref{Thetaest1}, \eqref{Thetaest2}, and Proposition \ref{Mestimates}, it is enough to show that
\begin{align}\label{Theta_psi}
	\p\left(\dot{\Theta}_{\psi_m}(m,a)\neq \emptyset,E_m\right)<e^{-cm^{1-2\kappa_1}}.
\end{align}

We divide the points in $\dot{\Theta}_{\psi_m}(m,a)$ into two parts: 
\begin{align*}
\begin{aligned}	&\dot{\Theta}^{1}_{\psi_m}(m,a):=\left\{x\in \dot{\Theta}_{\psi_m}(m,a):\widetilde{\xi}(x,\psi_m)\ge {\frac{15}{16}a-c_*m^{3/4}}\right\};\\
		&\dot{\Theta}^{2}_{\psi_m}(m,a):=\left\{x\in \dot{\Theta}_{\psi_m}(m,a): \widetilde{\xi}(x,\psi_m)< {\frac{15}{16}a-c_*m^{3/4}}\right\}.
	\end{aligned}
\end{align*}
By Proposition \ref{condilaw1} (In the proof of \eqref{Thetakm2}, we shall use Proposition \ref{condilaw11} instead.), conditionally on $\widetilde{S}_{[0,\psi_m]}$, the events $\left\{\dot{\Xi}_m(x)+\sum_{j=1}^{\dot{\Xi}_m(x)}h(x,j)\ge a\right\}$, for $x\in \z^2_{\rm e}\cap \widetilde{S}_{[0,\psi_m]}$, are independent with probability {(recall that $\bar{p}(i,j):=p(i,j-i)$)}
\begin{align}\label{x_in_M}
	\sum_{j\ge a}\bar{p}(\dot{\Xi}_m(x),j).
\end{align}
In particular, it follows from \eqref{moderate1} that for any $\varepsilon>0$ and $\eta\in(0,1)$, there exists $M=M(\varepsilon,\eta)\in \z^+$ such that for all $m>M$ and $a\ge m-m^{4/5-\varepsilon}$, we have
\begin{itemize}
	\item for $x$ with $\widetilde{\xi}(x,N_{T^k_m})>{\frac{15}{16}a-c_*m^{3/4}}$,
	\begin{align*}
		\eqref{x_in_M}\le\sum_{j\ge \frac{1}{16}a+c_*m^{1-\kappa_1}}p(\dot{\Xi}_m(x),j)\le  \sum_{j\ge \frac{1}{16}a+c_*a^{1-\kappa_1}}p\left(\frac{15}{16}a,j\right)<e^{-\frac{1}{2\sigma^2}\frac{16}{15}\eta c^2_*m^{1-2{\kappa_1}}};
	\end{align*}
	\item for $x$ with $\widetilde{\xi}(x,N_{T^k_m})\le {\frac{15}{16}a-c_*m^{3/4}}$, 
	\begin{align*}
		\eqref{x_in_M}\le \sum_{j\ge {\frac{1}{16}a+c_*m^{3/4}}}p(\dot{\Xi}_m(x),j)\le \sum_{j\ge {\frac{1}{16}a+c_*a^{3/4}}}p\left(\frac{15}{16}a,j\right)<e^{-\frac{1}{2\sigma^2}\frac{16}{15}\eta c^2_*m^{1/2}}.
	\end{align*}
\end{itemize}
{Also note that there exists $M'=M'(\varepsilon)>0$ such that for all $m\ge M'$ and $a>m-m^{4/5-\varepsilon}$, 
\begin{align*}
    \#\big\{x\in \z^2_{\rm e}:\widetilde{\xi}(x,\psi_m)\ge \frac{15}{16}a-c_*m^{3/4}\big\}&\le \#\big\{x\in \z^2_{\rm e}:\widetilde{\xi}(x,\psi_m)\ge \frac{15}{16}m-m^{4/5}\big\}\\&\le  e^{16m^{1-2\kappa_1}}
\end{align*}
on $E_m$.}

Fixing $\eta\in[\frac{17}{18},1)$, since $\frac{1}{2\sigma^2}\frac{16}{15}c^2_*\ge 18$, we have $\frac{1}{2\sigma^2}\frac{16}{15}\eta c^2_*\ge 17$. It follows from the above analysis that for $m>M\vee M'$ and $a>m-m^{4/5-\varepsilon}$,
\begin{align*}
    &\p\left(\dot{\Theta}^1_{\psi_m}(m,a)\neq \emptyset\,\Big\vert\,\widetilde{S}_{[0,\psi_m]}\right)1_{E_m}\le 1-(1-e^{-17m^{1-2\kappa_1}})^{e^{16m^{1-2\kappa_1}}},\\
    &\p\left(\dot{\Theta}^2_{\psi_m}(m,a)\neq \emptyset\,\Big\vert\,\widetilde{S}_{[0,\psi_m]}\right)\le 1-(1-e^{-17m^{{1/2}}})^{\psi_m}.
\end{align*}	
Consequently, there exists $c_2=c_2(\varepsilon)>0$ such that
\begin{align}\label{Theta2}
	\p\left(\dot{\Theta}_{\psi_m}(m,a)\neq \emptyset\,\Big\vert\,\widetilde{S}_{[0,\psi_m]}\right)1_{E_m}<e^{-c_2m^{1-2\kappa_1}},
\end{align}
which implies \eqref{Theta_psi}.

The proof of 
\begin{align}\label{Theta+}
	\p\left(\Theta_+(k,m,I)\neq \emptyset,M^k_m\right)<e^{-cm^{1-2\kappa_1}}
\end{align}
is similar. Let 
\[\mathring{\Theta}_{\psi_m}(m,b):=\bigg\{x\in \z^2_{\rm e}:\widetilde{\xi}(x,\psi_m)> \frac{15}{16}b+c_* m^{1-\kappa_1},\,\mathring{\Xi}_m+\sum_{j=1}^{\mathring{\Xi}_m}h(x,j)< b\bigg\}\]
with $\mathring{\Xi}_m:=\frac{15}{16}b+c_*m^{1-\kappa_1}$. Then on $\{N_{T^k_m}\le \psi_m\}$,
\begin{align*}
	\left\{\Theta_+(k,m,I)\neq \emptyset\right\}\subset\Big\{\mathring{\Theta}_{\psi_m}(m,b)\neq \emptyset\Big\}.
\end{align*}
Thus it is sufficient to show that \begin{align*}
	\p\left(\mathring{\Theta}_{\psi_m}(m,b)\neq \emptyset,E_m\right)<e^{-cm^{1-2\kappa_1}},
\end{align*}
which is obtained using the same arguments as before and \eqref{moderate2}. We omit the details.

Combining \eqref{Theta-} and \eqref{Theta+}, we get \eqref{Thetakm1}.
\end{proof}
\begin{rem}\label{rem:1.3use}
    {As we shall see in the next subsection, it is quite crucial that the exponent $1-\kappa_1$ in the ``deviation term'' in the definition of 
    $\Theta(k,m,I)$ (see \eqref{eq:thetadef} and below) is strictly less than $2/3$. To this end, Proposition \ref{lower} plays a significant role, which provides extra information on the behaviors of $\widetilde{\xi}(\cdot, \psi_m)$. We emphasize that using only the conditional law of $\xi(\cdot, \psi_m)$ given $\widetilde{S}_{[0,\psi_m]}$ (as that in the proof of Proposition \ref{Theta}) without any extra information, one cannot obtain a bound with exponent less than $3/4$. 
    To be more precise, writing
    \[\dot{\Theta}^{(\delta)}_{\psi_m}(m,a):=\Big\{x\in \dot{\Theta}_{\psi_m}(m,a):\widetilde{\xi}(x,\psi_m)<\frac{15}{16}a-m^{3/4-\delta}\Big\},\quad\text{for }\delta>0,\]
    and using the same arguments as that in the proof of Proposition \ref{Theta}, we can show that
    \begin{align*}
        &\quad\,\p\Big\{\dot{\Theta}^{(\delta)}_{\psi_m}(m,a)\neq \emptyset\,\Big\vert\,\widetilde{S}_{[0,\psi_m]}\Big\}1_{\big\{\#\{x:\,\widetilde{\xi}(x,\psi_m)>\frac{15}{16}m-m^{3/4-\delta}\}>e^{m^{1/2}}\big\}}\\
        &\ge 1-(1-e^{-cm^{1/2-2\delta}})^{e^{m^{1/2}}}
        \rightarrow 1.
    \end{align*}
    Namely, when the path $\widetilde{S}_{[0,\psi_m]}$ exhibits atypical behaviors (from the perspective of Proposition \ref{Mestimates}), it is possible that $\dot{\Theta}^{(\delta)}_{\psi_m}(m,a)\neq \emptyset$. This suggests it is necessary to rule out certain atypical behaviors of $\widetilde{\xi}$ to get a finer bound.}
\end{rem}

\subsection{Proof of the upper bound in Theorem \ref{mainthm2}}\label{se:1.1ub} 
\red{In this subsection, we first wrap up the proof of the upper bound in Theorem \ref{mainthm2} into Proposition \ref{pro:finalreduction}, which is in turn implied (along with various other ingredients) by Propositions \ref{bigalpha} and \ref{smallalpha}, the main ``screening estimates'' mentioned in Section \ref{sec:sketch}, whose proofs are postponed till the next subsection.}

\red{We start with notation.} In the sequel, with $\kappa_1$ fixed as in \eqref{eq:kappa1}, we also fix
\begin{equation}\label{eq:kappa2}
\red{\mbox{$\kappa_2\in(1/3,\kappa_1)$ and $\delta=\kappa_2/N$ with large $N$ such that $\kappa_1>\kappa_2+2\delta>1/3+4\delta$}}
\end{equation}
 \red{and treat $\kappa_2$ and $\delta$ as constants.} Let \begin{equation}
     \mbox{$\Lambda:=\{\delta,2\delta,\ldots\}$ and $\Lambda_0:=\Lambda\cap [0,\kappa_2]$.}
 \end{equation} 
\red{We also write
\begin{equation}\label{eq:Kappa}
    \kappa:=\kappa_2-2\delta.
\end{equation}
By \eqref{eq:kappa2}, $\kappa>1/3$.}

\red{We now divide the upper bound of in Theorem \ref{mainthm2} into different cases according to the location of the (possible) most favorite sites. In order to benefit from the local time decompositions from Section \ref{Preliminaries}, one must find a pairing (i.e., domino tiling) of $\mathbb{Z}^2$ such that no two such sites share one domino. Luckily, as we only need to treat at most four sites simultaneously, it is always possible to find such a pairing.} Recall the pairing $\X$ of $\mathbb{Z}^2$ introduced \red{in \eqref{eq:Xdef}} and define $\Pi^k_m:=\bigcap_{i,j=1}^k\left\{(L^i_m,L^j_m)\notin\X\right\}$. We \red{also} introduce two more pairings. Write 
\begin{equation}\label{eq:Ydef}
\Y:= \{(x,x+{\bf e}_1): x=(2x_1,x_2), x_1,x_2\in\mathbb{Z}\}    
\end{equation}
and 
\begin{equation}\label{eq:Yprimedef}
\Y':= \{(x,x+{\bf e}_1): x=(2x_1+1,x_2), x_1,x_2\in\mathbb{Z}\}.    
\end{equation}
Let
$\Upsilon_m=\bigcap_{i,j=1}^4\left\{(L^i_m,L^j_m)\notin\Y\right\}$ and define $\Upsilon'_m$ similarly with $\Y$ replaced by $\Y'$.

\begin{pro}\label{pro:finalreduction}
\red{With $\kappa$ fixed in \eqref{eq:Kappa},} uniformly in $m\ge 1,$
	\begin{align}\label{eq:Piub}
		\p\left(M^4_m,\Pi^4_m\right)\lesssim \frac{1}{m^{3\kappa}}.
	\end{align}
Similarly, \begin{align}\label{eq:Upub}
		\p\left(M^4_m,\Upsilon_m\cup \Upsilon'_m \right)\lesssim \frac{1}{m^{3\kappa}}.
	\end{align}
\end{pro}
\begin{proof}[Proof of the upper bound in Theorem \ref{mainthm2} assuming Proposition \ref{pro:finalreduction}]
	In this proof, we write $\X_1=\X$ and $\Pi^k_{m,1}=\Pi^k_m$ to emphasize the role of ``${\bf e}_1$'' in the two definitions. For $j=2,3,4$, we define $\X_j$ (resp. $\Pi^k_{m,j}$) through replacing ${\bf e}_1$ by ${\bf e}_j$ in the definition of $\X_1$ (resp. $\Pi_{m,1}$). Then using the same arguments, we can show that for any $j=1,2,3,4$,
	\begin{align*}
		\p\left(M^4_m,\Pi^4_{m,j}\right)\lesssim \frac{1}{m^{3\kappa}}.
	\end{align*}
	
 By an elementary combinatorial argument, given any four distinct points $x_1,x_2,x_3,x_4\in\mathbb{Z}^2$, one can find at least one pairing of $\mathbb{Z}^2$ among $\X_j$, $j=1,2,3,4$, $\Y$ and $\Y'$, such that each pair in that pairing contains at most one of $x_j$, $j=1,2,3,4$. Hence, one can readily check that 
	\[\p(M^4_m)\le \sum_{j=1}^4 \p\left(M^4_m,\Pi^4_{m,j}\right)+\p\left(M^4_m,\Upsilon_m\cup \Upsilon'_m \right).\]
	We finish the proof by applying Proposition \ref{pro:finalreduction} and Borel-Cantelli.
\end{proof}

We now turn to the proof of Proposition \ref{pro:finalreduction}. We only prove \eqref{eq:Piub} as the proof of \eqref{eq:Upub} is almost identical. To control the probability of $M_m^4$, as discussed in Section \ref{Introduction}, we define the following set of near-favorite sites of $S$ at time $T_m^k$: for $\alpha\in(0,1]$, let
\begin{align}\label{eq:nearfavorite}
	\mathcal{M}^k(m,\alpha)&:=\left\{x\in\z^2:\x(x)\in\X\setminus\{\x(L^1_m),\ldots,\x(L^k_m)\},\,{\xi(x,T^k_m)\in(m-m^\alpha,m)}\right\}.
\end{align}
As we are going to see (in the proof of Lemma \ref{negligible}), if the event $M_m^k$ occurs, with very high probability the $(k+1)$-th favorite site of local time $m$ has to be one of the near-favorite points at time $T_m^k$. Hence the analysis is reduced to that of the size and location of near-favorite sites with certain local time gap.

To this end, we need the following propositions, which, with two rounds of ``screening'' on the size of the set $\mathcal{M}^k(\cdot,\cdot)$, provide (in \eqref{malpha2}) a probability bound on the appearance of near-favorite sites which is sufficient to rule out the i.o. occurrence of $M_m^4$.

Recall the definition of $\Theta(k,m,I)$ and $\Theta'(k,m,I)$ above Proposition \ref{Theta} and 
 write $$\Theta^k_m:=\Theta(k,m,[m-m^{\kappa_1}+1,m))\cup \Theta'(k,m,[m-m^{\kappa_1}+1,m))$$ for short.

\begin{pro}\label{bigalpha}
     Fixing $\bar{c}=\bar{c}(\varepsilon)>1$ satisfying \eqref{pbound1} below, let $c_0=\log (2\bar{c})$. Then for any $\varepsilon,C>0$, there exists $c=c(\varepsilon,C)>0$ such that for any $m,k\ge 1$
 and $\alpha\in[\kappa_1,4/5-\varepsilon]$, 
 \begin{align}\label{malpha1}
		\p\left(\#\mathcal{M}^k(m,\alpha)>Ce^{c_0m^{\alpha-\kappa_1}}(\log m)^2,M^k_m,\Pi^k_m\right)<e^{-c(\log m)^2}.
	\end{align}
\end{pro}
\red{Note that the statement includes the endpoint $\alpha=\kappa_1$, and the constant $c$ in \eqref{malpha1} 
is uniform over $\alpha\in[\kappa_1,4/5-\varepsilon]$.}
\begin{pro}\label{smallalpha}
    Uniformly in $m,k\ge 1$ and $\alpha\in (0,\kappa_1)$,
	\begin{align}\label{malpha2}
	\begin{aligned}
		\p\Big\{\mathcal{M}^k(m,\alpha)\neq \emptyset,\,&\Theta^k_m=\emptyset\,\Big\vert\,\#\mathcal{M}^k(m,{\kappa_1})\le (\log m)^2,\\
		&M^k_m,\,\widetilde{S}_{[0,N_{T^k_m}]},\,\left(L^j_m:j=1,\ldots,k\right)\Big\}\lesssim \frac{(\log m)^2}{m^{\kappa_1-\alpha}}.		
	\end{aligned}
	\end{align}
\end{pro}

The proof of Propositions \ref{bigalpha} and \ref{smallalpha} is postponed to the next subsection. 

\begin{proof}[Proof of Proposition \ref{pro:finalreduction}]
Recall the notations defined in the beginning of Section \ref{se:1.1ub}. {Note that
\begin{align*}
    \quad\,\p\big(M^4_m,\abs{L^k_m-L^{k+1}_m}>e^m\text{ for some }k=1,2,3\big)\le \p\big(M^4_m,T^4_m>e^m\big)\overset{\eqref{Tkm1}}{<}e^{-cm}
\end{align*}
for some universal $c>0$.} Hence it suffices to show that uniformly for $(\alpha_k:k=1,2,3)\in {(\Lambda\cap [0,1])^3}$ and $m\ge 1$,
	\begin{align*}
		\p\left(M^4_m,\Pi^4_m,\,e^{m^{\alpha_k-\delta}}/3\le \abs{L^k_m-L^{k+1}_m}\le e^{m^{\alpha_k}}\text{ for }k=1,2,3\right)\lesssim \frac{1}{m^{3\kappa}}.
	\end{align*}
Here the ``$/3$'' in the event is only used to ensure that when $\alpha_k=\delta$, $e^{m^{\alpha_k-\delta}}/3=e/3<1$, so that all the possibilities of $\abs{L^k_m-L^{k+1}_m}$ are included.

We begin with a lemma, which states that the event that $m-\xi(L^{k+1}_m,T^k_m)$ is far greater than $\log \abs{L^k_m-L^{k+1}_m}$ is of negligible probability. The proof is postponed to the end of this subsection.
\begin{lem}\label{negligible}
	There exists $c=c(\delta)>0$ such that for any $m,k\ge 1$ and $\alpha\in \Lambda_0$,
	\[\p\Big( M^{k+1}_m,\Pi^{k+1}_m,\abs{L^k_m-L^{k+1}_m}\le e^{m^{\alpha}},\xi(L^{k+1}_m,T^k_m)< m-m^{\alpha+\delta}\Big) \le e^{-c(\log m)^2}.\]
\end{lem}
Recall the representation $M_m^k=U^2_m\cap\cdots\cap U^k_m$ in \eqref{representation}.
Based on this lemma, Proposition \ref{Theta}, and Proposition \ref{bigalpha}, we need only to prove that for any $k=1,2,3$:
\begin{itemize}
	\item In the case $\alpha_k\in (\Lambda\cap [0,1])\setminus \Lambda_0$,
	\begin{align}\label{condi}
	\begin{aligned}
		\p\Big(&U^{k+1}_m,\Pi^{k+1}_m,e^{m^{\alpha_k-\delta}}/3\le \abs{L^k_m-L^{k+1}_m}\le e^{m^{\alpha_k}}\,\Big\vert\,
		\\&M^k_m,\Pi^k_m,\,e^{m^{\alpha_j-\delta}}/3\le \abs{L^j_m-L^{j+1}_m}\le e^{m^{\alpha_j}}\text{ for }j\le k-1\Big)\lesssim \frac{1}{m^{\kappa}};
	\end{aligned}
	\end{align}
	\item In the case $\alpha_k\in \Lambda_0$,
	\begin{align}\label{condi2}
	\begin{aligned}
		\p\Big(& U^{k+1}_m,\Pi^{k+1}_m,e^{m^{\alpha_k-\delta}}/3\le \abs{L^k_m-L^{k+1}_m}\le e^{m^{\alpha_k}},\,Q^k_m\,\Big\vert
		\\ &M^k_m,\Pi^k_m,\,e^{m^{\alpha_j-\delta}}/3\le \abs{L^j_m-L^{j+1}_m}\le e^{m^{\alpha_j}}\text{ for }j\le k-1\Big)\lesssim \frac{1}{m^{\kappa}},	
	\end{aligned}
	\end{align}
	where $Q^k_m:=\big\{\xi(L^{k+1}_m,T^k_m)\ge m-m^{\alpha_k+\delta},\,\Theta^k_m=\emptyset,\,\#M^k(m,m^{\kappa_1})\le (\log m)^2\big\}$.
\end{itemize}

	In the first case, we have
	\[\left\{U^{k+1}_m,e^{m^{\alpha_k-\delta}}/3\le \abs{L^k_m-L^{k+1}_m}\le e^{m^{\alpha_k}}\right\}\subset\left\{\h_0>\h_{D(0,e^{m^{\kappa_2}}/3)^c}\right\}\circ \theta_{T^k_m}.\]
	Then \eqref{condi} follows from the strong Markov property and \eqref{hitting}.
	
	In the second case, we note that
	\begin{align*}
		&\quad\,\left\{M^{k+1}_m,\Pi^{k+1}_m,\,e^{m^{\alpha_k-\delta}}/3\le \abs{L^k_m-L^{k+1}_m}\le e^{m^{\alpha_k}},\,Q^k_m\right\}\\
		&\subset \Big\{\h_0>\h_{D(0,e^{m^{\alpha_k-\delta}}/3)^c}\Big\}\circ \theta_{T^k_m}\\
  &\quad\,\cap\left\{\mathcal{M}^k(m,{\alpha_k+\delta})\neq \emptyset,\,\Theta^k_m=\emptyset,\,\#\mathcal{M}^k(m,{\kappa_1})\le (\log m)^2\right\}.
	\end{align*}
An application of the strong Markov property gives
\begin{align*}
	\text{LHS of }\eqref{condi2}\lesssim \frac{1}{m^{\alpha_k-\delta}}\p \Big(&\mathcal{M}^k(m,{\alpha_k+\delta})\neq \emptyset,\,\Theta^k_m=\emptyset,\,\#\mathcal{M}^k(m,{\kappa_1})\le (\log m)^2\,\Big\vert\\&M^k_m,\,\Pi^k_m,\,e^{m^{\alpha_j-\delta}}/3\le \abs{L^j_m-L^{j+1}_m}\le e^{m^{\alpha_j}}\text{ for }j\le k-1\Big).
\end{align*}
	Further using Proposition \ref{smallalpha} gives \eqref{condi2}.
\end{proof}

\begin{proof}[Proof of Lemma \ref{negligible}]
	For $x\in \z^2$ and $n,l\ge 0$, we define
	\[T(x,n,l):=\inf\{j\ge l:\xi(x,j)-\xi(x,l)=n\},\]
	i.e., the $n$-th hitting time of $x$ after time $l$. In this proof, we will frequently consider the events
	\[A(\sigma,\beta):=\left\{T(S_\sigma, m^\beta, \sigma)<\h_{L^k_m}(\sigma),\,\sigma<\infty\right\}\]
	for $\beta>0$ and various random times $\sigma\ge T^k_m$.
	
	Fixing $m,k,\alpha$, let $\beta_0=\alpha+\delta$, $\beta_j=\alpha+\delta+j(\kappa_1-\alpha-\delta)$ for $j=1,\ldots,[\frac{1-\alpha-\delta}{\kappa_1-\alpha-\delta}]+1$. We partition the event
	\[\left\{M^{k+1}_m,\Pi^{k+1}_m,\abs{L^k_m-L^{k+1}_m}\le e^{m^{\alpha}},\xi(L^{k+1}_m,T^k_m)< m-m^{\alpha+\delta}\right\}\]
	into disjoint events $B_j$ ($j=1,\ldots,[\frac{1-\alpha-\delta}{\kappa_1-\alpha-\delta}]+1$), where
	\begin{align*}
		&B_j=\left\{M^{k+1}_m,\Pi^{k+1}_m,\abs{L^k_m-L^{k+1}_m}\le e^{m^{\alpha}},m-m^{\beta_j}\le \xi(L^{k+1}_m,T^k_m)< m-m^{\beta_{j-1}}\right\}\\
		&\quad\,\subset \Big\{M^k_m,\Pi^k_m,A\big(\h_{L^{k+1}_m}(T^k_m),\beta_{j-1}\big),\abs{L^k_m-L^{k+1}_m}\le e^{m^{\alpha}},\xi(L^{k+1}_m,T^k_m)\ge m-m^{\beta_j}\Big\}\\
  &\quad\,\quad\,=:B'_j.
	\end{align*}
	We now show that  $\h_{L^{k+1}_m}(T^k_m)$ is contained in a collection of stopping times, which enables us to use the strong Markov property to calculate the probability. 
	

	Let us define the stopping times we are concerned with. For $j=1,\ldots, [\frac{1-\alpha-\delta}{\kappa_1-\alpha-\delta}]+1$, let
	 \begin{align*}
	 	F_j:=\left\{x\in \z^2:\abs{x-L^k_m}\le e^{m^{\alpha}},\,\xi(x,T^k_m)\ge m-m^{\beta_j}\right\}.
	 \end{align*}
	 We define $\sigma^j_i$ as follows:
	 \begin{align*}
		\sigma^j_1&:=\inf\left\{n\ge T^k_m:S_n\in F_j\right\},\\
		\sigma^j_{i+1}&:=\inf\left\{n\ge T^k_m:S_n\in F_j,\,S_n\notin\{S_{\sigma_1},\ldots,S_{\sigma_i}\}\right\}.
	\end{align*}
	Namely, $\sigma^j_i$ are the times after $T^k_m$ when the process hits a new point in $F_j$. Since $L^{k+1}_m\in F_j$ on $B'_j$, we have $\h_{L^{k+1}_m}(T^k_m)\in  \{\sigma^j_i:i\ge 1\}$ on $B'_j$. Consequently,
	\[B'_j\subset \bigcup_{i=1}^\infty\left\{A(\sigma^j_i,\beta_{j-1}),M^k_m,\Pi^k_m \right\}.\]

	Note that a.s., $\sigma^j_i<\infty$ for $i=1,\ldots, \# F_j$ and $\sigma^j_i=\infty$ for $i>\# F_j$. It follows from Proposition \ref{bigalpha} that for $j\in \big[1, [\frac{1-\alpha-\delta}{\kappa_1-\alpha-\delta}]+1\big]$ with $\beta_j\le \frac{7}{10}$,
	\begin{align*}
		&\quad\,\p\left(\bigcup_{i=e^{c_0m^{\beta_j-\kappa_1}}(\log m)^2+1}^\infty\left\{A(\sigma^j_i,\beta_{j-1}),M^k_m,\Pi^k_m \right\}\right)\\
  &\le \p\left(\# F_j> e^{c_0m^{\beta_j-\kappa_1}}(\log m)^2,M^k_m,\Pi^k_m\right)\\
	&\le \p\left(\#\mathcal{M}^k(m,m^{\beta_j})>e^{c_0m^{\beta_j-\kappa_1}}(\log m)^2,M^k_m,\Pi^k_m \right)\overset{\eqref{malpha1}}{<}e^{-c(\log m)^2}
	\end{align*}
for some universal constant $c>0$. 
For $j\in \big[1, [\frac{1-\alpha-\delta}{\kappa_1-\alpha-\delta}]+1\big]$ with $\beta_j>\frac{7}{10}$, it holds trivially that $\# F_j\le C e^{2m^\alpha}\le Ce^{2m^{\kappa_2}}<e^{c_0m^{\beta_j-\kappa_1}}(\log m)^2$ for $m$ large, which implies that 
\[\p\left(
\bigcup_{i=e^{c_0m^{\beta_j-\kappa_1}}(\log m)^2+1}^\infty\left\{A(\sigma^j_i,\beta_{j-1}),M^k_m,\Pi^k_m\right\}\right)\le e^{-c'(\log m)^2}\]
for some universal constant $c'>0$ and all $m\ge 1$, $j\in \big[1, [\frac{1-\alpha-\delta}{\kappa_1-\alpha-\delta}]+1\big]$ with $\beta_j>\frac{7}{10}$.

On the other hand, {let $g(y):=\p\left\{{T(0,m^{\beta_{j-1}},0)}<\h_y\right\}$ for $y\in \z^2$}. Then an application of the strong Markov property yields that for any $j=1,\ldots,[\frac{1-\alpha-\delta}{\kappa_1-\alpha-\delta}]+1$ and $i\ge 1$,
\begin{align*}
	\p\left(A(\sigma^j_i,\beta_{j-1})\right)&=\e\left[{g\big(L^k_m-S_{\sigma^j_i}\big)}1_{\big\{\sigma^j_i<\infty\big\}}\right]\\
	& \overset{(*)}{\le} \left(1-\frac{c_1}{m^{\alpha}}\right)^{m^{\beta_{j-1}}}\le \exp\left\{-c_1m^{\beta_{j-1}-\alpha}\right\}=\exp\left\{-c_1m^{\beta_{j}-\kappa_1+\delta}\right\},
\end{align*}
where $(*)$ follows from the fact that $\abs{L^k_m-S_{\sigma^j_i}}\le e^{m^\alpha}$.

Therefore for any $j=1,\ldots, [\frac{1-\alpha-\delta}{\kappa_1-\alpha-\delta}]+1$,
\begin{align*}
	\p\left(B'_j\right)&\le e^{-(c\wedge c')(\log m)^2}+\sum_{i=1}^{e^{c_0m^{\beta_j-\kappa_1}}(\log m)^2}\p\left(A(\sigma^j_i,\beta_{j-1})\right)\le e^{-c_2(\log m)^2}
\end{align*}
for some constant $c_2=c_2(\delta)>0$.
Note that $[\frac{1-\alpha-\delta}{\kappa_1-\alpha-\delta}]+1<(\kappa_1-\kappa_2-\delta)^{-1}+1<1/\delta+1$. Hence there exists $c_3=c_3(\delta)$ such that
\begin{align*}
	&\quad\,\p\left\{M^{k+1}_m,\Pi^{k+1}_m,\abs{L^k_m-L^{k+1}_m}\le e^{m^{\alpha}},\xi(L^{k+1}_m,T^k_m)< m-m^{\alpha+\delta}\right\}\\
	&\le \sum_{j=1}^{[\frac{1-\alpha-\delta}{\kappa_1-\alpha-\delta}]+1}\p\left(B'_j\right)\le e^{-c_3(\log m)^2}.
\end{align*}
Then the proof is completed. 
\end{proof}

\subsection{Proof of Propositions \ref{bigalpha} and \ref{smallalpha}}
This subsection is dedicated to the proof of Propositions \ref{bigalpha} and \ref{smallalpha}. We refer readers to Section \ref{tysec} for two variants of the urn model which serve as toy models of the arguments in this subsection. More precisely, the proof of \eqref{Me1} and \eqref{Me} will be accomplished using the same ideas as that in the proof of Lemmas \ref{ty2} and \ref{ty1}, respectively.

To ease the analysis, we define
\begin{align}\label{MeMo}
	\begin{aligned}
		\mathcal{M}_{\rm e}^k(m,\alpha)&:=\Big\{x\in \z^2_{\rm e}:\x(x)\in\X\setminus\{\x(L^1_m),\ldots,\x(L^k_m)\},\\
		&\quad\quad\,\xi(x,T^k_m)=\xi^+(\x(x),T^k_m)\in(m-m^\alpha,m)\Big\},\\
		\mathcal{M}_{\rm o}^k(m,\alpha)&:=\Big\{x\in \z^2_{\rm o}:\x(x)\in\X\setminus\{\x(L^1_m),\ldots,\x(L^k_m)\},\\
		&\quad\quad\,\xi(x,T^k_m)=\xi^+(\x(x),T^k_m)\in(m-m^\alpha,m)\Big\},		
	\end{aligned}
	\end{align}
 {which are related to $\mathcal{M}^k(m,\alpha)$ via:
 \begin{align}
     &\big\{\mathcal{M}^k(m,\alpha)\neq \emptyset\big\}=\big\{\mathcal{M}^k_{\rm e}(m,\alpha)\cup \mathcal{M}^k_{\rm o}(m,\alpha)\neq \emptyset\big\},\quad\text{and}\label{relation1}\\
     &\#\mathcal{M}^k(m,\alpha)\le 2\big[\#\mathcal{M}_{\rm e}^k(m,\alpha)+\# \mathcal{M}_{\rm o}^k(m,\alpha)\big].\label{relation2}
 \end{align}}

\begin{proof}[Proof of Proposition \ref{bigalpha}]  
Fixing $\varepsilon,C>0$, we simply write 
\[\rho=\rho(C,\alpha,m):=Ce^{c_0m^{\alpha-\kappa_1}}(\log m)^2.\] 
Recall \eqref{relation2}. 
We claim that it is enough to prove that there exists $c=c(\varepsilon,C)>0$ such that for all $m,k\ge 1$ and $\alpha\in[\kappa_1,4/5-\varepsilon]$,
\begin{align}\label{Me1}
	\p\left(\#\mathcal{M}_{\rm e}^k(m,\alpha)>\rho,M^k_m,\Pi^k_m\right)<e^{-c(\log m)^2}.
\end{align}
Indeed, we have $\#\mathcal{M}_{\rm o}^k(m,\alpha)=\#\mathcal{M}_{\rm e}^k(m,\alpha)\circ \theta_1$, $M^k_m=M^k_m\circ\theta_1$, and $\Pi^k_m=\Pi^k_m\circ\theta_1$ on $\{\xi(0,T^k_m)<m-m^\alpha\}$. Note that by Lemmas \ref{erdos_taylor} and \ref{Tkm},
\begin{align*}
	\p\left(\xi(0,T^k_m)\ge m-m^\alpha,M^k_m\right)&\le \p\left(\xi(0,e^{2m^{1/2}})\ge m-m^\alpha\right)+\p\left(T^k_m> e^{2m^{1/2}},M^k_m\right)\\&<e^{-c_1m^{1/2}}
\end{align*}
for some $c_1>0$. It follows that
\begin{align*}
	&\quad\,\p\left(\#\mathcal{M}_{\rm o}^k(m,\alpha)>\rho,M^k_m,\Pi^k_m\right)\\
	&\le \p\left(\#\mathcal{M}_{\rm e}^k(m,\alpha)>\rho,M^k_m,\Pi^k_m\right)+\p\left(\xi(0,T^k_m)\ge m-m^\alpha,M^k_m\right)\\
	&<e^{-c_2(\log m)^2}
\end{align*}
for some $c_2>0$, which leads to \eqref{malpha1}.

We now prove \eqref{Me1}. 
To estimate the probability, we shall first restate the above event in terms of the local times $\xi\big(\x,T^k_m\big)$ for $\x\in\X$. To this end, we define the following three types $V^{(i)}=V^{(i)}(\alpha,k,m)$ of $\x$:
\begin{align*}
	V^{(1)}&=\left\{\x\in\X:\xi^+(\x,T^k_m)=m,\xi^-(\x,T^k_m)<m\right\},\\
	V^{(2)}&=\left\{\x=(x,x+{\bf e}_1)\in\X:\xi^+(\x,T^k_m)=\xi(x,T^k_m)\in(m-m^\alpha,m)\right\},\\
	V^{(3)}&=\left\{\x=(x,x+{\bf e}_1)\in\X:\xi(x,T^k_m)\le m-m^\alpha\text{ or }\xi(x,T^k_m)<\xi(x+{\bf e}_1,T^k_m)<m\right\}.
\end{align*}
Then we have
\begin{align*}
 	\big\{\#\mathcal{M}_{\rm e}^k(m,\alpha)>\rho,M^k_m,\Pi^k_m\big\}=\left\{\#V^{(2)}>\rho,D^k_m\right\},
 \end{align*}
  where
  \begin{align*}
  	D^k_m:=\left\{\#V^{(1)}=k,V^{(1)}\cup V^{(2)}\cup V^{(3)}=\X\right\}.
  \end{align*}
Thus 
\begin{align}\label{restate}
	\mbox{ LHS of }\eqref{Me1}=\p\left(\# V^{(2)}>\rho,D^k_m\right).
\end{align}

Next, we shall further divide the elements in $\# V^{(2)}$ and estimate the cardinality of each part. Let
\begin{equation}\label{eq:Ildef}
I_\ell=[a_\ell,b_\ell):=[m-\ell m^{\kappa_1},m-(\ell-1)m^{\kappa_1})\text{ for }\ell=1,\ldots,[m^{\alpha-\kappa_1}]+1,    
\end{equation}
and divide the elements in $V^{(2)}$ into $([m^{\alpha-\kappa_1}]+1)$ disjoint parts: 
\begin{align*}
	V^{(2)}=\bigcup_{j=1}^{[m^{\alpha-\kappa_1}]+1}V^{(2)}(\ell)\text{ with }V^{(2)}(\ell):=\left\{\x\in V^{(2)}:\xi^+(\x,T^k_m)\in I_\ell\right\}.
\end{align*}

\begin{lem}\label{Card}
Let $\rho_\ell=\rho_\ell(C,m):=Ce^{c_0(\ell-1)}(\log m)^2$. There exists $c=c(\varepsilon,C)>0$ such that for all $m,k\ge 1$, $\alpha\in[\kappa_1,4/5-\varepsilon]$, and $\ell=1,\ldots,[m^{\alpha-\kappa_1}]+1$,
\begin{align}\label{Card1}
	\p\left(\# V^{(2)}(\ell)>\rho_\ell,D^k_m\right)<e^{-c (\log m)^2}.
\end{align}
\end{lem}
The proof of Lemma \ref{Card} is postponed. Back to \eqref{restate}, using Lemma \ref{Card} and the equality that $\# V^{(2)}=\sum_{\ell=1}^{[m^{\alpha-\kappa_1}]+1} \# V^{(2)}(\ell)$, we have with $\rho_\ell'=\rho_\ell\big((1-e^{-c_0})C,m\big)$,
$$
	\mbox{ LHS of }\eqref{Me1}\le \sum_{\ell=1}^{[m^{\alpha-\kappa_1}]+1}\p\left\{\# V^{(2)}(\ell)>\rho_\ell' ,D^k_m\right\}<\sum_{\ell=1}^{[m^{\alpha-\kappa_1}]+1} e^{-c_1(\log m)^2}<e^{-c_2 (\log m)^2}
$$
for some $c_1=c_1(\varepsilon,C),c_2=c_2(\varepsilon,C)>0$ and all $m,k\ge 1$ and $\alpha\in[\kappa_1,4/5-\varepsilon]$.
That concludes the proof.
\end{proof}

Now we prove Lemma \ref{Card}. The main idea is as follows. We simply write $\Theta_\ell:=\Theta(k,m,I_{\ell-1}\cup I_\ell)$ which equals $\emptyset$ with high probability by Proposition \ref{Theta}. Note that on $\{\Theta_\ell=\emptyset\}$,
\begin{align}\label{V2}
\begin{aligned}
	&\quad\,V^{(2)}(\ell)\cup V^{(2)}(\ell-1)\\
	&\subset\left\{x\in \z^2_{\rm e}:\widetilde{\xi}(x,N_{T^k_m}){=\widetilde{\xi}^+(\x(x),N_{T^k_m})}\in \left[\frac{15}{16}a_{\ell}-c_* m^{1-\kappa_1},\frac{15}{16}b_{\ell-1}+c_*m^{1-\kappa_1}\right)\right\}\\
 &\quad\,=:\widetilde{V}^{(2)}(\ell)	
\end{aligned}
\end{align}
for each $\ell=2,\ldots,[m^{\alpha-\kappa_1}]+1$. It can be shown that conditionally on $\widetilde{\xi}(\cdot,N_{T^k_m})$ and $\Theta_\ell=\emptyset$, independently for each $x\in \widetilde{V}^{(2)}(\ell)$, the probabilities of $\{x\in V^{(2)}(\ell)\}$ and $\{x\in V^{(2)}(\ell-1)\}$ are comparable, which together with \eqref{V2} and the large deviation principle implies that with high probability,  $\#V^{(2)}(\ell)$ is bounded by a constant multiple of $\#V^{(2)}(\ell-1)$ for $\ell=2,\ldots,[m^{\alpha-\kappa_1}]+1$. This gives \eqref{Card1} in the case $\ell\ge 2$ (assuming the case $\ell=1$). 

The proof for the case $\ell=1$ is similar yet more involved due to the fact that $\xi(x,T^k_m)\le m$ on $M^k_m$ \red{which prevents us from making the same comparison at time $T^k_m$.} To overcome this, we enumerate all the possibilities of the path $\widetilde{S}_{[0,N_{T^k_m}]}$ and compare the probabilities at time $N^{-1}_+(n)$ on $\{N_{T^k_m}=n\}$. 

Let us start with the following property of $\bar{p}(i,j)$ which is immediate from Lemma \ref{stirling01}. Recall the definition of $I_l$, $a_l$ and $b_l$ from \eqref{eq:Ildef}.
\begin{lem}\label{pbound}
	For any $\varepsilon>0$, there exists $\bar{c}=\bar{c}(\varepsilon)>1$ such that for all $\alpha\in[\kappa_1,4/5-\varepsilon]$, $\ell\in \big[1,[m^{\alpha-\kappa_1}]+1\big]$, $i\in \left[\frac{15}{16}a_{\ell}-c_* m^{1-\kappa_1},\frac{15}{16}b_{\ell-1}+c_*m^{1-\kappa_1}\right]$, and {$j_1,j_2\in I_\ell=[a_\ell,b_{\ell})$,
	\begin{align}\label{pbound1}
		\bar{c}^{-1}\bar{p}(i,j_2)\le \bar{p}(i,j_1)\le \bar{c}\bar{p}(i,j_2).
	\end{align}}
\end{lem}
\red{Note that here (and in a similar argument below, see \eqref{eq:seesaw}), if we choose a larger range for the possible values of $j_\cdot$, then such an up-to-constants bound will have to fail.}
\begin{proof}[Proof of Lemma \ref{Card}]
First, we assume \eqref{Card1} for the case $\ell=1$, i.e.,
there exists $c=c(\varepsilon,C)>0$ such that for all $m,k\ge 1$ and $\alpha\in[\kappa_1,4/5-\varepsilon]$, 
\begin{align}\label{Card2}
	\p\left(\# V^{(2)}(1)>C(\log m)^2,D^k_m\right)<e^{-c (\log m)^2}.
\end{align}
We shall prove \eqref{Card1} under the assumption \eqref{Card2}.

 The proof uses the same ideas as that in the proof of Lemma \ref{ty2}. We claim that there exists $c>0$ such that for any $\ell=2,\ldots,[m^{\alpha-\kappa_1}]+1$, 
 \begin{align}\label{large_deviation1}
 \begin{aligned}
 	&\p\left(\# V^{(2)}(\ell)>2\bar{c}\# V^{(2)}(\ell-1),D^k_m,\Theta_\ell=\emptyset\,\Big\vert\, \#V^{(2)}(\ell)+\#V^{(2)}(\ell-1)\right)\\
 	&\le e^{-c\big[\#V^{(2)}(\ell)+\#V^{(2)}(\ell-1)\big]}. 
 \end{aligned}
 \end{align}
 Indeed, observe that conditionally on $V^{(2)}(\ell)\cup V^{(2)}(\ell-1)$, $D^k_m$, and $\Theta_\ell=\emptyset$, \eqref{V2} holds and the events $\{x\in V^{(2)}(\ell)\}$, for $x\in V^{(2)}(\ell)\cup V^{(2)}(\ell-1)$, are independent with probability
 \begin{align*}
	\frac{\sum_{j\in I_\ell}\bar{p}(\widetilde{\xi}(x,\abs{\eta}),j)}{\sum_{j\in I_\ell}\bar{p}(\widetilde{\xi}(x,\abs{\eta}),j)+\sum_{j\in I_{\ell-1}}\bar{p}(\widetilde{\xi}(x,\abs{\eta}),j)}\overset{\eqref{pbound1}}{\le} \frac{\bar{c}}{1+\bar{c}}.
\end{align*}
Therefore under the conditioning, $\#V^{(2)}(\ell)$ is stochastically bounded by a binomial random variable with parameter $\big(\#V^{(2)}(\ell)+\#V^{(2)}(\ell-1),\frac{\bar{c}}{1+\bar{c}}\big)$. Then \eqref{large_deviation1} follows from the large deviation principle.

Recall the definition of $\rho_\ell$ in the statement of Lemma \ref{Card} {and that $c_0=\log(2\bar{c})$. In particular, we have $\rho_\ell=2\bar{c}\rho_{\ell-1}$.} Fixing $C>0$, we write \[q_\ell:=\p\big(\#V^{(2)}(\ell)>\rho_\ell,D^k_m\big).\]
Then for any $\ell=2,\ldots,[m^{\alpha-\kappa_1}]+1$, 
\begin{align*}
	q_\ell&\le \p\left(\#V^{(2)}(\ell)>\rho_\ell,\#V^{(2)}(\ell-1)>\rho_{\ell-1},D^k_m\right)\\
	&\quad\,+\p\left(\#V^{(2)}(\ell)>\rho_\ell,\#V^{(2)}(\ell-1)\le \rho_{\ell-1},D^k_m,\Theta_\ell=\emptyset\right)+\p\left(\Theta_\ell\neq \emptyset\right)\\
	&\le q_{\ell-1}+\p\left(\Theta_\ell\neq \emptyset\right)\\
	&\quad\,+\p\left(\#V^{(2)}(\ell)>2\bar{c}\#V^{(2)}(\ell-1),\#V^{(2)}(\ell)+\#V^{(2)}(\ell-1)>\rho_{\ell},D^k_m,\Theta_\ell=\emptyset\right)\\
	&\le q_{\ell-1}+e^{-c_1m^{1-2\kappa_1}}+e^{-c \rho_{\ell}}
\end{align*}
for some $c_1=c_1(\varepsilon)>0$ by Proposition \ref{Theta} and \eqref{large_deviation1}. Further using \eqref{Card2}, we get
\[q_\ell\le q_1+\sum_{j=2}^{\ell}(e^{-c\rho_j}+e^{-c_1m^{1-2\kappa_1}})<e^{-c_2(\log m)^2}\]
for some $c_2=c_2(\varepsilon,C)>0$ \red{and all $\ell=2,\ldots, [m^{\alpha-\kappa_1}]+1$}, which concludes this part of the proof.

Next, we shall prove \eqref{Card2}. Let $\Psi^k_m:=\{S_{T^k_m+1}=S_{T^k_m}+{\bf e}_2\}$ with ${\bf e}_2=(0,1)$. 
By the strong Markov property, Proposition \ref{Theta}, {and the observation that $D^k_m\subset M^k_m$}, it suffices to show that
\begin{align}\label{Card3}
	\p\left(\# V^{(2)}(1)>C(\log m)^2,D^k_m,\Psi^k_m,\Theta^k_m=\emptyset\right)<e^{-c (\log m)^2}.
\end{align}
Here the introduction of $\Psi^k_m$ ensures that $L^k_m=\widetilde{S}_{N_{T^k_m}}$, which simplifies the following proof.

Define $I_0=[a_0,b_0):=[m+1,m+m^{\kappa_1})$. We use $\eta=(\eta_0,\eta_1,\ldots,\eta_n)$ to denote a generic nearest-neighbor path in $\z^2$ with finite length $\abs{\eta}=n$. Let $\eta_{\rm end}:=\eta_n$ be the ending point of the path. 

The notation $V^{(j)}_\eta$ for $j=1,2,3$ is defined through replacing $T^k_m$ in the definition of $V^{(j)}$ by $N^{-1}_+(\abs{\eta})$. Let
\[V^{(2)}_\eta(\ell):=\left\{\x=(x,x+{\bf e}_1)\in\X:\xi^+(\x,N^{-1}_+(\abs{\eta}))=\xi(x,N^{-1}_+(\abs{\eta}))\in I_\ell\right\}\]
for $\ell=0,1,\ldots,[m^{\alpha-\kappa_1}]+1$ (here we include the case $\ell=0$) and
\begin{align*}
	D_\eta:=\left\{\#V^{(1)}_\eta=k,V^{(1)}_\eta\cup V^{(2)}_\eta\cup V^{(3)}_\eta=\X,\xi\big(\eta_{\rm end},N^{-1}_+(\abs{\eta})\big)=m\right\}.
\end{align*}
{In particular, we have
\begin{align}\label{D_eta}
    \x(\eta_{\rm end})\in V^{(1)}_\eta\text{ on }D_\eta.
\end{align}}

In the sequel, we mainly focus on the even points in each $\x\in V^{(2)}(\ell)$, i.e. $\big\{x\in \z^2_{\rm e}:\x(x)\in V^{(2)}(\ell)\big\}$. By abuse of notation, we continue to write $V^{(2)}(\ell)$ for this set.

 With the above notations, we have for any $r>C(\log m)^2$,
 \begin{align}\label{V_eta}
 \begin{aligned}
	&\quad\,\p\left(\# V^{(2)}(1)=r,D^k_m,\Psi^k_m,\Theta^k_m=\emptyset\right)\\
	&=\sum_{\eta}\p\left(\# V^{(2)}(1)=r,D^k_m,\Theta^k_m=\emptyset,\Psi^k_m,\widetilde{S}_{[0,N_{T^k_m}]}=\eta\right)\\
	&\le\sum_{\eta}\p\left(\# V^{(2)}_\eta(1)=r,D_\eta,\theta(\abs{\eta},m,\kappa_1)=\emptyset,\widetilde{S}_{[0,\abs{\eta}]}=\eta\right),
\end{aligned}
\end{align}
where the sum is over all the paths with finite length and
\begin{align*}
    \theta(n,m,\kappa_1):=\Big\{&x\in \z^2_{\rm e}:\xi\big(x,N^{-1}_+(n)\big)\in {I_0\cup I_1},\\
&\widetilde{\xi}(x,n)\notin \Big[\frac{15}{16}(m-m^{\kappa_1})-c_*m^{1-\kappa_1},\frac{15}{16}m+c_*m^{1-\kappa_1}\Big)\Big\}.
\end{align*}
We mention that in the definition of $\theta(n,m,\kappa_1)$, the enlargement of the range of $\xi(x,n)$ from $I_1=(m-m^{\kappa_1},m)$ to $I_0\cup I_1$ is prepared for the later proof.

Now we use the same ideas as that in the proof of Lemma \ref{ty2}. With \eqref{D_eta} in mind, let
\begin{align*}
    \widetilde{D}_\eta:=\Big\{&\#V^{(1)}_\eta=k,V^{(1)}_\eta\cup V^{(2)}_\eta\cup V^{(3)}_\eta\cup V^{(2)}_\eta(0)=\X,\\
    &\xi\big(\eta_{\rm end},N^{-1}_+(\abs{\eta})\big)=m,{\x(\eta_{\rm end})\in V^{(1)}_\eta}\Big\}.
\end{align*}
Then we can write
\begin{align}\label{compare3}
\begin{aligned}
	&\quad\,\p\left(\# V^{(2)}_\eta(1)=r,D_\eta,\theta(\abs{\eta},m,\kappa_1)=\emptyset,\widetilde{S}_{[0,\abs{\eta}]}=\eta\right)\\
	&= \p\Big(\#V^{(2)}_\eta(1)=r,\#V^{(2)}_\eta(0)+\#V^{(2)}_\eta(1)=r,\widetilde{D}_\eta,\theta(\abs{\eta},m,\kappa_1)=\emptyset,\widetilde{S}_{[0,\abs{\eta}]}=\eta\Big).	
\end{aligned}
\end{align}
It follows from Proposition \ref{condilaw1} that conditionally on $V^{(2)}_\eta(0)\cup V^{(2)}_\eta(1)$,  $\widetilde{D}_\eta$, $\theta(\abs{\eta},m,\kappa_1)=\emptyset$, and $\widetilde{S}_{[0,\abs{\eta}]}=\eta$, we have 
\begin{align*}
\begin{aligned}
	V^{(2)}_\eta(0)\cup V^{(2)}_\eta(1)\subset\Big\{&x\in \z^2_{\rm e}:\widetilde{\xi}(x,\abs{\eta}){=\widetilde{\xi}^+(\x(x),\abs{\eta})}\\&\in \Big[\frac{15}{16}(m-m^{\kappa_1})-c_*m^{1-\kappa_1},\frac{15}{16}m+c_*m^{1-\kappa_1}\Big)\Big\}
\end{aligned}
\end{align*}
and the events $\big\{x\in V^{(2)}_\eta(1)\big\}$, for $x\in V^{(2)}_\eta(0)\cup V^{(2)}_\eta(1)$, are independent with probability
\begin{align*}
	\frac{\sum_{j\in I_1}\bar{p}(\widetilde{\xi}(x,\abs{\eta}),j)}{\sum_{j\in I_1}\bar{p}(\widetilde{\xi}(x,\abs{\eta}),j)+\sum_{j\in I_0}\bar{p}(\widetilde{\xi}(x,\abs{\eta}),j)}\le \frac{\bar{c}}{1+\bar{c}}.
\end{align*}
In particular, let
\[\q=\q(r):=\p\left(\,\cdot\,\Big\vert\,\#V^{(2)}_\eta(0)+\#V^{(2)}_\eta(1)=r,\widetilde{D}_\eta,\theta(\abs{\eta},m,\kappa_1)=\emptyset,\widetilde{S}_{[0,\abs{\eta}]}=\eta\right).\]
Then by the Stirling's formula,
\begin{align*}
	\q\left(\#V^{(2)}_\eta(1)=r\right)&\le O\left(\frac{1}{\sqrt{r}}\right)\left(\frac{\bar{c}}{1+\bar{c}}\right)^r \q\left(\#V^{(2)}_\eta(1)=\left[\frac{\bar{c}}{1+\bar{c}}r\right]\right)\\
	&\le e^{-c_2 r}\q\left(\#V^{(2)}_\eta(1)=\left[\frac{\bar{c}}{1+\bar{c}}r\right]\right)
\end{align*}
for some $c_2>0$.
Consequently, 
\begin{align}\label{compare4}
\begin{aligned}
	&\quad\,\mbox{ RHS of }\eqref{V_eta}\\
	&=\sum_\eta\q\left(\#V^{(2)}_\eta(1)=r\right)\p\Big(\#V^{(2)}_\eta(0)+\#V^{(2)}_\eta(1)=r,\widetilde{D}_\eta,\theta(\abs{\eta},m,\kappa_1)=\emptyset,\widetilde{S}_{[0,\abs{\eta}]}=\eta\Big)\\
	&\le e^{-c_2 r}\sum_\eta  \q\left(\#V^{(2)}_\eta(1)=\left[\frac{\bar{c}}{1+\bar{c}}r\right]\right)\\
	&\quad\,\cdot \p\Big(\#V^{(2)}_\eta(0)+\#V^{(2)}_\eta(1)=r,\widetilde{D}_\eta,\theta(\abs{\eta},m,\kappa_1)=\emptyset,\widetilde{S}_{[0,\abs{\eta}]}=\eta\Big)\\
	&\le e^{-c_2 r}\sum_\eta \p\Big(\#V^{(2)}_\eta(1)=\left[\frac{\bar{c}}{1+\bar{c}}r\right],\#V^{(2)}_\eta(0)=r-\left[\frac{\bar{c}}{1+\bar{c}}r\right],\widetilde{D}_\eta,\widetilde{S}_{[0,\abs{\eta}]}=\eta\Big)\\
	&\overset{(**)}{\le} e^{-c_2 r},	
\end{aligned}
\end{align}
where $(**)$ follows from the fact that the events $B_\eta:=\Big\{\#V^{(2)}_\eta(1)=\left[\frac{\bar{c}}{1+\bar{c}}r\right],\#V^{(2)}_\eta(0)=r-\left[\frac{\bar{c}}{1+\bar{c}}r\right],\widetilde{D}_\eta,\widetilde{S}_{[0,\abs{\eta}]}=\eta\Big\}$ are disjoint, i.e.
\begin{align}\label{disjoint}
	B_{\eta_1}\cap B_{\eta_2}=\emptyset\text{ if }\eta^1\neq \eta^2.
\end{align}
Indeed, \eqref{disjoint} trivially holds when $\eta^1\neq \eta^2$ and $\abs{\eta^1}=\abs{\eta^2}$. Next, we consider the case $\abs{\eta^1}<\abs{\eta^2}$. Note that on $B_\eta$,
\begin{align*}
	\left\{\x\in\X:\xi^+(\x,N^{-1}_+(\abs{\eta})-1)\ge m\right\}&=\big(V^{(1)}_\eta\setminus \{\x(\eta_{\rm end})\}\big)\cup V^{(2)}_\eta(0),\\
	\left\{\x\in\X:\xi^+(\x,N^{-1}_+(\abs{\eta}))\ge m\right\}&=V^{(1)}_\eta\cup V^{(2)}_\eta(0).
\end{align*}
In particular,
\begin{align*}
	\#\left\{\x\in\X:\xi^+(\x,N^{-1}_+(\abs{\eta})-1)\ge m\right\}&=k+r-\left[\frac{\bar{c}}{1+\bar{c}}r\right]-1,\\
	\#\left\{\x\in\X:\xi^+(\x,N^{-1}_+(\abs{\eta}))\ge m\right\}&=k+r-\left[\frac{\bar{c}}{1+\bar{c}}r\right].
\end{align*}
In the case $\abs{\eta^1}<\abs{\eta^2}$, we have $N^{-1}_+(\abs{\eta^1})\le N^{-1}_+(\abs{\eta^2})-1$. Then the above analysis tells us that on $B_{\eta^1}\cap B_{\eta^2}$,
\begin{align*}
	k+r-\left[\frac{\bar{c}}{1+\bar{c}}r\right]&=\#\left\{\x\in\X:\xi^+(\x,N^{-1}_+(\abs{\eta^1}))\ge m\right\}\\
	&\le \#\left\{\x\in\X:\xi^+(\x,N^{-1}_+(\abs{\eta^2})-1)\ge m\right\}=k+r-\left[\frac{\bar{c}}{1+\bar{c}}r\right]-1,
\end{align*}
which implies \eqref{disjoint}. 

Finally, by summing over $r>C(\log m)^2$ in \eqref{compare4}, we get  \eqref{Card3}.
\end{proof}

\begin{proof}[Proof of Proposition \ref{smallalpha}]
By \eqref{relation1}, it is enough to show that 
\begin{align}\label{Me}
	\begin{aligned}
		\p\Big\{\mathcal{M}_{\rm e}^k(m,\alpha)\neq \emptyset\,\Big\vert\,&\Theta(k,m,[m-m^{\kappa_1}+1,m))=\emptyset,\,\#\mathcal{M}^k(m,{\kappa_1})\le (\log m)^2,\\
		&M^k_m,\,\widetilde{S}_{[0,N_{T^k_m}]},\,\left(L^j_m:j=1,\ldots,k\right)\Big\}\lesssim \frac{(\log m)^2}{m^{\kappa_1-\alpha}},
	\end{aligned}
	\end{align}
	and 
	\begin{align}\label{Mo}
	\begin{aligned}
		\p\Big\{\mathcal{M}_{\rm o}^k(m,\alpha)\neq \emptyset\,\Big\vert\,&\Theta'(k,m,[m-m^{\kappa_1}+1,m))=\emptyset,\,\#\mathcal{M}^k(m,{\kappa_1})\le (\log m)^2,\\
		&M^k_m,\,\widetilde{S}_{[0,N_{T^k_m}]},\,\left(L^j_m:j=1,\ldots,k\right)\Big\}\lesssim \frac{(\log m)^2}{m^{\kappa_1-\alpha}}.
	\end{aligned}
	\end{align}
We shall only give the proof of \eqref{Me}, for \eqref{Mo} can be proved in the same way.

In addition to the conditioning in \eqref{Me}, we further condition on $\mathcal{M}^k(m,{\kappa_1})$ and denote by {$\q=\q\big(\mathcal{M}^k(m,{\kappa_1})\big)$} the conditional probability. Then we have 
\begin{align}\label{M^k_e}
\begin{aligned}
    \mathcal{M}_{\rm e}^k(m,\alpha)\subset\Big\{&x\in \mathcal{M}^k(m,{\kappa_1})\cap \z^2_{\rm e}:\widetilde{\xi}(x,N_{T^k_m})=\widetilde{\xi}^+(\x(x),N_{T^k_m})\\&\in \Big(\frac{15}{16}(m-m^{\kappa_1})-c_*m^{1-\kappa_1},\frac{15}{16}m+c_*m^{1-\kappa_1}\Big)\Big\}.    
\end{aligned}
\end{align}
It follows from Proposition \ref{condilaw2} (In the proof of \eqref{Mo}, Proposition \ref{condilaw21} will be used instead.) that the events $\big\{x\in \mathcal{M}_{\rm e}^k(m,\alpha)\big\}$, for $x$ {in the set on the RHS of \eqref{M^k_e}}, are conditionally independent with probability {(recall again that $\bar{p}(i,j):=p(i,j-i)$)}
\[\frac{\sum_{j=m-m^{\alpha}+1}^{m-1}\bar{p}(\widetilde{\xi}(x,N_{T^k_m}),j)}{\sum_{j=m-m^{\kappa_1}+1}^{m-1}\bar{p}(\widetilde{\xi}(x,N_{T^k_m}),j)}.\]
\red{Similar to Lemma \ref{pbound}, as a corollary of Lemma \ref{stirling01}}, uniformly in $i\in \big(\frac{15}{16}(m-m^{\kappa_1})-c_*m^{1-\kappa_1},\frac{15}{16}m+c_*m^{1-\kappa_1}\big)$ and $j_1,j_2\in (m-m^{\kappa_1},m)$, 
\begin{equation}\label{eq:seesaw}
\bar{p}(i,j_1)\asymp \bar{p}(i,j_2).    
\end{equation}
Consequently, there exists universal $c_1,c_2>0$ such that for all $x\in  \mathcal{M}^k(m,{\kappa_1})\cap \z^2_{\rm e}$,
\[\frac{c_1}{m^{\kappa_1-\alpha}}\le \q\big\{x\in \mathcal{M}_{\rm e}^k(m,\alpha)\big\}\le \frac{c_2}{m^{\kappa_1-\alpha}}.\]
Further noting the conditioning $\big\{\#\mathcal{M}^k(m,{\kappa_1})\le (\log m)^2\big\}$, we get
\[\q\left\{\mathcal{M}_{\rm e}^k(m,\alpha)\neq \emptyset\right\}\le \frac{c_2(\log m)^2}{m^{\kappa_1-\alpha}},\]
which implies \eqref{Me}.
\end{proof}

\appendix
\section{Proof of Proposition \ref{lower}}
In this appendix, we give a proof of Proposition \ref{lower}, adapting the arguments from Sections 3-6 of \cite{Rosen05}.
Let $K_n=16e^nn^9$ and fix $\epsilon>0$. It suffices to show that for $n$ large,
\begin{align}\label{Kn}
	\p\left[\xi^*({\mathsf{H}_{D(0,K_n)^{\rm c}}})\ge \frac{4}{\pi}(\log K_n)^2-(\log K_n)^{8/5+2\epsilon}\right]&> \exp\left\{-n^{3/5+\epsilon/3}\right\}.
\end{align}
Indeed, we note that there exists $c>0$ such that for any $n\ge 1$,
\[\p\left[\mathsf{H}_{D(0,K_n)^{\rm c}}\ge(K_ne^{(\log K_n)^{3/5+\epsilon}})^2\right]\le \exp\left\{-c(\log K_n)^{3/5+\epsilon}\right\}\le \exp\left\{-cn^{3/5+\epsilon}\right\}.\]
Letting $J_n:=(K_ne^{(\log K_n)^{3/5+\epsilon}})^2$, \eqref{Kn} implies that for $n$ large,
\begin{align*}
	&\quad\,\p\left[\xi^*({J_n})\ge \frac{1}{\pi}(\log J_{n+1})^2-(\log J_{n+1})^{8/5+3\epsilon}\right]\\
	&\ge \p\left[\xi^*({\mathsf{H}_{D(0,K_n)^{\rm c}}})\ge \frac{4}{\pi}(\log K_n)^2-(\log K_n)^{8/5+2\epsilon}\right]-\p\left[\mathsf{H}_{D(0,K_n)^c}\ge J_n\right]\\
	&> \exp\left\{-n^{3/5+\varepsilon/2}\right\}\ge \exp\left\{-(\log J_n)^{3/5+\epsilon/2}\right\}.
	\end{align*}
	By interpolation, we have
\begin{align*}
	\p\left[\xi^*(n)\ge \frac{1}{\pi}(\log n)^2-(\log n)^{8/5+3\epsilon}\right]>\exp\left\{-(\log n)^{3/5+\epsilon/2}\right\},
\end{align*}
and consequently,
\begin{align*}
	\p\left[\xi^*(n/e^{(\log n)^{3/5+2\epsilon}})\ge \frac{1}{\pi}(\log n)^2-(\log n)^{8/5+4\epsilon}\right]>\exp\left\{-(\log n)^{3/5+\epsilon}\right\}.
\end{align*}
By partitioning the random walk path of length $n$ to $e^{(\log n)^{3/5+2\epsilon}}$ segments of length $n/e^{(\log n)^{3/5+2\epsilon}}$ each and using the independence of increments, we obtain Proposition \ref{lower}.

{\begin{rem}
    The proof of \eqref{Kn} is a refinement of the arguments in Sections 3-6 of \cite{Rosen05}. Both \cite{Rosen05} and our work consider certain subsets of thick points, whose probability is easier to bound below. 
    The major difference is that we modify the definition of $K_n$ and $r_{n,k}$ below, which enables us to improve the estimates of the first moment as shown in \eqref{first_moment} below; while with the definitions in \cite{Rosen05}, one can check that the first moment (of the corresponding subsets of thick points considered here) can not exceed $\exp\big\{-C(n/\log n)\big\}$ for some $C>0$, which is not enough for the proof of \eqref{Kn}.
\end{rem}}

Fixing $x\in \z^2$ and $n\ge 1$, let $\bar{\eta}_0=0$ and for $i\ge 1$,
\begin{align*}
	\eta_i=\inf\{t\ge \bar{\eta}_{i-1}:S_t\in \partial D(x,r^0_n)\},\quad
	\bar{\eta}_i=\inf\{t\ge \eta_i:S_t\in \partial D(x,R_n)\}.
\end{align*}
We define $\mathcal{G}^x_{(R_n,r^0_n)}$ as the $\sigma$-algebra generated by the excursions $\{e^{(i)}:i\ge 1\}$, where $e^{(i)}=\{S_t:\bar{\eta}_{i-1}\le t\le \eta_i\}$ is the $i$-th excursion from $\partial D(x,R_n)$ to $\partial D(x,r^0_n)$ (for $i=1$, the excursion begins at $t=0$).
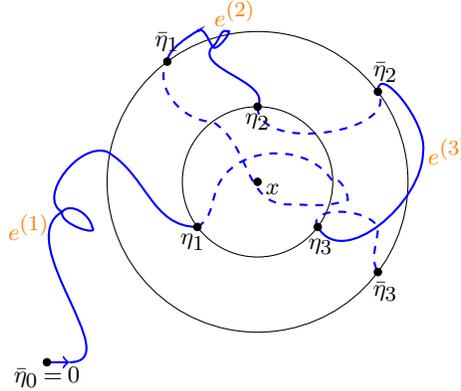
\begin{figure}[h]
\centering
\begin{tikzpicture}[scale=2]
    \coordinate (center) at (0,0);
    \def\outerRadius{1}
    \def\innerRadius{0.5}

    \draw[fill=white] (center) circle (\outerRadius);
    \draw[fill=white] (center) circle (\innerRadius);

    \coordinate (start) at (-1.4,-1.2);

    \draw[thick, blue,->]
    (start)
    to[out=0,in=-180] (-1.25,-1.2);
    \draw[thick, blue]
        (-1.25,-1.2)
        to[out=0,in=-150] (-1.3,-0.2)
        to[out=50,in=100] (-1.1,-0.3)
        to[out=-50,in=-100] (-1.3,-0.2)
        to[out=150,in=160] (-1,0.2)
        to[out=-30, in=180] (-0.4,-0.3); 
    \draw[thick, blue, dashed]
        (-0.4,-0.3)
        to[out=30, in=240] (-0.25,0)
        to[out=60, in=120] (0.5,0)
        to[out=60, in=90] (0.6,-0.12)
        to[out=200, in=0] (0.2,-0.15)
        to[out=-180, in=-30] (-0.3,0.4)
        to[out=-180, in=-120] (-0.6,0.8);
    \draw[thick, blue]
        (-0.6,0.8)
    to[out=90,in=30] (-0.4,1)
    to[out=45,in=120] (-0.3,0.9)
    to[out=-60,in=30] (-0.2,1)
    to[out=150,in=30] (-0.3,0.9)
    to[out=-130,in=60] (0,0.5);
    \draw[thick, blue, dashed]
        (0,0.5)
        to[out=-90,in=-60] (0.8,0.6);
    \draw[thick, blue]
        (0.8,0.6)
        to[out=90,in=80] (1.1,0.2)
        to[out=-90,in=-60] (0.4,-0.3);
    \draw[thick, blue, dashed]
        (0.4,-0.3)
        to[out=120,in=120] (0.8,-0.3)
        to[out=-90,in=120] (0.8,-0.6);

    \filldraw [black] (center) circle (0.7pt);
    \node at (0.1,-0.05) {$x$};
    \filldraw [black] (start) circle (0.7pt);
    \node at (-1.4,-1.3) {$\bar{\eta}_0=0$};
    \filldraw [black] (-0.4,-0.3) circle (0.7pt);
    \node at (-0.43,-0.4) {$\eta_1$};
    \filldraw [black] (-0.6,0.8) circle (0.7pt);
    \node at (-0.6,0.93) {$\bar{\eta}_1$};
    \filldraw [black] (0,0.5) circle (0.7pt);
    \node at (-0,0.4) {$\eta_2$};
    \filldraw [black] (0.8,0.6) circle (0.7pt);
    \node at (0.85,0.73) {$\bar{\eta}_2$};
    \filldraw [black] (0.4,-0.3) circle (0.7pt);
    \node at (0.42,-0.43) {$\eta_3$};
    \filldraw [black] (0.8,-0.6) circle (0.7pt);
    \node at (0.85,-0.73) {$\bar{\eta}_3$};
    \node at (-1.52,-0.3) {\color{orange}$e^{(1)}$};
    \node at (-0.15,1.12) {\color{orange}$e^{(2)}$};
    \node at (1.27,0.2) {\color{orange}$e^{(3)}$};
\end{tikzpicture}
\caption{Excursions $\{e^{(i)}:i\ge 1\}$. The two concentric disks are $D(x,R_n)$ and $D(x,r^0_n)$. The blue curve (both solid and dashed parts) is the sample path of $S$, the three solid sections of which are $e^{(1)}$, $e^{(2)}$, and $e^{(3)}$ respectively.}
\end{figure}

For each $i\ge 1$, let $\bar{\zeta}_{i,0}=\eta_i$ and for $j\ge 1$, define
\begin{align*}
	\zeta_{i,j}&:=\inf\{t\ge \bar{\zeta}_{i,j-1}:S_t\in \partial D(x,r_n)\},\\
	\bar{\zeta}_{i,j}&:=\inf \{t\ge \zeta_{i,j}:S_t\in \partial D(x,r^0_n)\}.
\end{align*}
Let $v_{i,j}:=\{S_t:\zeta_{i,j}\le t\le \bar{\zeta}_{i,j}\}$ and $Z^i:=\sup\{j\ge 0: \bar{\zeta}_{i,j}<\bar{\eta}_i\}$. Then $(v_{i,j}:j=1,\ldots,Z^i)$ are excursions from $\partial D(x,r_n)$ to $\partial D(x,r^0_n)$ during the $i$-th excursion from $\partial D(x,r^0_n)$ to $\partial D(x, R_n)$ and we define $\mathcal{H}^x_{(R_n,r^0_n,r_n),i}:=\sigma(v_{i,j}:j=1,\ldots,Z^i)$.
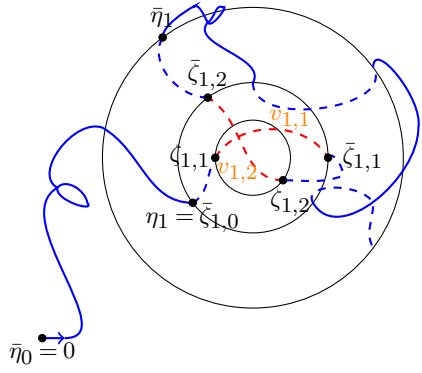
\begin{figure}[h]
\centering
\begin{tikzpicture}[scale=2]
    \coordinate (center) at (0,0);
    \def\outerRadius{1}
    \def\innerRadius{0.5}

    \draw[fill=white] (center) circle (\outerRadius);
    \draw[fill=white] (center) circle (\innerRadius);
    \draw[fill=white] (center) circle (0.25);

    \coordinate (start) at (-1.4,-1.2);

    \draw[thick, blue,->]
    (start)
    to[out=0,in=-180] (-1.25,-1.2);
    \draw[thick, blue]
        (-1.25,-1.2)
        to[out=0,in=-150] (-1.3,-0.2)
        to[out=50,in=100] (-1.1,-0.3)
        to[out=-50,in=-100] (-1.3,-0.2)
        to[out=150,in=160] (-1,0.2)
        to[out=-30, in=180] (-0.4,-0.3); 
    \draw[thick, blue, dashed]
        (-0.4,-0.3)
        to[out=30, in=240] (-0.25,0);
    \draw[thick, red, dashed]
        (-0.25,0)
        to[out=60, in=120] (0.5,0);
    \draw[thick, blue, dashed]
    	(0.5,0)
        to[out=60, in=90] (0.6,-0.12)
        to[out=200, in=0] (0.2,-0.15);
    \draw[thick, red, dashed]
    	(0.2,-0.15)
        to[out=-180, in=-30] (-0.3,0.4);
    \draw[thick, blue, dashed]
        (-0.3,0.4)
        to[out=-180, in=-120] (-0.6,0.8);
    \draw[thick, blue]
        (-0.6,0.8)
    to[out=90,in=30] (-0.4,1)
    to[out=45,in=120] (-0.3,0.9)
    to[out=-60,in=30] (-0.2,1)
    to[out=150,in=30] (-0.3,0.9)
    to[out=-130,in=60] (0,0.5);
    \draw[thick, blue, dashed]
        (0,0.5)
        to[out=-90,in=-60] (0.8,0.6);
    \draw[thick, blue]
        (0.8,0.6)
        to[out=90,in=80] (1.1,0.2)
        to[out=-90,in=-60] (0.4,-0.3);
    \draw[thick, blue, dashed]
        (0.4,-0.3)
        to[out=120,in=120] (0.8,-0.3)
        to[out=-90,in=120] (0.8,-0.6);

    \filldraw [black] (start) circle (0.7pt);
    \node at (-1.4,-1.3) {$\bar{\eta}_0=0$};
    \filldraw [black] (-0.4,-0.3) circle (0.7pt);
    \node at (-0.4,-0.4) {$\eta_1=\bar{\zeta}_{1,0}$};
    \filldraw [black] (-0.6,0.8) circle (0.7pt);
    \node at (-0.6,0.93) {$\bar{\eta}_1$};
    \filldraw [black] (-0.25,0) circle (0.7pt);
    \node at (-0.42,0) {$\zeta_{1,1}$};
    \filldraw [black] (0.5,0) circle (0.7pt);
    \node at (0.73,0) {$\bar{\zeta}_{1,1}$};
    \filldraw [black] (0.2,-0.15) circle (0.7pt);
    \node at (0.25,-0.3) {$\zeta_{1,2}$};
    \filldraw [black] (-0.3,0.4) circle (0.7pt);
    \node at (-0.3,0.55) {$\bar{\zeta}_{1,2}$};
    \node at (0.25,0.25) {\color{orange}$v_{1,1}$};
    \node at (-0.1,-0.1) {\color{orange}$v_{1,2}$};
\end{tikzpicture}
\caption{Excursions $\{v_{i,j}:j=1,\ldots,Z^i\}$. The disk $D(x,r_n)$ is added in this figure. The two red dashed sections are $v_{1,1}$ and $v_{1,2}$ respectively and $Z^1=2$ in this case.}
\end{figure}

Fixing $n\ge 1$, for any $m\ge 1$, let $\mathcal{H}^x_{(R_n,r^0_n,r_n)}(m)$ be the collection of events $H_n$ which can be written as a countable union of disjoint events in the form $H_{n,1}\cap H_{n,2}\cap\cdots\cap H_{n,m}$ with $H_{n,i}\in \mathcal{H}^x_{(R_n,r^0_n,r_n),i}$. (Be careful that $\mathcal{H}^x_{(R_n,r^0_n,r_n)}(m)$ is not a $\sigma$-algebra.) Moreover, we say 
\[\mbox{$H_n\in\mathcal{H}^x_{(R_n,r^0_n,r_n)}(m)$ satisfies Condition \eqref{On} with parameter $\lambda>0$}\]

if for any $y,z\in \partial D(x,r_n)$ and each $H_{n,i}$ in the above union,
\begin{align}\label{On}\tag{$*$}
 	\big(1-\lambda n^{-3}\big)\p^y(H_{n,i})\le \p^z(H_{n,i})\le \left(1+\lambda n^{-3}\right)\p^y(H_{n,i}).
 \end{align}
 Here with slight abuse of notation we also denote by $\p^y$ the law of $(v_{i,j}:j\ge 1)$ conditioned on $Z^i\ge 1$ and $v_{i,1}$ starts from $y$. 
We show the following lemma using the same arguments as that in Lemma 6.3 of \cite{Rosen05}.
 \begin{lem}\label{preliminary}
 Fix $(R_n,r^0_n,r_n)$ with $r^0_n\ge 2r_n$ and $\frac{R_n}{r^0_n}\wedge  r_n\ge n^3$. Then for any $\lambda>0$, there exists $\lambda'>0$ such that for all $m,n\in \N$ with $1\le m\le n^{5/2}$, $H_n\in \mathcal{H}^x_{(R_n,r^0_n,r_n)}(m)$ satisfying Condition \eqref{On} with parameter $\lambda$, $x\in \z^2$, and $y_0,y_1\in \z^2\setminus D(x,r^0_n)$,
  	\begin{align}\label{On2}
 		\left(1-\lambda'mn^{-3}\log n\right)\p^{y_1}(H_n)\le \p^{y_0}(H_n\,\vert\,\mathcal{G}^x_{(R_n,r^0_n)})\le \left(1+\lambda'mn^{-3}\log n\right)\p^{y_1}(H_n).
 	\end{align}
 \end{lem}

\subsection{Proof of \eqref{Kn}}
Let $K_n=16e^nn^9$. We define $r_{n,k}=e^{n-k}n^9$ for $k=0,\ldots,n$ and $r_{n,n+1}=n^6$. In particular, we have $r_{n,0}=\frac{1}{16}K_n$. Write $U_n=[2r_{n,0},3r_{n,0}]^2$. Let $N^x_{n,k}$ be the number of excursions from $\partial D(x,r_{n,k-1})$ to $\partial D(x,r_{n,k})$ before $\mathsf{H}_{\partial D(0,K_n)}$ for $k=1,\ldots, n+1$. Fixing $0<\delta<\delta'<1$, a point $x\in U_n$ is called $(n,\delta)$-successful if $$
\mbox{$N^x_{n,1}=1$, $\abs{N^x_{n,k}-2k^2}\le k^{1+\delta}$,  $\forall k=2,\ldots, n$, and $N^x_{n,n+1}\in \left[\frac{2n^2-n^{1+\delta}}{3\log n},n^3\right]$.}
$$
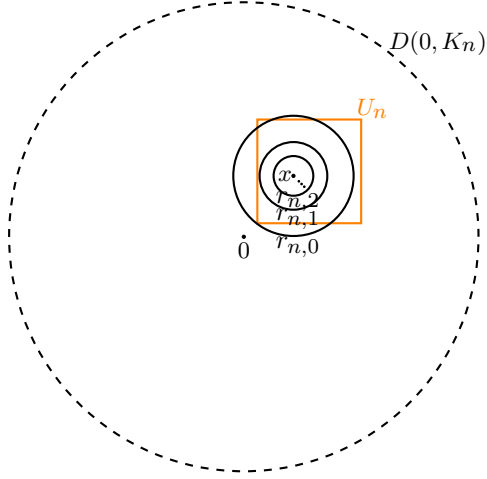
\begin{figure}
\centering
\begin{tikzpicture}[scale=0.6]
    \draw[thick, orange] (0.7*1cm-0.4cm, 0.7*1cm-0.4cm) rectangle (3*1cm-0.4cm, 3*1cm-0.4cm);
    \node at (3.25*1cm-0.4cm, 3.25*1cm-0.4cm) {\color{orange}$U_n$};

    \node at (1.31*1cm-0.4cm, 1.75*1cm-0.4cm) {\( x \)};

    \draw[thick] (1.5*1cm-0.4cm, 1.75*1cm-0.4cm) circle (1.333cm);
    \draw[thick] (1.5*1cm-0.4cm, 1.75*1cm-0.4cm) circle (0.75cm);
    \draw[thick] (1.5*1cm-0.4cm, 1.75*1cm-0.4cm) circle (0.44cm);

    \node at (1.6cm-0.4cm, 0.18cm-0.4cm) {\small{\( r_{n,0} \)}};
    \node at (1.6cm-0.4cm, 0.8cm-0.4cm) {\small{\( r_{n,1} \)}};
    \node at (1.6cm-0.4cm, 1.18cm-0.4cm) {\small{\( r_{n,2} \)}};

    \draw[thick, dashed] (0,0) circle (5.2*1cm);
    \node at (4.3*1cm, 4.3*1cm) {\( D(0, K_n) \)};

    \filldraw[black] (1.5*1cm-0.4cm, 1.75*1cm-0.4cm) circle (1pt);
    \filldraw[black] (0,0) circle (1pt);
    \filldraw[black] (1.6cm+0.02cm-0.4cm,1.65cm-0.02cm-0.4cm) circle (0.5pt);
    \filldraw[black] (1.66cm+0.02cm-0.4cm,1.59cm-0.02cm-0.4cm) circle (0.5pt);
    \filldraw[black] (1.72cm+0.02cm-0.4cm,1.53cm-0.02cm-0.4cm) circle (0.5pt);

    \node at (0, -0.3cm) {\( 0 \)};
\end{tikzpicture}
\caption{$D(0,K_n)$, $D(x,r_{n,k})$, and $U_n$.}
\end{figure}
We set
\begin{align*}
Y(n,x)&=1\left\{x\text{ is $(n,\delta)$-successful}\right\},\\
	Y'(n,x)&=1\left\{\xi(x,\mathsf{H}_{D(0,K_n)^c})\ge \frac{4}{\pi}(\log K_n)^2-(\log K_n)^{1+\delta'},\ x\text{ is $(n,\delta)$-successful}\right\}.
\end{align*}
Obviously,\begin{align*}
	\p\left\{\xi^*(\mathsf{H}_{D(0,K_n)^c})\ge \frac{4}{\pi}(\log K_n)^2-(\log K_n)^{1+\delta'}\right\}\ge \p\left\{\sum_{x\in U_n}Y'(n,x)\ge 1\right\}.
\end{align*}
In the following, we shall prove
\begin{align}\label{second}
	\p\left\{\sum_{x\in U_n}Y'(n,x)\ge 1\right\}\ge \exp\left\{-n^{(1-2\delta)\vee3\delta+o(1)}\right\}.
\end{align}
Once it is proved, taking $\delta=1/5$ and $\delta'=3/5+2\epsilon$ gives \eqref{Kn}.
The proof will be completed by the second moment method. Recall that for non-negative sequences $(a_n)$ and $(b_n)$, we write $a_n\lesssim b_n$ if there exists $c>0$ such that for any $n\ge 1$, $a_n\le cb_n$. And write $a_n\asymp b_n$ if $a_n\lesssim b_n$ and $b_n\lesssim a_n$. When $a_n(x)$ and $b_n(x)$ are functions, we say $a_n \lesssim b_n$ uniformly in $x$ (resp. $a_n\asymp b_n$ uniformly in $x$) if the above constant $c$ does not depend on $x$.
\begin{pro}\label{moment}
(1) One has,
$$\inf_{x\in U_n}E[Y'(n,x)]\asymp \sup_{x\in U_n}E[Y'(n,x)]\ge \exp\left\{-2n-n^{[(1-2\delta)\vee 3\delta]+o(1)}\right\}.$$
		Consequently, 
  \begin{align}\label{first_moment}
      I'_n:=\e\left[\sum_{x\in U_n}Y'(n,x)\right]\ge \exp\left\{-n^{[(1-2\delta)\vee 3\delta]+o(1)}\right\};
  \end{align}
\noindent(2) Let $$l(x,y):=\min\{m\ge 1:D(x,r_{n,m})\cap D(y,r_{n,m})=\emptyset\}.$$ Then uniformly in $x,y\in U_n$ with $l(x,y)\le n-3\log n$,
		\begin{align*}
			\e\left[Y'(n,x)Y'(n,y)\right]\le \exp\left\{2l(x,y)+n^{[(1-2\delta)\vee 3\delta]+o(1)}\right\}\e[Y'(n,x)]\e[Y'(n,y)].
		\end{align*}
\end{pro}
\begin{proof}[Proof of \eqref{second}]
We write $Q_n:=\inf_{x\in U_n}\e[Y'(n,x)]$. Then it follows from Proposition \ref{moment} (1) that uniformly in $x\in U_n$,
$Q_n\asymp \e[Y'(n,x)]$. And consequently, $I'_n\asymp K_n^2 Q_n$.
	
	Note that the definition of $l(x,y)$ implies that $\abs{x-y}\le 2r_{n,l(x,y)-1}$. Since there are at most $C_0r^2_{n,l-1}\asymp K_n^2/e^{2l}$ points in the disc $D(x,2r_{n,l-1})$, it follows from Proposition \ref{moment} (2) that
	\begin{align}\label{second2}
		\begin{aligned}
			&\quad\,\sum_{\substack{x,y\in U_n\\2r_{n,n-3\log n}\le \abs{x-y}\le 2r_{n,0}}}\e\left[Y'(n,x)Y'(n,y)\right]\\
			&\le \sum_{x\in U_n}\sum_{l=1}^{n-3\log n+1}\sum_{y\in U_n:l(x,y)=l}\exp\left\{2l+n^{(1-2\delta)\vee3\delta+o(1)}\right\}\e[Y'(n,x)]\e[Y'(n,y)]\\
			&\asymp K_n^2 \sum_{l=1}^{n-3\log n+1}\frac{K_n^2}{e^{2l}}\exp\left\{2l+n^{(1-2\delta)\vee3\delta+o(1)}\right\}Q_n^2\\
			&\asymp (I'_n)^2\exp\left\{n^{(1-2\delta)\vee3\delta+o(1)}\right\}.
		\end{aligned}
	\end{align}
	
	On the other hand,
	\begin{align}\label{second3}
		\begin{aligned}
			&\quad\,\sum_{\substack{x,y\in U_n\\ \abs{x-y}\le 2r_{n,n-3\log n}}}\e\left[Y'(n,x)Y'(n,y)\right]\lesssim \sum_{\substack{x,y\in U_n\\ \abs{x-y}\le 2r_{n,n-3\log n}}}Q_n\lesssim n^{24} K_n^2 Q_n\\
			&\lesssim n^{24} K_n^2 Q_n \exp\left\{n^{(1-2\delta)\vee3\delta+o(1)}\right\} I'_n \asymp (I'_n)^2\exp\left\{n^{(1-2\delta)\vee3\delta+o(1)}\right\},
		\end{aligned}
	\end{align}
	where the third inequality comes from Proposition \ref{moment} (1).
	
	Combining \eqref{second2} and \eqref{second3} and then using the second moment method gives \eqref{second}.
\end{proof}

\subsubsection{Proof of Proposition \ref{moment} (1)}\label{first}
First, we shall show that uniformly in $x\in U_n$,
\begin{align}\label{equiv1}
	\e[Y'(n,x)]=\big(1-O(n^{-1})\big) \e[Y(n,x)].
\end{align}
Based on this, it suffices to prove the following lemma.
\begin{lem}\label{Y}
	\begin{align*}
	\inf_{x\in U_n}E[Y(n,x)]\gtrsim \sup_{x\in U_n}E[Y(n,x)]\ge \exp\left\{-2n-O\big(n^{[(1-\delta)\vee 3\delta]+o(1)}\big)\right\}.
\end{align*}
\end{lem}
\begin{proof}[Proof of \eqref{equiv1}]
It is plain that $\e[Y'(n,x)]\le \e[Y(n,x)]$. Thus we need only to show $\e[Y'(n,x)]\ge\big(1-O(n^{-1})\big) \e[Y(n,x)]$. Recall that $\mathcal{G}^x_{(r_{n,n},r_{n,n+1})}$ is the $\sigma$-algebra of excursions from $\partial D(x,r_{n,n})$ to $\partial D(x,r_{n,n+1})$. Since $Y(n,x)\in \mathcal{G}^x_{(r_{n,n},r_{n,n+1})}$,
\begin{align}\label{YY}
\begin{aligned}
	&\e[Y'(n,x)]\\
 \ge\;&\e\left\{\p\left(\sum_{j=1}^{\frac{2n^2-n^{1+\delta}}{3\log n}} L^{x,j}\ge \frac{4}{\pi}(\log K_n)^2-(\log K_n)^{1+\delta'}\,\bigg\vert\,\mathcal{G}^x_{(r_{n,n},r_{n,n+1})}\right);\,Y(n,x)=1\right\},	
\end{aligned}
\end{align}
where $L^{x,j}$ is the number of visits to $x$ during the $j$-th excursion from $\partial D(x,r_{n,n+1})$ to $\partial D(x,r_{n,n})$. Let $m_n:=\frac{2n^2-n^{1+\delta}}{3\log n}$. The main idea is to apply Lemma \ref{preliminary} to the event $\Big\{\sum_{j=1}^{m_n} L^{x,j}\ge \frac{4}{\pi}(\log K_n)^2-(\log K_n)^{1+\delta'}\Big\}$.
 We claim that
\begin{align}\label{cH}
	\left\{\sum_{j=1}^{m_n} L^{x,j}\ge \frac{4}{\pi}(\log K_n)^2-(\log K_n)^{1+\delta'}\right\}\in \mathcal{H}^x_{(r_{n,n},r_{n,n+1},n^6),j}(m_n)
\end{align}
and satisfies Condition \eqref{On} with parameter $1$ for $n$ large.

 Indeed, the event can be written as a disjoint union of events
$\bigcap_{j=1}^{m_n}\left\{L^{x,j}=a_j\right\}$,
for non-negative integers $\{a_j:j=1,\ldots,m_n\}$ with $$\sum_{j=1}^{m_n}a_j\ge \frac{4}{\pi}(\log K_n)^2-(\log K_n)^{1+\delta'},$$ which implies \eqref{cH}.

For Condition \eqref{On}, we note that for each excursion, conditionally on starting from $y\in \partial D(x,r_{n,n+1})$, its number of visits at $x$ has the same law as the product of a Bernoulli random variable with success probability $\p^{y-x}(\mathsf{H}_0<\mathsf{H}_{\partial D(0,r_{n,n})})$ and an independent $\{1,2,\ldots\}$-valued geometric random variable with success probability $\p^0(\mathsf{H}_{\partial D(0,r_{n,n})}<\mathsf{H}^+_0)$, where $\mathsf{H}^+_0=\inf\{t\ge1 :S_t=0\}$. It follows from \eqref{hitting} that there exists constant $c_0>0$ such that
\begin{align*}
	\p^{y-x}(\mathsf{H}_0<\mathsf{H}_{\partial D(0,r_{n,n})})&=\frac{\frac{2}{\pi}\cdot 3\log(n)}{\frac{2}{\pi}\cdot 9\log(n)+c_0}\left[1+O\left(\frac{1}{n^6\log n}\right)\right],\\
	\p^0(\mathsf{H}_{\partial D(0,r_{n,n})}<\mathsf{H}^+_0)&=\frac{1}{\frac{2}{\pi}\cdot 9\log(n)+c_0}\left[1+O\left(\frac{1}{n^6\log n}\right)\right],
\end{align*}
from which we can easily deduce Condition \eqref{On}.

Moreover, since the excursions are independent conditionally on their starting points, it is also seen from the above analysis that
\begin{align*}
	\e\Bigg[\sum_{j=1}^{\frac{2n^2-n^{1+\delta}}{3\log n}} L^{x,j}\Bigg]&=\frac{4}{\pi}n^2-\frac{18}{\pi}n^{1+\delta}+O\left(\frac{1}{n^4\log n}\right)\\
	&=\frac{4}{\pi}(\log K_n)^2-\frac{18}{\pi}(\log K_n)^{1+\delta}+o\left((\log K_n)^{1+\delta}\right).
\end{align*}
Since $\delta<\delta'$, it is not hard to use the properties of geometric distribution to show that
\[\p\left(\sum_{j=1}^{\frac{2n^2-n^{1+\delta}}{3\log n}} L^{x,j}\ge \frac{4}{\pi}(\log K_n)^2-(\log K_n)^{1+\delta'}\right)=1-O(n^{-1}).\]
Hence applying Lemma \ref{preliminary} to \eqref{YY} gives \eqref{equiv1}.
\end{proof}

Turning to Lemma \ref{Y}, the basic idea is to relate $N^x_{n,k}$ to the number of upcrossings of some Markov jump process (MJP) in the same way as that in Section 4 of \cite{Rosen05}, which translates the estimates of $\e(n,x)$ into estimates on upcrossings of MJP. The hardest part of the estimates is to handle the contribution made by the deviation $\abs{N^x_{n,k}-2k^2}$. We shall show that this contribution can be bounded by the small value probability of the maximum absolute value of one-dimensional Brownian motion.

Let us introduce the MJP which we are concerned with. Let $\bar{m}=(m_2,\ldots,m_n)\in \N^{n-1}$ and
$\mathcal{M}=\mathcal{M}_n(\delta):=\left\{\bar{m}\in \N^{n-1}:\forall\,k=2,\ldots,n,\,\abs{m_k-2k^2}\le k^{1+\delta}\right\}$.
Let $X^0$ be a MJP on $\{0,\ldots,n+1\}$ which is killed upon hitting $0$ and has transition probability
\begin{align*}
	p^0(x,x\pm1)&=1/2\text{ for }x\in[1,n-1],\\
	p^0(n,n+1)&=1-p^0(n,n-1)=\frac{1}{1+3\log n},\quad p^0(n+1,n)=1.
\end{align*}
In particular, the trace of $X^0$ on $\{0,\ldots,n\}$, denoted by $X$, is a simple random walk on $\{0,\ldots,n\}$ killed upon hitting $0$.
\begin{rem}[Basics on upcrossings]
	Denote by $u_\ell$ the number of upcrossings of $X^0$ from $\ell-1$ to $\ell$. Then $u_\ell$ is a Markov process whose transition laws are negative binomial distributions:
\begin{align*}
	\big(u_{\ell+1}\,\vert\,u_\ell=b\big)&\overset{d}{=}\sum_{j=1}^b \gamma^\ell_j,
\end{align*}
where for $\ell=1,\ldots,n-1$ (resp. $\ell=n$), $\{\gamma^\ell_j:j\ge 1\}$ are i.i.d. $\{0,1,\ldots\}$-valued geometric random variables with success probability $1/2$ (resp. $\frac{3\log n}{1+3\log n}$).
\end{rem}

\begin{lem}\label{relate}
	Let $p(a,b):=\p^1(u_{\ell+1}=b\,\vert\,u_\ell=a)$ for $\ell=1,\ldots,n-1$. Then uniformly in $x\in U_n$,
	\[\e[Y(n,x)]\asymp\sum_{\bar{m}\in\mathcal{M}}\prod_{\ell=2}^{n-1}p(m_\ell,m_{\ell+1}).\]
\end{lem}
\begin{proof}
	For any $\bar{m}$, we let $\mathcal{H}_n(\bar{m})$ (resp. $\mathcal{H}^0_n(\bar{m})$) be the collection of sample paths of $X$ (resp. $X^0$) such that the number of upcrossings from $\ell-1$ to $\ell$ equal to $m_{\ell}$ for $\ell\in[2,n]$ (resp. and the number of crossings from $n$ to $n+1$ is $\ge \frac{2n^2-n^{1+\delta}}{3\log n}$). Then we have
\begin{align*}
\begin{aligned}
	\sum_{\bar{m}\in \mathcal{M}}\p^1[X^0\in\mathcal{H}^0_{n}(\bar{m})]&=\sum_{\bar{m}\in \mathcal{M}}\p^1(u_2=1)\prod_{\ell=2}^n\p^1(u_{\ell+1}=m_{\ell+1}\,\vert\,u_\ell=m_\ell)\\
	&\overset{(*)}{\asymp} \sum_{\bar{m}\in\mathcal{M}}\prod_{\ell=2}^{n-1}p(m_\ell,m_{\ell+1})\Big(\asymp \sum_{\bar{m}\in \mathcal{M}}\p^1[X\in\mathcal{H}_n(\bar{m})]\Big),	
\end{aligned}
\end{align*}
where $(*)$ follows from $$\p^1(u_2=1)=\p^1(X^0_1=2,\,X^0_{T^{0,+}_1+1}=0)=1/4$$
(here $T^{0,+}_1:=\inf\{j\ge 1:X^0_j=1\}$) and the fact that there exists universal $c>0$ such that for any $n\ge 1$ and $\abs{b-2n^2}\le n^{1+\delta}$,
\[\p\left(\sum_{j=1}^{b} \gamma^n_j\in\left[\frac{2n^2-n^{1+\delta}}{3\log n},n^3\right]\right)>c.\]

Thus it suffices to show that uniformly in $x\in U_n$,
\begin{align*}
	\e[Y(n,x)]\asymp \sum_{\bar{m}\in \mathcal{M}}\p^1[X^0\in\mathcal{H}^0_{n}(\bar{m})].
\end{align*}
To this end, we let
$t_0:=\mathsf{H}_{\partial D(x,r_{n,1})}$ and define $t_1,\ldots,t_q$ ($q$ is random) to be the successive hitting times after $t_0$ of different elements of
	$\{\partial D(x,r_{n,0}),\ldots,\partial D(x,r_{n,n+1})\}$
	until the first crossing from $\partial D(x,r_{n,1})$ to $\partial D(x,r_{n,0})$. Setting $\Phi(y)=k$ if $y\in \partial D(x,r_{n,k})$, let $h_j:=\Phi(t_j)$ for $j=1,\ldots,q$. Fixing $\bar{m}$, it follows from \eqref{hitting2} that uniformly for any sample path $\varphi\in \mathcal{H}^0_n(\bar{m})$ and $z\in \partial D(x,r_{n,1})$,
	\begin{align*}
		\p^z\left(h=\varphi\right)=\left(1+O(n^{-6})\right)^{2(\abs{\bar{m}}+n^3)}\p^1\left(X^0=\varphi\right),
	\end{align*}
	where $\abs{\bar{m}}=\sum_{j=2}^n m_j$ and the power comes from the fact that the length of the path $\varphi\in \mathcal{H}^0_n(\bar{m})$ is bounded by $2(\abs{\bar{m}}+n^3)$. Since uniformly in $\bar{m}\in \mathcal{M}$, $\left(1+O(n^{-6})\right)^{2(\abs{\bar{m}}+n^3)}=1+O(n^{-3})$, we have uniformly in $z\in \partial D(x,r_{n,1})$,
	\[\sum_{\bar{m}\in\mathcal{M}}\p^z(h\in \mathcal{H}^0_n(\bar{m}))=\left(1+O(n^{-3})\right)\sum_{\bar{m}\in\mathcal{M}}\p^1(X^0\in \mathcal{H}^0_n(\bar{m})).\]
	Hence by the strong Markov property, we have uniformly in $x\in U_n$,
	\begin{align*}
	\begin{aligned}
 \e[Y(n,x)]
		=\;&\e\left(1_{\{\mathsf{H}_{\partial D(x,r_{n,1})}<\mathsf{H}_{\partial D(0, K_n)}\}}\sum_{m\in \mathcal{M}}\e^{S_{t_0}}\left[1_{\{h\in \mathcal{H}^0_n(\bar{m})\}}\p^{S_{t_q}}\left(\mathsf{H}_{\partial D(0,K_n)}<\mathsf{H}_{\partial D(x,r_{n,1})}\right)\right]\right)\\
		\asymp\; &\sum_{\bar{m}\in \mathcal{M}}\p^1[X^0\in\mathcal{H}^0_{n}(\bar{m})],		
	\end{aligned}
	\end{align*}
	where we use the fact that there exists constant $c>0$ such that for all $x\in U_n$, $y\in \partial D(x,r_{n,0})\cup\{0\}$, and $n\ge 1$,
	\begin{align*}
	c<\p^y\left(\mathsf{H}_{\partial D(x,r_{n,1})}<\mathsf{H}_{\partial D(x, \frac{1}{2}K_n)}\right)
	\le \p^y\left(\mathsf{H}_{\partial D(x,r_{n,1})}<\mathsf{H}_{\partial D(0, K_n)}\right)\\
	\le \p^y\left(\mathsf{H}_{\partial D(x,r_{n,1})}<\mathsf{H}_{\partial D(x, 2K_n)}\right)<1-c.
	\end{align*}
This finishes the proof.
 \end{proof}

\begin{proof}[Proof of Lemma \ref{Y}]
It can be seen from Lemma \ref{relate} that $$\inf_{x\in U_n}E[Y(n,x)]\asymp \sup_{x\in U_n}E[Y(n,x)].$$
For the second inequality in Lemma \ref{Y}, it suffices to show the following proposition.\phantom{\qedhere}

\begin{pro}\label{pp}
	\[\sum_{\bar{m}\in \mathcal{M}}\prod_{\ell=2}^{n-1}p(m_\ell,m_{\ell+1})\ge \exp\left\{-2n-n^{[(1-2\delta)\vee3\delta]+o(1)}\right\}.\]
\end{pro}
\begin{proof}
	It is enough to fix $\varepsilon>0$ and show that for $n$ large,
\begin{align}\label{suffices}
	\sum_{\bar{m}\in\mathcal{M}}\prod_{\ell=2}^{n-1}p(m_\ell,m_{\ell+1})\ge \exp\left\{-2n-n^{[(1-2\delta)\vee3\delta]+\varepsilon}\right\}.
\end{align}

By the local limit theorem (Cf. Section 2.3.1 of \cite{lawler2010random}), there exists $\rho>0$ such that uniformly in all $b,b'\ge 0$ with $\abs{b'-b}<\rho b$,
\[p(b,b')=\frac{1}{2\sqrt{\pi b}}\exp\left\{-\frac{(b-b')^2}{4b}+O\left(\frac{1}{\sqrt{b}}+\frac{\abs{b-b'}^3}{b^2}\right)\right\}.\]
In particular, since for any $\bar{m}\in\mathcal{M}$, $\abs{m_\ell-2\ell^2}\le \ell^{1+\delta}$ for all $\ell\in[2,n]$, we have uniformly in $\bar{m}\in \mathcal{M}$,
\begin{align*}
	p(m_\ell,m_{\ell+1})&=\frac{1}{2\sqrt{\pi m_\ell}}\exp\left\{-\frac{(m_{\ell+1}-m_\ell)^2}{4m_\ell}+O\left(\frac{1}{\sqrt{m_\ell}}+\frac{(m_{\ell+1}-m_\ell)^3}{m_\ell^2}\right)\right\}\\
	&=\frac{1}{\sqrt{2\pi}2\ell}\exp\left\{-\frac{(m_{\ell+1}-m_\ell)^2}{8\ell^2}+O(\ell^{3\delta-1})\right\}.
\end{align*}
Therefore
\begin{align*}
	\prod_{\ell=2}^{n-1}p(m_\ell,m_{\ell+1})=\left(\prod_{\ell=2}^{n-1}\frac{1}{\sqrt{2\pi}2\ell}\right)\exp\left\{-\sum_{\ell=2}^{n-1}\frac{(m_{\ell+1}-m_\ell)^2}{8\ell^2}+O(n^{3\delta})\right\}.
\end{align*}

For any $\bar{m}\in\mathcal{M}$, we let $\Delta_\ell=\Delta_\ell(\bar{m}):=m_\ell-2\ell^2$. Then uniformly in $\bar{m}\in\mathcal{M}$,
\begin{align*}
	&\quad\,\frac{1}{8}\sum_{\ell=2}^{n-1}\frac{(m_{\ell+1}-m_\ell)^2}{\ell^2}=\frac{1}{8}\sum_{\ell=2}^{n-1}\frac{(4\ell+2+\Delta_{\ell+1}-\Delta_{\ell})^2}{\ell^2}\\
	&=O(n^\delta)+\frac{1}{8}\sum_{\ell=2}^{n-1}\frac{(4\ell+\Delta_{\ell+1}-\Delta_{\ell})^2}{\ell^2}\\
	&=O(n^\delta)+2n+\sum_{\ell=2}^{n-1}\frac{\Delta_{\ell+1}-\Delta_\ell}{\ell}+\frac{1}{8}\sum_{\ell=2}^{n-1}\frac{(\Delta_{\ell+1}-\Delta_\ell)^2}{\ell^2}.
\end{align*}
Observe that
\begin{align*}
	\sum_{\ell=2}^{n-1}\frac{\Delta_{\ell+1}-\Delta_\ell}{\ell}&=\sum_{\ell=2}^{n-1}\left(\frac{\Delta_{\ell+1}}{\ell+1}-\frac{\Delta_\ell}{\ell}\right)+\sum_{\ell=2}^{n-1}\frac{\Delta_{\ell+1}}{\ell(\ell+1)}\\
	&=\frac{\Delta_n}{n}-\frac{\Delta_2}{2}+\sum_{\ell=2}^{n-1}\frac{\Delta_{\ell+1}}{\ell(\ell+1)}=O(n^\delta).
\end{align*}
Hence we get uniformly in $\bar{m}\in \mathcal{M}$,
\begin{align*}
	\prod_{\ell=2}^{n-1}p(m_\ell,m_{\ell+1})=\left(\prod_{\ell=2}^{n-1}\frac{1}{\sqrt{2\pi}2\ell}\right)\exp\left\{-2n-\sum_{\ell=2}^{n-1}\frac{(\Delta_{\ell+1}-\Delta_\ell)^2}{8\ell^2}+O(n^{3\delta})\right\},
\end{align*}
and consequently,
\begin{align}\label{aa}
\begin{aligned}
	\sum_{\bar{m}\in\mathcal{M}}\prod_{\ell=2}^{n-1}p(m_\ell,m_{\ell+1})&=\exp\left\{-2n+O(n^{3\delta})\right\}\sum_{\substack{(\Delta_2,\ldots,\Delta_n):\\\abs{\Delta_\ell}\le\ell^{1+\delta}\,\forall\,\ell\in[2,n]}}\prod_{\ell=2}^{n-1}b_\ell(\Delta_\ell,\Delta_{\ell+1}),	
\end{aligned}
\end{align}
where $b_\ell(k_1,k_2):=\frac{1}{\sqrt{2\pi}2\ell}\exp\left\{-\frac{(k_1-k_2)^2}{8\ell^2}\right\}$. For the following proof, we shall introduce a lemma whose proof will be postponed till the end of this subsection.
\begin{lem}\label{maxabs}
	For any $\rho\in (0,1)$,
	\begin{align*}
		\sum_{\substack{(\Delta_{n^\rho+1},\ldots,\Delta_n):\\\abs{\Delta_\ell}\le \ell^{1+\delta}\,\forall\,\ell\in(n^\rho,n]}}\prod_{\ell=n^\rho+1}^{n-1}b_{\ell}(\Delta_\ell,\Delta_{\ell+1})\ge \exp\left\{-O(n^{[3-2\rho(1+\delta)]\vee 2\delta})\right\}.
	\end{align*}
\end{lem}
Take $\rho\in(0,1)$ such that $3-2\rho(1+\delta)<1-2\delta+\varepsilon$. Let $k=\sup\{m\in\N:\rho^m\ge 3\delta\}$. Note that uniformly in $\ell\in[2,n-1]$, $b_{\ell}(\Delta_\ell,\Delta_{\ell+1})=\exp\left\{-O(n^{2\delta})\right\}$.
Using this in the last term in \eqref{aa} for $\ell=n^{3\delta/\rho^m}=:q_m$, $m=0,1,\ldots,k$,
we can decompose the sum into
\begin{align}\label{decompose}
\begin{aligned}
	&\quad\,\sum_{\substack{(\Delta_2,\ldots,\Delta_n):\\\abs{\Delta_\ell}\le\ell^{1+\delta}\,\forall\,\ell\in[2,n]}}\prod_{\ell=2}^{n-1}b_{\ell}(\Delta_\ell,\Delta_{\ell+1})\\
	&=\exp\left\{-O(n^{2\delta})\right\}\sum_{\substack{(\Delta_2,\ldots,\Delta_n):\\\abs{\Delta_\ell}\le\ell^{1+\delta}\,\forall\,\ell\in[2,n]}}\prod_{\ell\in[2,n-1]\setminus\{q_0,\ldots,q_k\}}b_{\ell}(\Delta_\ell,\Delta_{\ell+1})\\
	&=\exp\left\{-O(n^{2\delta})\right\}\left[\sum_{\substack{(\Delta_2,\ldots,\Delta_{q_0}):\\\abs{\Delta_\ell}\le\ell^{1+\delta}\,\forall\,\ell\in[2,q_0]}}\prod_{\ell=2}^{q_0-1}b_{\ell}(\Delta_\ell,\Delta_{\ell+1})\right]\\
	&\cdot\prod_{j=0}^{k} \left[\sum_{\substack{(\Delta_{q_j+1},\ldots,\Delta_{q_{j+1}}):\\\abs{\Delta_\ell}\le\ell^{1+\delta}\,\forall\,\ell\in[q_j+1,q_{j+1}]}}\prod_{\ell=q_j+1}^{q_{j+1}}b_{\ell}(\Delta_\ell,\Delta_{\ell+1})\right],	
\end{aligned}
\end{align}
where $q_{k+1}:=n$. It follows from Lemma \ref{maxabs} that
\begin{align*}
	&\quad\,\prod_{j=0}^{k} \left[\sum_{\substack{(\Delta_{q_j+1},\ldots,\Delta_{q_{j+1}}):\\\abs{\Delta_\ell}\le\ell^{1+\delta}\,\forall\,\ell\in[q_j+1,q_{j+1}]}}\prod_{\ell=q_j+1}^{q_{j+1}}b_{\ell}(\Delta_\ell,\Delta_{\ell+1})\right]\\
	&\ge \exp\left\{-O(n^{[3-2\rho(1+\delta)]\vee 2\delta})\right\}\ge \exp\left\{-O(n^{(1-2\delta)\vee 2\delta})\right\}
\end{align*}
for $n$ large.
For the other term, we have
\begin{align*}
	&\quad\,\sum_{\substack{(\Delta_2,\ldots,\Delta_{q_0}):\\\abs{\Delta_\ell}\le\ell^{1+\delta}\,\forall\,\ell\in[2,q_0]}}\prod_{\ell=2}^{q_0-1}b_{\ell}(\Delta_\ell,\Delta_{\ell+1})\ge \sum_{\substack{(\Delta_2,\ldots,\Delta_{q_0}):\\\abs{\Delta_\ell}\le\ell\,\forall\,\ell\in[2,q_0]}}\prod_{\ell=2}^{q_0-1}b_{\ell}(\Delta_\ell,\Delta_{\ell+1})\\
	&=\sum_{\substack{(\Delta_2,\ldots,\Delta_{q_0}):\\\abs{\Delta_\ell}\le\ell\,\forall\,\ell\in[2,q_0]}}\prod_{\ell=2}^{q_0-1}\left(\frac{1}{\sqrt{2\pi}2\ell}\exp\left\{-O(1)\right\}\right)=\exp\left\{-O(n^{3\delta})\right\}.
\end{align*}
Therefore for $n$ large,
\begin{align*}
	\sum_{\substack{(\Delta_2,\ldots,\Delta_n):\\\abs{\Delta_\ell}\le\ell^{1+\delta}\,\forall\,\ell\in[2,n]}}\prod_{\ell=2}^{n-1}b_{\ell}(\Delta_\ell,\Delta_{\ell+1})\ge \exp\left\{-n^{[(1-2\delta)\vee 3\delta]+\varepsilon}\right\}.
\end{align*}
Substituting back into \eqref{aa}, we get \eqref{suffices}. That concludes the proof of Proposition \ref{pp} and Lemma \ref{Y}.
\end{proof}
\end{proof}
\begin{proof}[Proof of Lemma \ref{maxabs}]
	We define $f$: $\r^{n-n^\rho}\rightarrow\r$ by
	\begin{align*}
		f(\Delta_{n^\rho+1},\ldots,\Delta_n)=\left(\frac{1}{\sqrt{2\pi}2n^\rho}\exp\left\{-\frac{\Delta^2_{n^\rho+1}}{8n^{2\rho}}\right\}\right)\prod_{\ell=n^\rho+1}^{n-1}b_{\ell}(\Delta_\ell,\Delta_{\ell+1}).
	\end{align*}
Recall that $[x]$ is the largest integer less than or equal to $x$. Then it is easy to see that uniformly in $(\Delta_{n^\rho+1},\ldots,\Delta_n)$ with $\abs{\Delta_\ell}\le \ell^{1+\delta}$ for $\ell\in(n^\rho,n]$,
\[f([\Delta_{n^\rho+1}],\ldots,[\Delta_n])=\exp\left\{O(n^\delta)\right\}f(\Delta_{n^\rho+1},\ldots,\Delta_n).\]
It follows that
\begin{equation}
    \begin{split}
&\quad\,\sum_{\substack{(\Delta_{n^\rho+1},\ldots,\Delta_n):\\\abs{\Delta_\ell}\le \ell^{1+\delta}\,\forall\,\ell\in(n^\rho,n]}}\prod_{\ell=n^\rho+1}^{n-1}b_{\ell}(\Delta_\ell,\Delta_{\ell+1})\\
		&=\exp\left\{-O(n^{2\delta})\right\}\sum_{\substack{(\Delta_{n^\rho+1},\ldots,\Delta_n)\in\z^{n-n^\rho}:\\\abs{\Delta_\ell}\le \ell^{1+\delta}\,\forall\,\ell\in(n^\rho,n]}}f(\Delta_{n^\rho+1},\ldots,\Delta_n)\\
		&=\exp\left\{-O(n^{2\delta})\right\}\int_{\substack{\abs{[\Delta_\ell]}\le \ell^{1+\delta}\\\forall\,\ell\in(n^\rho,n]}}f([\Delta_{n^\rho+1}],\ldots,[\Delta_n])\prod_{\ell=n^\rho+1}^n \d \Delta_{\ell}\\
		&\ge\exp\left\{-O(n^{2\delta})\right\}\int_{\substack{\abs{\Delta_\ell}\le \ell^{1+\delta}-1\\\forall\,\ell\in(n^\rho,n]}}f(\Delta_{n^\rho+1},\ldots,\Delta_n)\prod_{\ell=n^\rho+1}^n \d \Delta_{\ell}\\
		&=\exp\left\{-O(n^{2\delta})\right\}\p\left(\abs{B_{4\sum_{j=n^\rho}^{\ell-1} j^2}}\le \ell^{1+\delta}-1\,\forall\,\ell\in(n^\rho,n]\cap\N\right)\\
		&\ge \exp\left\{-O(n^{2\delta})\right\}\p\left(\abs{B_t}\le n^{\rho(1+\delta)}-1\,\forall\,t\in(0,4n^3]\right),
    \end{split}
\end{equation}
	where $B$ is a 1D Brownian motion starting from $0$ and we use the fact that $f(\Delta_{n^\rho+1},\ldots,\Delta_n)$ is the density function of $\left(B_{4\sum_{j=n^\rho}^{\ell-1} j^2}:\ell\in(n^\rho,n]\cap\N\right)$. The law of the maximum absolute value of Brownian motion is (Cf. page 342 in \cite{feller1991introduction})
	\begin{align*}
		\p\left(\sup_{0\le t\le u}\abs{B_t}<r\right)=\frac{4}{\pi}\sum_{k\ge 0}(-1)^k\frac{1}{2k+1}\exp\left\{-\frac{(2k+1)^2\pi^2}{8r^2}u\right\}.
	\end{align*}
	Thus we have
	\[\p\left(\abs{B_t}\le n^{\rho(1+\delta)}-1\,\forall\,t\in(0,4n^3]\right)=\exp\left\{-O(n^{3-2\rho(1+\delta)})\right\},\]
	which leads to the conclusion.
\end{proof}
\subsubsection{Proof of Proposition \ref{moment} (2)}
Recall that $l(x,y)=\min\{m\ge 1:D(x,r_{n,m})\cap D(y,r_{n,m})=\emptyset\}$. In this part, we fix $x,y\in U_n$ such that $l(x,y)\le n-3\log n$ and simply write $l=l(x,y)$. Let $\widetilde{l}:=l+3\log n$.

 The idea of the proof is as follows. Lemma \ref{preliminary} tells us that what happens inside $D(y,r_{n,\widetilde{l}})$ is almost independent of what happens outside $D(y,r_{n,l})$. We will see that most of $N^x_{n,k}$ are determined by what happens outside $D(y,r_{n,l})$ and it is trivial that $N^y_{n,k}$, $k\ge \widetilde{l}+1$, depend on what happens inside $D(y,r_{n,\widetilde{l}})$. Hence by removing the restrictions on other $N^x_{n,k}$ and $N^y_{n,k}$ from the event $\{Y(n,x)=Y(n,y)=1\}$, the obtained event can be decomposed into two almost independent parts. The probability of each part can be estimated using the connection between $N^x_{n,k}$ (or $N^y_{n,k}$) and the upcrossings of some MJP.

 Let $\widetilde{N}^y_{n,\widetilde{l}}$ be the number of excursions from $\partial D(y,r_{n,\ell})$ to $\partial D(y,r_{n,\widetilde{\ell}})$ before $\mathsf{H}_{\partial D(0, K_n)}$.
For any subset $I\subset \{1,\ldots,n\}$ and $x\in U_n$, we set
\begin{align*}
\dot{\Gamma}^x_n(I)&=\dot{\Gamma}^{x}_n(I,\delta):=\Big\{\abs{N^x_{n,k}-2k^2}\le k^{1+\delta}\,\forall\,k\in I\Big\},\\
	\Gamma^x_n(I)&=\Gamma^{x}_n(I,\delta):=\dot{\Gamma}^x_n(I)\cap\left\{N^x_{n,n+1}\in\left[ \frac{2n^2-n^{1+\delta}}{3\log n},n^3\right]\right\}.
\end{align*}
The cases we will pay special attention to are $I=I_l\text{ or } J_l$ with
 $I_l:=\{1,\ldots,l-3\}\cup\{\widetilde{l}+1,\ldots,n\}$ and $J_l:=\{\widetilde{l}+1,\ldots,n\}$.

 In the sequel, we shall use the link between $\widetilde{N}^y_{n,\widetilde{l}}$, $N^y_{n,k}$ ($\widetilde{l}+1\le k\le n+1$) and  upcrossings of  MJP. So it is necessary to give an upper bound for $\widetilde{N}^y_{n,\widetilde{l}}$ to control the error term. Observe that
 \begin{align}\label{5/2}
 	\e\left(Y'(n,x)Y'(n,y)\,;\,\widetilde{N}^y_{n,\widetilde{l}}>n^{5/2}\right)\le \p\left(\abs{N^y_{n,l}-2l^2}\le l^{1+\delta},\,\widetilde{N}^y_{n,\widetilde{l}}>n^{5/2}\right).
 \end{align}
With the same idea as in Lemma \ref{relate}, $N^y_{n,l}$ and $\widetilde{N}^y_{n,\widetilde{l}}$ can be related to the number of upcrossings of MJP. Letting $\{\gamma_j:j\ge 1\}$ (resp. $\{X_j:j\ge 1\}$) be a sequences of i.i.d. $\{0,1,\ldots\}$-valued geometric random variables with success probability $\frac{3\log n}{1+3\log n}$ (resp. Bernoulli random variables with success probability $\frac{3\log n}{1+3\log n}$), we have
\begin{align*}
	\eqref{5/2}&\lesssim \p\left(\sum_{j=1}^{2l^2+l^{1+\delta}}\gamma_j>n^{5/2}\right)\le \p\left(\sum_{j=1}^{2n^2+n^{1+\delta}}\gamma_j\ge n^{5/2}\right)\\
	&=\p\left(\sum_{j=1}^{n^{5/2}+2n^2+n^{1+\delta}}X_j<2n^2+n^{1+\delta}\right)<e^{-cn^{5/2}}
\end{align*}
for some constant $c>0$. Hence thanks to Proposition \ref{moment} (1), it suffices to show that
\begin{equation}\label{n^5/2}
\e\left(Y'(n,x)Y'(n,y)\,;\,\widetilde{N}^y_{n,\widetilde{l}}\le n^{5/2}\right)\le \exp\left\{2l+n^{[3\delta\vee(1-2\delta)]+o(1)}\right\}\e[Y'(n,x)]\e[Y'(n,y)]. 
\end{equation}

Recall that $\mathcal{G}^y_{(r_{n,l},r_{n,\widetilde{l}})}$ is the $\sigma$-algebra generated by the excursions from $\partial D(y,r_{n,\ell})$ to $\partial D(y,r_{n,\widetilde{\ell}})$.
It is easily seen that $D(y,r_{n,l-3})\supset D(x,r_{n,l})$, which together with $D(y,r_{n,l})\cap D(x,r_{n,l})=\emptyset$ implies that $\Gamma^x_n(I_l)\in \mathcal{G}^y_{(r_{n,l},r_{n,\widetilde{l}})}$ (because we only lose the information on the excursions from $\partial D(y,r_{n,\widetilde{l}})$ to $\partial D(y,r_{n,l})$ which does not affect $N^x_{n,k}$ for $k\in\{1,\ldots,l-3\}\cup\{l+1,\ldots,n\}\supset I_l$). Therefore
\begin{align}\label{decouple}
\begin{aligned}
	&\quad\,\e\left(Y'(n,x)Y'(n,y);\,\widetilde{N}^y_{n,\widetilde{l}}\le n^{5/2}\right)\le \p\left(\Gamma^x_n(I_l)\cap \Gamma^y_n(J_l)\,;\,\widetilde{N}^y_{n,\widetilde{l}}\le n^{5/2}\right)\\
	&=\e\left[\p\left(\Gamma^y_n(J_l)\,\big\vert\,\mathcal{G}^y_{(r_{n,l},r_{n,\widetilde{l}})}\right);\,\Gamma^x_n(I_l),\,\widetilde{N}^y_{n,\widetilde{l}}\le n^{5/2}\right]\\
	&=\sum_{m=0}^{n^{5/2}}\e\left[\p\left(\Gamma^y_n(J_l)\,\big\vert\,\mathcal{G}^y_{(r_{n,l},r_{n,\widetilde{l}})}\right);\,\Gamma^x_n(I_l),\,\widetilde{N}^y_{n,\widetilde{l}}=m\right].
\end{aligned}
\end{align}
Let us denote by $\widetilde{N}^y_{n,k}(m,\widetilde{l})$ (resp. $N^y_{n,k}(m,\widetilde{l}+1)$) the number of excursions from $\partial D(y,r_{n,k-1})$ to $\partial D(y,r_{n,k})$ until completion of the first $m$ excursions from $\partial D(y,r_{n,l})$ to $\partial D(y,r_{n,\widetilde{l}})$ (resp. from $\partial D(y,r_{n,\widetilde{l}})$ to $\partial D(y,r_{n,\widetilde{l}+1})$) for $k\ge \widetilde{l}$. Set
$$
\widetilde{\Gamma}^y_n(J_l,m,\widetilde{l})=\Big\{\abs{\widetilde{N}^y_{n,k}(m,\widetilde{l})-2k^2}\le k^{1+\delta}\,\forall\,k\in J_l,\,\widetilde{N}^y_{n,n+1}(m,\widetilde{l})\ge \frac{2n^2-n^{1+\delta}}{3\log n}\Big\},
$$
 and
$$\Gamma^y_n(J_l,m,\widetilde{l}+1)=\Big\{\abs{N^y_{n,k}(m,\widetilde{l}+1)-2k^2}\le k^{1+\delta}\,\forall\,k\in J_l,\,N^y_{n,n+1}(m,\widetilde{l}+1)\ge \frac{2n^2-n^{1+\delta}}{3\log n}\Big\},
$$
i.e., replacing all $N^y_{n,k}$ in the definition of $\Gamma^y_n(J_l)$ by $\widetilde{N}^y_{n,k}(m,\widetilde{l})$ and $N^y_{n,k}(m,\widetilde{l}+1)$ respectively. Observe that $\Gamma^y_n(J_l)=\widetilde{\Gamma}^y_n(J_l,m,\widetilde{l})$ on $\big\{\widetilde{N}^y_{n,\widetilde{l}}=m\big\}$. It is easy to see that $$\widetilde{\Gamma}^y_n(J_l,m,\widetilde{l})\in \mathcal{H}^y_{(r_{n,l},r_{n,\widetilde{l}},r_{n,\widetilde{l}+1})}(m).$$ Moreover,
by considering the link between $\widetilde{N}^y_{n,k}(m,\widetilde{l})$ and the number of upcrossings of some MJP, we can deduce that $\widetilde{\Gamma}^y_n(J_l,m,\widetilde{l})$ satisfies Condition \eqref{On} with some parameter $\lambda>0$ uniform for all $0\le m\le n^{5/2}$.
Hence it follows from Lemma \ref{preliminary} that uniformly in $m,n\in \N$ with $0\le m\le n^{5/2}$ and $y\in U_n$,
\begin{equation}\label{decoupling}
\begin{split}
\p\left(\Gamma^y_n(J_l)\,\big\vert\,\mathcal{G}^y_{(r_{n,l},r_{n,\widetilde{l}})}\right)1_{\big\{\widetilde{N}^y_{n,\widetilde{l}}=m\big\}}
=\;&\p\left(\widetilde{\Gamma}^y_n(J_l,m,\widetilde{l})\,\big\vert\,\mathcal{G}^y_{(r_{n,l},r_{n,\widetilde{l}})}\right)1_{\big\{\widetilde{N}^y_{n,\widetilde{l}}=m\big\}}\\
	=\;& \left(1+O(n^{-1/2}\log n)\right)\p\left(\widetilde{\Gamma}^y_n(J_l,m,\widetilde{l})\right)1_{\big\{\widetilde{N}^y_{n,\widetilde{l}}=m\big\}}.	
 \end{split}
\end{equation}
\begin{rem}\label{remark}
The following results are also obtained using the link between $\widetilde{N}^y_{n,k}(m,\widetilde{l})$ and the number of upcrossings of MJP. Uniformly in $y\in U_n$,

\noindent (1):
$$
\begin{aligned}[t]
			\sup\limits_{m\ge 0}\p\left(\widetilde{\Gamma}^y_n(J_l,m,\widetilde{l})\right)&\le \left(1+O(n^{-3})\right)\sup\limits_{m_{\widetilde{l}+1}:\abs{m_{\widetilde{l}+1}-2(\widetilde{l}+1)^2}\le (\widetilde{l}+1)^{1+\delta}}\p\left({\Gamma}^y_n(J_l,m,\widetilde{l}+1)\right)\\
		&\le \exp\left\{O(\widetilde{l}^{2\delta})\right\}\inf\limits_{m_{\widetilde{l}+1}:\abs{m_{\widetilde{l}+1}-2(\widetilde{l}+1)^2}\le (\widetilde{l}+1)^{1+\delta}}\p\left({\Gamma}^y_n(J_l,m,\widetilde{l}+1)\right);
\end{aligned}
$$
(2):
$$
\begin{aligned}[t]
			\e[Y'(n,y)]&\ge \left(1+O(n^{-3})\right)\p\left[\dot{\Gamma}^y_n\big(\{1,\ldots,\widetilde{l}+1\}\big)\right]\\		&\cdot\inf\limits_{m_{\widetilde{l}+1}:\abs{m_{\widetilde{l}+1}-2(\widetilde{l}+1)^2}\le (\widetilde{l}+1)^{1+\delta}}\p\left({\Gamma}^y_n(J_l,m,\widetilde{l}+1)\right);
\end{aligned}
$$
(3):
$$
\begin{aligned}[t]
		\p\left[\dot{\Gamma}^y_n\big(\{1,\ldots,\widetilde{l}+1\}\big)\right]\ge\exp\left\{-2\widetilde{l}-\widetilde{l}^{[3\delta\vee(1-2\delta)]+o_{\widetilde{l}}(1)}\right\}\ge \exp\left\{-2l-n^{[3\delta\vee(1-2\delta)]+o(1)}\right\};
		\end{aligned}
$$
(4):
$$
\begin{aligned}[t]			\p\left[\Gamma^y_n(I_l)\right]\le \exp\left\{O(n^{2\delta}\log n)\right\}\e[Y'(n,y)].
\end{aligned}
$$
	The first two inequalities come from the Markov property of the upcrossings. The third one can be seen from the analysis in \S\ref{first}. The last one is obtained using the same arguments as that in \eqref{decompose}.
\end{rem}
	The first three inequalities in the remark imply that
	\begin{align*}
		\sup\limits_{m\ge 0}\p\left(\widetilde{\Gamma}^y_n(J_l,m,\widetilde{l})\right)\le \exp\left\{2l+n^{[3\delta\vee(1-2\delta)]+o(1)}\right\}\e[Y'(n,y)].
	\end{align*}
	
Substituting back into \eqref{decoupling} and \eqref{decouple}, we get uniformly in $y\in U_n$,
\begin{align*}
	\sum_{m=0}^{n^{5/2}}\p\left(\Gamma^y_n(J_l)\,\big\vert\,\mathcal{G}^y_{(r_{n,l},r_{n,\widetilde{l}})}\right)1_{\big\{\widetilde{N}^y_{n,\widetilde{l}}=m\big\}}&\le\left(1+O(n^{-1/2}\log n)\right)n^{5/2}\sup_{m\ge 0}\p\left(\widetilde{\Gamma}^y_n(J_l,m,\widetilde{l})\right)\\
	&\le \exp\left\{2l+n^{[3\delta\vee(1-2\delta)]+o(1)}\right\}\e[Y'(n,y)],
\end{align*}
and consequently, uniformly in $x,y\in U_n$,
\begin{align*}
\e\left(Y'(n,x)Y'(n,y)\,;\,\widetilde{N}^y_{n,\widetilde{l}}\le n^{5/2}\right)
\le\;& \exp\left\{2l+n^{[3\delta\vee(1-2\delta)]+o(1)}\right\}\e[Y'(n,y)]\p\left[\Gamma^x_n(I_l)\right]\\
	\overset{(*)}{\le}\;& \exp\left\{2l+n^{[3\delta\vee(1-2\delta)]+o(1)}\right\}\e[Y'(n,x)]\e[Y'(n,y)],
\end{align*}
where $(*)$ comes from the last inequality in Remark \ref{remark}. That concludes the proof of \eqref{n^5/2} and Proposition \ref{moment} (2).



\end{document}